\documentclass[10pt]{amsart}
\usepackage{mystyle}   
\usepackage{placeins}

\usepackage[utf8]{inputenc}
\usepackage{amsfonts}
\usepackage{graphics}
\usepackage[all, cmtip]{xy}
\usepackage{amsthm,amsfonts,amssymb,amsmath,amsxtra}
\usepackage[all]{xy}
\usepackage{flexisym}
\usepackage{mathrsfs}
\usepackage{enumitem}
\usepackage[dvipsnames]{xcolor}
\usepackage[colorlinks]{hyperref}
\hypersetup{
  citecolor   = blue 
 }

\newenvironment{altenumerate}
   {\begin{list}
      {\textup{(\theenumi)} }
      {\usecounter{enumi}
       \setlength{\labelwidth}{0pt}
       \setlength{\labelsep}{0pt}
       \setlength{\leftmargin}{0pt}
       \setlength{\itemsep}{\the\smallskipamount}
       \renewcommand{\theenumi}{\arabic{enumi}}
      }}
   {\end{list}}
\newenvironment{altitemize}
   {\begin{list}
      {$\bullet$}
      {\setlength{\labelwidth}{0pt}
       \setlength{\itemindent}{5pt}
       \setlength{\labelsep}{5pt}
       \setlength{\leftmargin}{0pt}
       \setlength{\itemsep}{\the\smallskipamount}
      }}
   {\end{list}}

\usepackage{caption}
\usepackage{blkarray}
\newcommand{\xdownarrow}[1]{%
  {\left\downarrow\vbox to #1{}\right.\kern-\nulldelimiterspace}
}
 \newcommand{\quash}[1]{}  
\newcommand{\M}{{\rm M}}
\newcommand{\Mloc}{{\rm M}^{\rm loc}}
\newcommand{\Mspl}{{\rm M}^{\rm spl}}

\newcommand{\Mnaive}{{\rm M}^{\rm naive}}

\newcommand{\ad}{{\rm ad}}
\newcommand{\sgn}{{\rm sgn}}
    \newcommand{\Bl}{{\rm Bl}}

\newcommand{\vphi}{\varphi}

\newcommand{\diag}{{\rm diag}}

\newcommand{\Hom}{{\rm Hom}}

\newcommand{\Res}{{\rm Res}}
\newcommand{\Gal}{{\rm Gal}}

\newcommand{\etale}{\' etale }

\newcommand{\Proj}{{\rm Proj}}

\newcommand{\Tr}{{\rm Tr}}

\newcommand{\Gr}{{\rm Gr}}

\newcommand{\GL}{{\rm GL}}

\newcommand{\GO}{{\rm GO}}
\newcommand{\OGr}{{\rm OGr}}
\newcommand{\GSp}{{\rm GSp}}

\newcommand{\ZZ}{\mathbb{Z}}

\newcommand{\calA}{\mathcal{A}}
\newcommand{\calX}{\mathcal{X}}

\newcommand{\calF}{\mathcal{F}}
\newcommand{\calL}{\mathcal{L}}

\newcommand{\calR}{\mathcal{R}}

\newcommand{\calI}{\mathcal{I}}

\newcommand{\calU}{\mathcal{U}}
\newcommand{\calH}{\mathcal{H}}

\newcommand{\al}{\alpha}

\newcommand{\ga}{\gamma}

\newcommand{\bb }{\langle}
\newcommand{\pp}{\rangle}

\usepackage{multicol}
\numberwithin{equation}{subsection}

\newtheorem{Main Theorem}[Theorem]{Main Theorem}

\usepackage[colorlinks]{hyperref}
\makeatletter
\renewcommand*\env@matrix[1][*\c@MaxMatrixCols c]{%
  \hskip -\arraycolsep
  \let\@ifnextchar\new@ifnextchar
  \array{#1}}
\makeatother

\newif\ifgrading

\gradingfalse

\usepackage[colorlinks]{hyperref}
\hypersetup{linkcolor=black}
\usepackage{xcolor}
\usepackage{ wasysym }

\usepackage{tikz-cd}

\newcommand{\calG}{{\mathcal G}}
\newcommand{\calO}{{\mathcal O}}

\newcommand{\Spec}{{\rm Spec \, } }

\usepackage{xcolor}

\newcommand{\Addresses}{{
		\bigskip
		\footnotesize

        \textsc{Yau Mathematical Sciences Center, Tsinghua University,
			Beijing, 100084, China}\par\nopagebreak
		\textit{E-mail address:} 
        \texttt{jie-yang@mail.tsinghua.edu.cn}\\
        
		\textsc{Department of Mathematics, Universität Münster, Münster, 48149, Germany}\par\nopagebreak
		\textit{E-mail address:} \texttt{io.zachos@uni-muenster.de}\\
		
		\textsc{Department of Mathematics, University of Science and Technology Beijing,
			Beijing, 100083, China}\par\nopagebreak
		\textit{E-mail address:} \texttt{zhihaozhao@ustb.edu.cn}
}}	

\begin{document}

\title{On semi-stable integral models for Shimura varieties} 
	\date{}
	\author{Jie Yang, Ioannis Zachos and Zhihao Zhao}

\begin{abstract}
We construct (potentially) semi-stable integral models for a class of Shimura varieties with maximal parahoric level at an odd prime $p$. The
main input is an explicit construction of semi-stable equivariant modifications of the relevant canonical local models, where the underlying groups are Weil restrictions of unramified unitary similitude groups, symplectic similitude groups, or even orthogonal similitude groups. 
Via the local model diagram, these constructions give (potentially) semi-stable integral models of the corresponding Shimura varieties. In particular, the resulting models are regular and have reduced special fiber with normal crossings. {As an application, we deduce the unipotence of the inertia action on nearby cycles and the $\ell$-adic cohomology of the geometric generic fibers.}
\end{abstract}
	\maketitle
	\tableofcontents

\section{Introduction}\label{Intro}
\subsection{Motivation and main results}
Let $(\bG,X)$ be a Shimura datum. The associated Shimura varieties $$\Sh_\rK(\bG,X)$$ with level $\rK$ are quasi-projective smooth varieties over the reflex field $\bE$. They are central objects in arithmetic geometry and number theory. A fundamental problem is to construct integral models of $\Sh_\rK(\bG,X)$ with ``nice" geometric properties. In particular, one would like to understand when such integral models are regular, semi-stable, or potentially semi-stable. This is a hard problem and fits into a broader program about the resolution of singularities in mixed characteristics.  

 Let $p$ be a prime and $v|p$ be a place of $\bE$. Write $E$ for the completion of $\bE$ at $v$. Assume $\rK=\rK_p\rK^p$ and $\rK_p\sset \bG(\BQ_p)$ is parahoric. Pappas and Rapoport \cite{pappas2024p}, using the theory of $p$-adic shtukas, formulated a conjectural theory of \dfn{canonical integral models} over $\CO_E$ of the pro-system $\cbra{\Sh_{\rK_p\rK^p}(\bG,X)}_{\rK^p}$ which are unique and functorial. This conjecture has been verified, building on \cite{kisin2018integral,kisin2026integral}, when $(\bG,X)$ is of Hodge type, and when $(\bG,X)$ is of abelian type with $p>2$; see \cite{pappas2024p,daniels2025canonical,Wu,MW26}.
When $\rK_p$ is hyperspecial, the canonical integral model $\sS_\rK$ is smooth over the base $p$-adic ring; when $\rK_p$ is a general parahoric, however, singularities of $\sS_\rK$ often appear, and how to resolve these singularities becomes a natural and important problem.

A guiding principle is that the singularities of canonical integral models of Shimura varieties should be controlled by \dfn{canonical local models}. In the PEL setting, this philosophy goes back to Rapoport--Zink \cite{RZbook}. It was later extended in greater generality to abelian type Shimura varieties by \cite{kisin2018integral,kisin2026integral}, at least when $p>2$ (see also some progress in \cite{yang20232} for $p=2$). Let $$\RM^\loc_{\CG,\mu}$$ denote the canonical local model attached to a local model triple $(G,\CG,\mu)$, see Definition \ref{defn-lmt},  with reflex field $E/\BQ_p$. {Assume that $p$ is odd, $G$ is tame and the adjoint group $G^\ad$ is absolutely simple,} He--Pappas--Rapoport \cite{HPR} determined the complete list of $\RM^\loc_{\CG,\mu}$ with good resp. semi-stable reduction. In particular, their work classifies the smooth, resp. semi-stable canonical integral models of $\Sh_\rK(\bG,X)$ of abelian type when $p>2$, $\bG_{\bQ_p}$ is tamely ramified and $\bG_{\BQ_p}^\ad$ is absolutely simple. More generally, one expects (\cite[\S 8.4]{PRS}, \cite[\S 2.12]{P}) the following.
\begin{conj}[{see Conjecture \ref{conj-modif}}]  \label{intro-conj}
   There exists a $\CG_{\CO_E}$-\dfn{equivariant modification} $$N\lra \RM^\loc_{\CG,\mu},$$  which is an isomorphism on the generic fibers, such that $N$ is regular and its special fiber is a (possibly non-reduced) divisor with simple normal crossings. 
   
   Here a $\CG_{\CO_E}$-equivariant modification of $\RM^\loc_{\CG,\mu}$ is a $\CG_{\CO_E}$-equivariant proper birational morphism to $\RM^\loc_{\CG,\mu}$ over $\CO_E$. 
\end{conj}  

Via the local model diagram, an equivariant modification as in Conjecture \ref{intro-conj} gives rise to a regular integral model of the corresponding Shimura variety. For previous results toward this conjecture and related constructions, see \S \ref{section-canonicalloc}.

The purpose of the present paper is to construct (potentially) semi-stable $p$-adic integral models of Shimura varieties, with maximal parahoric level, for a large class of PEL type Shimura varieties. 

We now state the main result more precisely. Let $p>2$ be a prime number throughout the paper. Let $$(\bB, *,\bV, \pair{\ ,\ })$$ be a PEL datum: $\bB$ is a finite dimensional semisimple $\BQ$-algebra with positive involution $*$, and $\bV$ is a non-degenerate skew-hermitian $\bB$-module with respect to the non-degenerate alternating pairing $\pair{\ ,\ }$. Set \begin{flalign*}
	\bG \coloneqq \cbra{g\in \GL_B(V)\ |\ \pair{gx,gy}=c(g)\pair{x,y} \text{\ for $x,y\in V$ and some $c(g)\in \BG_m$}}.
\end{flalign*}
Let $\bh: \Res_{\BC/\BR}\BG_m\ra \bG_\BR$ be a homomorphism defining on $\bV_\BR$ a Hodge structure of type $\cbra{(0,1),(1,0)}$ such that the pairing $\pair{-,\bh(i)-}$ is symmetric and positive definite on $\bV_\BR$. Let $\rK_p\sset \bG(\BQ_p)$ be a maximal parahoric subgroup. For $\rK=\rK_p\rK^p$ with sufficiently small $\rK^p$, denote by $\Sh_\rK(\bG,X)$ the Shimura variety over the reflex field $\bE$ attached to $(\bG,\bh)$. We assume the following: 
    
    
\begin{enumerate}
    \item[($\dagger$)]  \label{dagger} $\bB_{\BQ_p}$ is isomorphic to a product of matrix algebras over finite field extensions of $\BQ_p$, such that factors are either stable by the action of $*$ or exchanged two by two by $*$. 
    
    Moreover, we assume that the adjoint group $\bG_{\BQ_p}^\ad$ has \dfn{no} factors of ramified unitary groups. 
    
\end{enumerate}

Let $\sS_\rK$ be the associated canonical integral model constructed in \cite[\S 8.2.4]{PZ}; see also \cite{kisin2026integral}. Let $\CG$ be the parahoric group scheme corresponding to $\rK_p$, and let $\mu$ be the $\bG(\ol{\bQ}_p)$-conjugacy class of the cocharacter associated with the Hodge cocharacter induced by $\bh$. Write $\RM^\loc_{\CG,\mu}$ for the associated canonical local model. Since $p>2$ and $(\bG,\bh)$ is of PEL type, there exists a local model diagram over $\CO_E$ 
\begin{equation}
\begin{tikzcd}[column sep={5em, between origins}, row sep={2em, between origins}]
&  \arrow[ld, "\alpha"']\wt{\sS}_\rK\arrow[rd, "\beta"]&\\
\sS_\rK   &&   \RM^\loc_{\CG,\mu} \ ,
\end{tikzcd}
\end{equation}
where $\alpha$ is a $\CG$-torsor and $\beta$ is smooth and $\CG$-equivariant; {see Remark \ref{rmk-lmdpara}}.

Our main local result is the following.

\begin{thm}\label{IntroductionMainResultMloc}
	Assume $(\dagger)$. Let $\CG$ be a maximal parahoric group scheme. Then there exists a finite extension $E'$ of $E$ and a $\CG_{\CO_{E'}}$-equivariant modification $\RM^{\rm ss}\ra \RM^\loc_{\CG,\mu}\otimes_{\CO_E}\CO_{E'}$ such that $\RM^{\rm ss}$ has semi-stable reduction over $\CO_{E'}$.
\end{thm}

\begin{remark}
The finite extension $E'$ in the above theorem may be chosen explicitly: $E'$ is the compositum of $E$ with fields obtained from the Galois closures of the centers of the matrix-algebra factors in $(\dagger)$ by at most two successive quadratic extensions; see \S\ref{NcyclesEcohomology}.
\end{remark}

\begin{corollary}
Assume $(\dagger)$ and suppose that $\CG$ is maximal parahoric. Then
Conjecture \ref{intro-conj} holds after a finite extension of the reflex field. 
\end{corollary}

We emphasize that our construction is explicit and uniform (to some extent): the regular integral models $\RM^{\rm ss}$ are obtained by successive blow-ups of the canonical local models or splitting models along (affine) Schubert varieties in the special fiber. Along the way, we also obtain an explicit moduli description of these Schubert varieties, which is of independent interest. 

As a consequence of the local model diagram, we obtain the following global result.

\begin{corollary}\label{introcoro-14}
	Assume $(\dagger)$. Then, after a finite extension $E'$ of $E$, there exists a semi-stable integral model $\sS^{\rm ss}_\rK$ of $\Sh_{\rK}(\bG,X)_{E'}$ with maximal parahoric level at $p$. 
\end{corollary}
By construction, we have a projective morphism $$f\colon \sS^{\rm ss}\coloneqq \sS^{\rm ss}_\rK\lra \sS\coloneqq \sS_{\rK}.$$ 
Denote by $R\Psi\coloneqq R\Psi^\sS(\text{resp. } R\Psi^{\rm ss}\coloneqq R\Psi^{\sS^{\rm ss}})$ the nearby cycles sheaves for $\sS$ (resp. $\sS^{\rm ss}$); see \S\ref{NcyclesEcohomology} for more details. 

\begin{corollary}[Theorem \ref{prop:inertia-action}]
    Assume $(\dagger)$. Let $E'/E$ be the finite extension as in Theorem \ref{IntroductionMainResultMloc} (see \eqref{eq-E'} for an explicit description). Write $I_{p}$ for the inertia group of $E'$. Then we have the following. 

\begin{altenumerate}
    \item For every prime
$\ell\neq p$ and any integer $r\geq 0$, the inertia group $I_p$ acts trivially on $R^r\Psi^{\rm ss}(\ol{\BQ}_\ell)$, and $I_p$ acts unipotently on $R^r\Psi(\ol{\BQ}_\ell)$.
\item We have a $\Gal(\ol{\BQ}_p/E')$-equivariant isomorphism $Rf_{\ol{s},*} R\Psi^{{\rm ss}}(\ol{\mathbb Q}_\ell)\simeq R\Psi(\ol{\mathbb Q}_\ell)$. In particular, we have \begin{flalign*}
    &H^r_c(\sS^{\rm ss}_{\ol{s}},R\Psi^{\rm ss}(\ol{\BQ}_\ell))\simeq H^r_c(\sS_{\ol{s}},R\Psi(\ol{\BQ}_\ell)) \text{\ and\ } \\  &H^r(\sS^{\rm ss}_{\ol{s}},R\Psi^{\rm ss}(\ol{\BQ}_\ell))\simeq H^r(\sS_{\ol{s}},R\Psi(\ol{\BQ}_\ell))
\end{flalign*} as $\Gal(\ol{\BQ}_p/E')$-modules. Moreover, $I_p$ acts unipotently on the above \etale cohomologies.
   \item For each integer $r\geq 0$, we have \begin{flalign*}
    &H^r_c(\sS_{\ol{\eta}},\ol{\BQ}_\ell)\simeq H^r_c(\sS_{\ol{s}},R\Psi(\ol{\BQ}_\ell))\ \text{and\ }
    H^r(\sS_{\ol{\eta}},\ol{\BQ}_\ell)\simeq H^r(\sS_{\ol{s}},R\Psi(\ol{\BQ}_\ell))
\end{flalign*} as $\Gal(\ol{\BQ}_p/E')$-modules. In particular, $I_p$ acts unipotently on
\[
   H^r\bigl(
      \operatorname{Sh}_{\rK}(\bG,X)_{\overline E},
      \ol{\mathbb Q}_\ell
   \bigr) \text{\ and\ } H^r_c\bigl(
      \operatorname{Sh}_{\rK}(\bG,X)_{\overline E},
      \ol{\mathbb Q}_\ell
   \bigr)
\]
for every $r\geq0$.
\end{altenumerate}
\end{corollary}

\begin{remark}
   \begin{altenumerate}
       \item 
        The case in which $\bG^\ad_{\BQ_p}$ has a ramified unitary factor will be treated in forthcoming work.
       \item In Corollary \ref{introcoro-14}, we only discuss Shimura varieties of PEL type. However, similar results should also apply to certain Shimura varieties of abelian type whose underlying reductive group differs from $\bG$ by a central extension; canonical local models are isomorphic under such central extensions by Proposition \ref{introprop1} (3).
   \end{altenumerate} 
\end{remark}


\subsection{Strategy of the proof}
We now outline the proof. There are three main reduction steps. First, we use functorial properties of $\RM^\loc_{\CG,\mu}$ under unramified base change, products, and central extension.

\begin{prop}\label{introprop1}
Let $(G,\CG,\mu)$ be a local model triple ({\rm LM} triple) over a $p$-adic field $F$ with reflex field $E$. 
    \begin{altenumerate}
    	\item Let $F'/F$ be a finite unramified extension, then the {\rm LM} triple $(G\otimes_FF', \CG\otimes_{\CO_F}\CO_{F'},\mu\otimes_FF')$ has reflex field $E'=EF'$, and \begin{flalign*}
		\RM^\loc_{\CG,\mu}\otimes_{\CO_E}\CO_{E'} \simeq \RM^\loc_{\CG\otimes\CO_{F'},\mu\otimes F'}.
	\end{flalign*}
	    \item Let $(G_i,\CG_i,\mu_i)$ be two {\rm LM} triples over $F$ with reflex fields $E_i$ for $i=1,2$.  If $(G,\CG,\mu) = (G_1,\CG_1,\mu_1)\times (G_2,\CG_2,\mu_2)$, then for $E=E_1E_2$, we have \begin{flalign*}
	    	\RM^\loc_{\CG,\mu} \simeq (\RM^\loc_{\CG_1,\mu_1}\otimes_{\CO_{E_1}}\CO_E)\times (\RM^\loc_{\CG_2,\mu_2}\otimes_{\CO_{E_2}}\CO_E).	
	    \end{flalign*}
	    \item If $\phi: (G,\CG,\mu)\ra (G',\CG',\mu')$ is a morphism of {\rm LM} triples such that $\phi: G\ra G'$ is a central extension of $G'$, then $\phi$ induces a $\CG_{\CO_E}$-equivariant isomorphism \begin{flalign*}
	    	\RM^\loc_{\CG,\mu}\simto \RM^\loc_{\CG',\mu'}\otimes_{\CO_{E'}}\CO_E.
	    \end{flalign*}
    \end{altenumerate}
\end{prop}
The corresponding statements are proved in the tamely ramified case in \cite[Proposition 2.14]{HPR}. Here we prove them in the general setting.

Second, we reduce the case for Weil restrictions to products of local models after passing to a splitting field. Here we use the splitting models introduced in \cite{PR2}; see also \cite{bijakowski2023geometry, shen2025f}.

\begin{prop}\label{introprop2}
	Let $K/F$ be a finite extension of $p$-adic fields. Let $L$ be the Galois closure of $K$ in $\ol{F}$. Let $G=\Res_{K/F}H$, where $H$ is $\GL_g, \GSp_{2g}$ or $\GO_{2g}$ over $K$.  We fix an ordering of the embeddings $\iota_i: K\ra \ol{F}$ for $i=1,\ldots, [K:F]$. Let $\mu$ be a geometric minuscule cocharacter of $G$, equivalently a tuple $(\mu_i)_{1\leq i\leq [K:F]}$. For the parahoric group scheme $\CG\coloneqq \Res_{\CO_K/\CO_F}\CH$, where $\CH$ is a parahoric group scheme, the associated splitting model $$\RM^\spl(\CG,\mu,(\iota_i))$$ is \etale locally isomorphic to $\prod_i\RM^\loc_{\CH_{\CO_L},\mu_i}$ over $\CO_L$. See Proposition \ref{prop-splitting} for a more precise statement. 
\end{prop}

Third, we use Hartl’s result on products of semi-stable schemes. 

\begin{prop}[{\cite[Proposition 2.1]{Ha}}] \label{introprop3}
	Let $X_1,\ldots, X_n$ be semi-stable schemes over a discrete valuation ring. Let $(Y_{ij})_{1\leq j\leq d_i}$ denote the irreducible components of the special fiber $\ol{X}_i$ of $X_i$. Here $d_i$ is the number of irreducible components of $\ol{X}_i$.  Then, after successively  blowing up ($(\prod_{i=1}^nd_i)$-times in total) closed subschemes $Y_{1j_1}\times\cdots\times Y_{nj_n}$, the product scheme $X_1\times\cdots\times X_n$ becomes a semi-stable scheme.
\end{prop}

Using Propositions \ref{introprop1} -- \ref{introprop3}, the proof of the main theorem is reduced to the following statement.

\begin{thm}\label{introthm-sstable}
	Let $F_0$ be a $p$-adic field with uniformizer $\pi_0$. Suppose that $$G=\GL_g,\ \GSp_{2g}\  \text{\ or\ }\ \GO(V,(\ ,\ )),$$ where $(V,(\ ,\ ))$ is a non-degenerate symmetric space of $F_0$-dimension $2g$, is a reductive group over $F_0$. Let $(G,\CG,\mu)$ be a {\rm LM} triple of PEL type with reflex field $E$. Note that $E$ equals $F_0$ or $F=F_0(\sqrt{-\pi_0})$. Assume that $\CG$ is maximal parahoric. 
	
	Then there exists a chain of (affine) Schubert varieties $S_0\sset S_1\sset \cdots \sset S_{d-1}$ in the special fiber of $\RM^\loc_{\CG,\mu}\otimes_{\CO_E}\CO_F$ such that $Z_d$ is semi-stable in the following sequence of blow-ups \begin{flalign*}
		\RM^\loc_{\CG,\mu}\otimes_{\CO_E}\CO_F=Z_0\xleftarrow{b_1} Z_1\xleftarrow{b_2} Z_2\xleftarrow{}\cdots \xleftarrow{b_d}Z_d.
	\end{flalign*}
	Here $b_1$ is the blow-up of $Z_0$ along $S_0$, and for $2\leq j\leq d$, the morphism $b_j$ is the blow-up of $Z_{j-1}$ along the strict transform of $S_{j-1}$ in $Z_{j-1}$. 
	
	In particular, there exists a semi-stable $\CG_{\CO_F}$-equivariant modification $$Z_d\ra \RM^\loc_{\CG,\mu}\otimes_{\CO_E}\CO_F.$$ 
\end{thm}

As a corollary, by local model diagram (see e.g., \cite[Theorem 1.1.2]{kisin2026integral}), 
Theorem \ref{introthm-sstable} produces (potentially) semi-stable integral models of $\Sh_\rK(\bG,X)$ whenever the induced {\rm LM} triple $(G,\CG,\mu)$ satisfies the assumptions in Theorem \ref{introthm-sstable}.

Let us briefly explain the proof of Theorem \ref{introthm-sstable}. When $G=\GL_g$, every maximal parahoric group scheme is hyperspecial, and hence the corresponding local model $\RM^\loc_{\CG,\mu}$ is smooth over $\CO_E$. Then no blow-up is needed. 

The main cases are $G=\GSp_{2g}$ or $\GO(V,(\ ,\ ))$. Using the results of \cite{Goertz03, Yang25,YZZ26}, we show that, for maximal parahoric level, the Schubert varieties of codimension $\geq 1$ in the special fiber of $\RM^\loc_{\CG,\mu}$ form a chain \begin{flalign*}
	S_0\sset S_1\sset\cdots\sset S_{d-1}
\end{flalign*}
for some integer $d$. We also have explicit moduli descriptions of these Schubert varieties; in the symplectic case this description is obtained here (see Proposition \ref{sympschubertmoduli}). We then construct an open affine subscheme $\RU=\Spec A$ of $\RM^\loc_{\CG,\mu}$, whose $\CG$-translates cover $\RM^\loc_{\CG,\mu}$. The ring $A$ admits an explicit presentation as a quotient of a matrix algebra $\CO_E[X]$, where $X$ is a generic square matrix. Moreover, the intersection $\RU\cap S_j$ is cut out by the ideal $(\wedge^{j+1}X)$ generated by $(j+1)\times (j+1)$ minors of $X$. 

Thus, the problem reduces to a concrete question about successive blow-ups of matrix-defined schemes along determinantal ideals $(\wedge^{j+1}X)$. Although direct computations of such successive blow-ups can be complicated, we develop an inductive method which makes the calculations tractable and proves semi-stability. This is one of the main technical inputs of the paper; see
Proposition \ref{prop:qs-affine-semistable} and the corresponding statements (Proposition \ref{prop:split-affine-semistable} and \ref{prop:symp-affine-semistable})
in the split orthogonal and symplectic cases for precise formulations. 
We believe that this inductive method is of independent interest and should have applications to related problems on resolving singularities of matrix-defined schemes; cf. \cite{gaube2024desingularization}.


\subsection{Organization}
The paper is organized as follows. 

In Part \ref{part-lmspl}, we collect the general facts on canonical local models and splitting models that will be used throughout the paper. In \S \ref{section-canonicalloc}, we recall the concept of canonical local models and establish their functorial properties. This allows us to reduce the study of local models attached to a general PEL datum to the local models considered in Part \ref{part-orth} and \ref{part-symp}. In Section \ref{sec-splitting}, we review the (naive) splitting models of Pappas--Rapoport for Weil-restricted groups. We describe the splitting-model diagram and deduce that the splitting model is \etale locally isomorphic to a product of local models, when the PEL datum has no factors of ramified unitary groups. 

In Part \ref{part-orth}, we prove Theorem \ref{IntroductionMainResultMloc} for even orthogonal local models. In \S \ref{orthogonal-local-models}, we recall the orthogonal local models at maximal parahoric level studied in \cite{Yang25, YZZ26}. We describe their Schubert stratifications in the special fiber and the affine charts around the minimal Schubert strata. In \S \ref{sec-qsorth} and \S \ref{s.s. resolution 2}, we obtain semi-stable models of orthogonal local models via successively blowing up Schubert varieties. Using the results of \S \ref{orthogonal-local-models}, the problem reduces to explicitly computing the blow-ups of certain matrix-defined schemes along determinantal ideals. The key ingredient is an inductive procedure on the size of the matrix appearing in the relavant affine chart; see the proof of Proposition \ref{prop:qs-affine-semistable}. In \S \ref{section ss ortho Res}, we consider the case of Weil restriction of orthogonal groups. Using the splitting-model diagram of Proposition \ref{prop-splitting} in \S \ref{sec-splitting} and the semi-stable resolutions constructed in \S \ref{sec-qsorth} and \ref{s.s. resolution 2}, we resolve the singularities of local models for Weil restricted orthogonal groups by applying Hartl’s construction \cite{Ha} for products of semi-stable schemes.

In Part \ref{part-symp}, we prove Theorem \ref{IntroductionMainResultMloc} for symplectic local models. In \S \ref{symplectic-local-models}, we recall the symplectic local models at maximal parahoric level and describe their Schubert stratification. We give an explicit moduli-theoretic description of the Schubert varieties in the special fiber and identify their intersections with the standard affine chart by determinantal equations. In \S \ref{symp-semistable-resolution} and \ref{section ss sym Res}, we construct semi-stable modifications of local models for (Weil restricted) symplectic local models. The constructions and proofs are analogous to those used in Part \ref{part-orth}.

Finally, Part \ref{part-application} discusses the global applications to Shimura varieties. We resolve the singularities of canonical integral models of PEL-type Shimura varieties via the local model diagram and the results in the preceding parts, under the assumption of $(\dagger)$. In particular, we obtain semi-stable integral models of these Shimura varieties after finite extensions of the reflex fields. {We also describe these extensions explicitly and show that the
corresponding inertia groups act unipotently on nearby cycles and $\ell$-adic cohomology of
the geometric generic fibers.}





{\bf Acknowledgements:} 
Part of this work was carried out during the visit of J. Yang and Z. Zhao to the University of Münster, which was supported by Eva Viehmann's Leibniz Prize. The authors would like to thank the University of Münster for its hospitality and excellent working environment, and E. Viehmann and U. Hartl for useful discussions. {We thank G. Pappas, Y. Luo and P. Wu for helpful comments and discussions.} {The authors are grateful to M. Rapoport for suggesting that they
describe explicitly the extension required for semi-stable reduction
and record the resulting consequences for nearby cycles and
$\ell$-adic cohomology.} The author J. Yang was supported by the Shuimu Tsinghua Scholar Program of Tsinghua
University. The author I. Zachos was supported by Germany's Excellence Strategy EXC~2044--390685587 ``Mathematics M\"unster: Dynamics--Geometry--Structure" and by the CRC~1442 ``Geometry: Deformations and Rigidity'' of the DFG. The author Z. Zhao was supported by the Tianyuan Fund for Mathematics of NSFC (Grant No. 12526543),  and the Fundamental Research Funds for the Central Universities (Grant No. FRF-TP-26-121).

\part{Local models and splitting models} \label{part-lmspl}

\section{Canonical local models}\label{section-canonicalloc} 
  Let $F$ be a finite extension of $\BQ_p$ or $\breve \BQ_p$.
\begin{defn} \label{defn-lmt}
    A \dfn{local model triple}, or simply \dfn{{\rm LM} triple}, over $F$ is a triple $$(G,\CG,\mu),$$ where $G$ is a \emph{connected} reductive group over $F$, $\CG$ is a quasi-parahoric group scheme for $G$ in the sense of \cite[\S 2.1.1]{kisin2026integral}, and ${\mu}$ is the $G(\ol{F})$-conjugacy class of a \dfn{minuscule} cocharacter $\mu: \BG_{m,\ol{F}}\ra G_{\ol{F}}$. We will also simply write $(\CG,\mu)$ for $(G,\CG,\mu)$.  
    
    A morphism of local model triples $(\CG,\mu)\ra (\CG',\mu')$ is a group scheme homomorphism $\CG\ra \CG'$ taking ${\mu}$ to ${\mu'}$. 
\end{defn}

Let $(\CG,\mu)$ be a local model triple over $F$. Denote by $E$ the reflex field of ${\mu}$. One can associate the v-sheaf local model \begin{equation*}
	\CM_{\CG,\mu}\sset \Gr_\CG  \label{vsheaf}
\end{equation*} in the Beilinson--Drinfeld Grassmannian $\Gr_\CG$ over $\Spd\CO_E$; see \cite[\S 4.2]{AGLR22}. We have the following result which is conjectured by Scholze--Weinstein \cite[Conjecture 21.4.1]{scholze2020berkeley}.

\begin{thm}[{\cite[Theorem 1.2]{AGLR22}}] \label{thm-repre}
	There exists a unique flat, projective and normal $\CO_E$-scheme $\RM^\loc_{\CG,\mu}$ representing $\CM_{\CG,\mu}$ in the sense that $(\RM^\loc_{\CG,\mu})^\Diamond\simeq \CM_{\CG,\mu}$. Moreover, the geometric special fiber of $\RM^\loc_{\CG,\mu}$ is reduced. 
\end{thm}
We refer to \cite[\S 2.2]{AGLR22} for the definition of the diamond operator $\Diamond$. We will call $\RM^\loc_{\CG,\mu}$ in the above theorem \dfn{canonical local model}.

We record the following functorial properties of canonical local models; see \cite[Proposition 2.1.4]{HPR} for the tamely ramified case.
\begin{prop}\label{prop-mloc}
Let $(\CG,\mu)$ be a local model triple over $F$ with reflex field $E$. 
    \begin{altenumerate}
    	\item Let $F'/F$ be a finite unramified extension, then the {\rm LM} triple $(G\otimes_FF', \CG\otimes_{\CO_F}\CO_{F'},\mu\otimes_FF')$ has reflex field $E'=EF'$, and \begin{flalign*}
		\RM^\loc_{\CG,\mu}\otimes_{\CO_E}\CO_{E'} \simeq \RM^\loc_{\CG\otimes\CO_{F'},\mu\otimes F'}.
	\end{flalign*}
	    \item Let $(G_i,\CG_i,\mu_i)$ be two {\rm LM} triples over $F$ with reflex fields $E_i$ for $i=1,2$.  If $(G,\CG,\mu) = (G_1,\CG_1,\mu_1)\times (G_2,\CG_2,\mu_2)$, then for $E=E_1E_2$, we have \begin{flalign*}
	    	\RM^\loc_{\CG,\mu} \simeq (\RM^\loc_{\CG_1,\mu_1}\otimes_{\CO_{E_1}}\CO_E)\times (\RM^\loc_{\CG_2,\mu_2}\otimes_{\CO_{E_2}}\CO_E).	
	    \end{flalign*}
	    \item If $\phi: (G,\CG,\mu)\ra (G',\CG',\mu')$ is a morphism of {\rm LM} triples such that $\phi: G\ra G'$ is a central extension of $G'$, then $\phi$ induces a $\CG_{\CO_E}$-equivariant isomorphism \begin{flalign*}
	    	\RM^\loc_{\CG,\mu}\simto \RM^\loc_{\CG',\mu'}\otimes_{\CO_{E'}}\CO_E.
	    \end{flalign*}
	     \item $\RM^\loc_{\CG,\mu}\simeq \RM^\loc_{\CG^\circ,\mu}$, where $\CG^\circ$ denotes the neutral component of $\CG$. 
    \end{altenumerate}
\end{prop}
\begin{proof}
	Part (4) is proven in \cite[Proposition 21.5.1]{scholze2020berkeley}. Next we show (1) to (3). It is clear that taking $\Diamond$ is compatible with base change and products. Note that special fibers of all schemes in (1) to (3) are reduced by Theorem \ref{thm-repre}. Thus, these schemes are normal by \cite[Proposition 9.2]{PZ}. Hence, by the uniqueness in  Theorem \ref{thm-repre}, it suffices to prove $\CM_{\CG,\mu}$ is compatible with taking unramified base change, products and central extension of $(\CG,\mu)$. By functoriality, there is a natural morphism \begin{equation*}
		f\colon \CM_{\CG,\mu} \lra (\CM_{\CG_1,\mu_1}\otimes_{\CO_{E_1}}\CO_E) \times (\CM_{\CG_2,\mu_2}\otimes_{\CO_{E_2}}\CO_E).
	\end{equation*} By \cite[Proposition 4.16]{AGLR22}, $f$ becomes an isomorphism after base change to $\Spa\CO_C$ for a complete algebraic closure of $E$. By v-descent, we obtain that $f$ is an isomorphism. This shows (2). Similarly, we prove (3). Since taking v-closures commutes with unramified base change (see e.g., \cite[Corollary 2.9]{AGLR22}), to show (1), it remains to prove that $\Gr_\CG$ commutes with unramified base change. Let $k_F$ be the residue field of $F$. Recall that $\Gr_\CG$ is isomorphic to the quotient v-sheaf which sends an affinoid perfectoid $k_F$-algebra  $(R,R^+)$ with an untilt $(R^\sharp,R^{\sharp+})$ over $\CO_F$ to $G(B_\dR(R^\sharp))/\CG(B_\dR^+(R^\sharp))$, where $B_\dR^+(R^\sharp)$, respectively $B_\dR(R^\#)$, denotes the ring of de Rham periods formed using $\CO_F$-Witt vectors; cf. \cite[p. 138]{scholze2020berkeley} and \cite[\S 4.1]{AGLR22}. More precisely, there exists a natural surjection  \begin{flalign*}
		  \theta\colon W_{\CO_F}(R^+)[[\varpi]\inverse]\lra R^\sharp=R^{\sharp+}[\varpi^{\sharp-1}]
	\end{flalign*}
	with $\ker\theta = (\xi)$,
	where $\varpi$ is a pseudouniformizer such that $\varpi^\sharp\in R^\sharp$ satisfies $\varpi^{\sharp p}|p$; we define $B^+_\dR(R^\sharp)$ as the $\xi$-adic completion of $W_{\CO_F}(R^+)[[\varpi]\inverse]$, and define $B_\dR(R^\sharp)\coloneqq B^+_\dR(R^\sharp)[\xi\inverse]$. Since $F'/F$ is unramified, by \cite[Lemma 5.1]{bertapelle2020greenberg}, we have \begin{equation}
		W_{\CO_F}(A) \simeq W_{\CO_{F'}}(A)  \label{eq-unrW}
	\end{equation}  for a perfect $k_{F'}$-algebra. Here $k_{F'}$ denotes the residue field of $F'$. Note that any affinoid perfectoid $k_{F'}$-algebra $(R,R^+)$ is v-covered by affinoid perfectoid of the form $(B,B^+)$ where  $B^+$ is a product of valuation rings with complete algebraically closed fraction field in characteristic $p$; see \cite[\S 2.1.4]{pappas2024p}. In particular, $B^+$ is a perfect ring. By \eqref{eq-unrW}, we obtain that \begin{equation*}
		\Gr_{\CG\otimes_{\CO_F}\CO_{F'}} \simeq \Gr_\CG\otimes_{\CO_F}\CO_{F'}.
	\end{equation*}
	Hence, we prove (1).
\end{proof}

Note that, by construction, the canonical local model $\RM^\loc_{\CG,\mu}$ carries a natural action of $\CG_{\CO_E}$. Inspired by many concrete examples, it is reasonable to formulate the following conjecture.
\begin{conj}[{\cite[\S 2.12]{P},\cite[\S 8.4]{PRS}}] \label{conj-modif}
    There exists a $\CG_{\CO_E}$-\dfn{equivariant modification} $$N\lra \RM^\loc_{\CG,\mu},$$  which is an isomorphism on the generic fibers, such that $N$ is regular and its special fiber is a (possibly non-reduced) divisor with simple normal crossings.  
\end{conj}

Partial results in this direction have produced many interesting examples of regular models or toroidal resolution of Shimura varieties\footnote{Constructions involving splitting models may require a finite extension of the relevant reflex field.}; see Table \ref{table1}. 
\begin{table}[htbp]
\centering
\scriptsize
\renewcommand{\arraystretch}{1.5}
\begin{tabular}{p{0.14\textwidth}p{0.14\textwidth}p{0.2\textwidth}p{0.38\textwidth}}
\hline
\textbf{Group} 
& \textbf{Minuscule cocharacter} 
& \textbf{Parahoric level} 
& \textbf{Result and construction} \\
\hline

Ramified \(\GU(V,\phi)\) &\(\mu_{n-1,1}\)  &Self-dual lattice
&Pappas \cite{P} and Kr\"amer \cite{Kr}: semi-stable models via blowing up of wedge local models or using splitting models \\

\hline
Ramified \(\GU(V,\phi)\) & \(\mu_{2,n-2}\)
&Self-dual lattice
&Zachos \cite{Zac1}: semi-stable models via blowing up the splitting models \\

\hline

Ramified \(\GU(V,\phi)\)
&\(\mu_{n-s,s}\) \newline \((s=1,2,3)\)
&\(\pi\)-modular lattice,\newline with \(n\) even
&Zachos--Zhao \cite{ZacZhao}: smooth models for \(s=1\) and semi-stable for \(s=2,3\), via splitting models \\

\hline

Ramified \(\GU(V,\phi)\)
&\(\mu_{n-1,1}\)
&Vertex level
&Zachos--Zhao \cite{zachos2025semi} and He--Luo--Shi \cite{he2024regular}: semi-stable models via (variants of) splitting models \\

\hline

$\Res_{F/F_0}\GL_n$
&arbitrary
&Weil restriction of\newline a hyperspecial level
&Pappas--Rapoport \cite{PR2}: smooth models via splitting models \\

\hline

$\Res_{F/F_0}\GSp_{2g}$
&Siegel
&Weil restriction of\newline a hyperspecial level
&Pappas--Rapoport \cite{PR2}: smooth models via splitting models \\

\hline

$\GSpin(V,q)$
&$\mu_1$
&Vertex level
&Pappas--Zachos \cite{PaZa}: regular models via blowing up of the local models \\

\hline

$\GSp_6$
&$\mu_3$
&Iwahori level
&Genestier\cite{genestier2000modele}: semi-stable models via blow-ups
\\
\hline

$\GSp_{2g}$
&\(\mu_g\)
&Siegel parahoric
&Faltings \cite{faltings1997explicit}: semi-stable models via explicit resolution of schemes of matrices\\
\hline

$\GL_n$
&$\mu_2$
&arbitrary
&Faltings \cite{faltings2001toroidal}: regular models via toroidal resolution of local models
\\
\hline
$\GL_5$
&
\(\mu_2\)
&
Iwahori level 
&
Gora \cite{gora2023local}: semi-stable models via blowing up Mustafin varieties. \\

%
%
%

\hline
\end{tabular}

\caption{Previous results}  \label{table1}
\end{table}
\FloatBarrier

We use the following terminology in Conjecture \ref{conj-modif}.
\begin{defn}
    \begin{altenumerate}
        \item A \dfn{$\CG_{\CO_E}$-equivariant modification} $N$ of the local model $\RM^\loc_{\CG,\mu}$ is a $\CG_{\CO_E}$-equivariant proper birational morphism $N\ra \RM^\loc_{\CG,\mu}$ over $\CO_E$.
    \item Let $Y$ be a regular Noetherian scheme of dimension $n$, and let $D$ be an effective Cartier divisor on $Y$ with the corresponding ideal sheaf $\CO_Y(-D)$. We say that (\cite[Definition 1.6, \S 9.1]{liu2006algebraic}) $D$ \dfn{has simple normal crossings}, or $D$ is a \dfn{divisor with simple normal crossings}, at a closed point $y\in Y$ if the maximal ideal $\fm_y$ can be generated by $n$ elements $f_1,\ldots,f_n$ such that $\CO_Y(-D)_y$ is generated by $f_1^{r_1}\cdots f_m^{r_m}$ for some integer $0\leq m\leq n$. We say that $D$ has normal crossings if it has simple normal crossings at every closed point $y\in Y$.

    We say that $D$ has \dfn{normal crossings} if there exists an \etale covering $f\colon Z\ra Y$ such that $f^*D$ has simple normal crossings. 
    \end{altenumerate}
\end{defn}

\section{Splitting models} \label{sec-splitting}
In this section, we review the splitting models introduced in \cite{PR2} (see also \cite{bijakowski2023geometry, shen2025f}) for Weil-restricted groups. We use the following notation:
\begin{altitemize}
    \item  $F_0$ is a complete discretely valued field with ring of integers $\CO_{F_0}$, uniformizer $\pi_0$, and perfect residue field $k_{F_0}$ of characteristic $p>2$;
    \item $B$ is a finite semi-simple algebra with involution $*$; we assume that \begin{flalign}
    	B\simeq \prod_{i\in I_1} M_{m_i}(F_i)  \label{Bprod}
    \end{flalign}  is isomorphic to a (finite) product of matrix algebras over finite field extensions $F_i$ of $F_0$ such that the factors $M_{m_i}(F_i)$ are either stable by $*$ or exchanged two by two by $*$;
    \item $V$ is a finite dimensional $B$-module equipped with a non-degenerate alternating $F_0$-bilinear form $\pair{\ ,\ }$ on $V$ such that $\pair{bv,w}=\pair{v,b^*w}$ for $v,w\in V$ and $b\in B$;
    \item $G$ denotes the reductive group over $F_0$, which sends an $F_0$-algebra $R$ to the group \begin{flalign*}
        G(R)\coloneqq \cbra{g\in \GL_{B\otimes_{F_0}R
    }(V\otimes_{F_0}R)\ |\ \pair{gv,gw}=c(g)\pair{v,w}, c(g)\in R^\times };
    \end{flalign*}
    denote by $c: G\ra \BG_m$ the similitude character;
    \item $\mu: \BG_{m,\ol{F}_0}\ra G_{\ol{F}_0}$ denotes a geometric cocharacter given up to conjugation; we assume that the corresponding eigenspace decomposition of $V\otimes_{F_0}\ol{F}_0$ is given by \begin{flalign}
        V\otimes_{F_0}\ol{F}_0 = V^0\oplus V^1  \label{eigen}
    \end{flalign} (i.e., the only weights are $0$ and $1$), and the composition $c\circ\mu$ is the identity;
    \item $\CO_B$ is a maximal order of $B$ stable under $*$-action.
\end{altitemize}
Note that $*$ induces an action on the index set $I_1$ in \eqref{Bprod}. Denote by $I_1/\sim_* $ the orbit space. Then by our assumption on $B$, we have \begin{flalign}
	B\simeq \prod_{[i]\in I_1/\sim_*}B_{[i]}.   \label{decompBi}
\end{flalign}
Here, if $i=*(i)$, then $B_{[i]}=M_{m_i}(F_i)$; if $i\neq *(i)$, then $B_{[i]}=M_{m_i}(F_i)\times M_{m_{*(i)}}(F_{*(i)})$ (necessarily $F_i= F_{*(i)}$ and $m_i=m_{*(i)}$). Denote by $F^+_i$  the set of $*$-invariants of the center of $B_{[i]}$.  
Following \cite[\S 2.2]{bijakowski2023geometry} and \cite[\S 2.1]{shen2025f}, we have five \dfn{types} of $[i]$ defined by \begin{itemize}
    \item [(AL):] if $i\neq *(i)$ and $*$-action on $B_{[i]}$ exchanges the two factors;
    \item [(AU):] if $i=*(i)$ and $F_i$ is an unramified quadratic extension of $F_i^+$, and $*$-action on $F_i$ corresponds to the unique non-trivial element in $\Gal(F_i/F_i^+)$;
    \item [(AR):] if $i=*(i)$ and $F_i$ is a ramified quadratic extension of $F_i^+$, and $*$-action on $F_i$ corresponds to the unique non-trivial element in $\Gal(F_i/F_i^+)$;
    \item [(C):] if $i=*(i)$ and $F_i=F_i^+$ and $*: B_{[i]}=M_{m_i}(F_i)\ra M_{m_i}(F_i)$ is given by $b^*=b^t$;
    \item [(D):] if $i=*(i)$ and $F_i=F_i^+$ and $*:B_{[i]}=M_{m_i}(F_i)\ra M_{m_i}(F_i)$ is given by $b^*=Jb^tJ\inverse$ for some non-degenerate alternating matrix $J$.
\end{itemize}  

Let $\CL$ denote a self-dual multichain of $\CO_B$-lattices in $V$ in the sense of \cite[Definition 3.4, 3.13]{RZbook}. 
Write $$\CG=\CG_\CL^\circ$$ for the neutral component of the similitude automorphism group of $\CL$. Let $E$ be the reflex field of $\mu$. 
For any $\CO_B$-lattice $\Lambda\sset V$, denote by \begin{flalign*}
    \Lambda^\vee \coloneqq \cbra{x\in V\ |\ \pair{x,\Lambda}\sset \CO_{F_0}}
\end{flalign*}
the dual lattice. Then the pairing $\pair{\ ,\ }$ induces a perfect bilinear pairing \begin{flalign}
    \pair{\ ,\ }: \Lambda\times \Lambda^\vee \lra \CO_{F_0}.  \label{perfpairing}
\end{flalign}
Let $b$ be a unit of $B$ which normalizes $\CO_B$.
If $M$ is an $(\CO_B\otimes_{\CO_{F_0}}\CO_S)$-module, we denote by $M^b$ the $(\CO_B\otimes_{\CO_{F_0}}\CO_S)$-module obtained by restriction of scalars with respect to the induced map $\CO_B\ra \CO_B, x\mapsto b\inverse xb$. Left multiplication by $b$ induces an $(\CO_B\otimes_{\CO_{F_0}}\CO_S)$-linear homomorphism $b: M^b\ra M$. 


\begin{defn}[{\cite[\S 3]{RZbook}}] \label{def-naivelocmod}
    Let $$\RM^\naive_\CL = \RM^\naive(\CL,\mu)$$ denote the \dfn{naive local model} which represents the functor sending an $\CO_E$-scheme $S$ to the set of $(\CO_B\otimes_{\CO_{F_0}}\CO_S)$-modules $(\CF_\Lambda)_{\Lambda\in \CL}$ such that \begin{altenumerate}
        \item for each $\Lambda\in\CL$, there exists an inclusion $$j_\Lambda\colon \CF_\Lambda\hookrightarrow \Lambda_S\coloneqq \Lambda\otimes_{\CO_{F_0}}\CO_S$$ such that $\CF_\Lambda$ is a locally direct summand as an $\CO_S$-module;
        \item for $\Lambda\sset \Lambda'$ in $\CL$, the natural map $\Lambda_S\ra \Lambda'_S$ induced by the inclusion $\Lambda\hookrightarrow \Lambda'$ sends $\CF_\Lambda$ to $\CF_{\Lambda'}$; and if $b\in B\cross$ that normalizes $\CO_B$, there are $(\CO_B\otimes_{\CO_{F_0}}\CO_S)$-module isomorphisms $$\theta_{b,\Lambda}\colon (\CF_\Lambda)^b\simto \CF_{b\Lambda},$$ which make the following commutative diagram:
        \begin{equation*}
    \begin{tikzcd}
        (\CF_\Lambda)^b \ar[r, "j_\Lambda"] \ar[d, "\theta_{b,\Lambda}"'] & 
        (\Lambda\otimes_{\CO_{F_0}}\CO_S)^b \ar[d, "b"] \\
        \CF_{b\Lambda} \ar[r, "j_{b\Lambda}"] & 
        b\Lambda\otimes_{\CO_{F_0}}\CO_S
    \end{tikzcd}
\end{equation*}
        \item for each $\Lambda\in\CL$, the $\CO_B$-action on $\CF_\Lambda$ satisfies the Kottwitz condition \begin{flalign*}
            \mathrm{det}_{\CO_S}(b\ |\ \CF_{\Lambda}) = \mathrm{det}_{\CO_S}(b\ |\ V^1),\ b\in \CO_B,
        \end{flalign*}
        see \eqref{eigen} for the definition of $V^1$;
        \item $\pair{\CF_\Lambda,\CF_{\Lambda^\vee}}=0$, i.e., $\CF_{\Lambda^\vee}$ is the orthogonal complement of $\CF_{\Lambda}$ with respect to the perfect pairing \begin{equation*}
            \pair{\ ,\ }\colon \Lambda_S \times \Lambda^\vee_S \lra \CO_S
        \end{equation*}
        induced by \eqref{perfpairing}.
    \end{altenumerate}
\end{defn}

Suppose from now on that \begin{itemize}
    \item  The PEL datum above has no factors of type (AR).
\end{itemize} 

Let $L$ denote the composite of the Galois closure of $F_i$ in $\ol{F}_0$, for $i\in I_1$. 
Let $F^\unra_i\sset F_i$ denote the maximal unramified extension of $F_0$ in $F_i$. Write $e_i\coloneqq [F_i:F_i^\unra]$ and $f_i\coloneqq [F_i^\unra:F_0]$. We order the embeddings $F_i^\unra\hookrightarrow L$ as $\sigma_{i,1},\ldots,\sigma_{i,f_i}$. For each $1\leq j\leq f_i$, there are $e_i$ extensions of $\sigma_{i,j}$ to embeddings $F_i\hookrightarrow L$, ordered and denoted by $\sigma_{i,j}^l\colon F_i\hookrightarrow L$, for $1\leq l\leq e_i$. Denote \begin{flalign*}
	\Sigma \coloneqq \cbra{\sigma_{i,j}^l\colon F_i\ra L}_{i\in I_1,1\leq j\leq f_i,1\leq l\leq e_i} \text{\ and\ } \Sigma^\unra\coloneqq \cbra{\sigma_{i,j}\colon F_i^\unra\ra L}_{i\in I_1, 1\leq j\leq f_i}.
\end{flalign*}  The $*$-action induces an action on $\Sigma$ and $\Sigma^\unra$. We choose an ordering on $\Sigma$ such that $\sigma_{i,j}^l=(\sigma^*_{i,j})^l$ for $1\leq l\leq e_i$.  
By Morita equivalence, we have that \begin{flalign*}
	{V\simeq \bigoplus_{i\in I_1} V_i^{m_i},}
\end{flalign*}
where $V_i$ is an ${F_i}$-vector space of dimension $d_i$. 
 We also have the decompositions \begin{flalign*}
	V_{i,L}=V_i\otimes_{F_0}L=\bigoplus_{1\leq j\leq f_i} V_{i,j} \text{\ and\ } V_{i,j}=\bigoplus_{1\leq l\leq e_i}V_{i,j}^l.
\end{flalign*}
Here, $V_{i,j}$ is the $F_i$-subspace of $V_{i,L}$ such that $F_i^\unra$ acts via $\sigma_{i,j}$ ($1\leq j\leq f_i$); and $V_{i,j}^l$ is the $F_i$-subspace such that $F_i$ acts via $\sigma_{i,j}^l$ ($1\leq l\leq e_i$). Since the $*$-action permutes $\sigma_{i,j}$, it permutes the indices $(i,j)$. The pairing $\pair{\ ,\ }$ on $V$ induces a non-degenerate $F_0$-bilinear alternating pairing \begin{flalign}
	\pair{\ ,\ }\colon V_{i,j}\times V_{*(i,j)}\ra F_0.  \label{Vipairing}
\end{flalign}   For an $\CO_B$-lattice $\Lambda\sset V$, we have the decompositions \begin{flalign*}
	\Lambda = \bigoplus_{i\in I_1}\Lambda_i^{m_i} \text{\ and\ } \Lambda_i = \bigoplus_{1\leq j\leq f_i}\Lambda_{i,j}.
\end{flalign*}
Then each $\Lambda_{i,j}$ is a lattice in $V_{i,j}$, stable under $\CO_{F_i}$-action.   
 The pairing $\pair{\ ,\ }$ in \eqref{Vipairing}  induces a perfect pairing \begin{flalign}
	\Lambda_{i,j}\times (\Lambda^\vee)_{*(i,j)}\lra \CO_{F_0}. \label{pairingLambda}
\end{flalign} 
Note that the cocharacter \begin{flalign*}
    \mu\colon \BG_{m,\ol{F}_0}\lra G_{\ol{F}_0}\sset \GL_B(V\otimes_{F_0}\ol{F}_0) = \prod_{\sigma_{i,j}^l\in\Sigma}\GL(V_{i,j}^l\otimes_{L}\ol{F}_0) 
\end{flalign*} is given by, up to conjugation, by $(1^{(r_{i,j}^l)},0^{(d_i-r_{i,j}^l)})_{\sigma_{i,j}^l\in\Sigma}$ with $d_i\coloneqq \dim_{F_i} V_i=\dim_LV_{i,j}^l$. Let $\pi_i$ be a uniformizer of $F_i$ for $i\in I_1$.

\begin{defn}[{cf. \cite[Definition 2.1]{shen2025f}, \cite[Definition 2.21]{bijakowski2023geometry}}]  \label{def-naivesplmod}
    Let \begin{flalign*}
        \RM^{\nspl}_\CL=\RM^\nspl({\CL,\mu,(\sigma_{i,j}^l)})
    \end{flalign*}
    denote the \dfn{naive splitting model}, a projective $\CO_L$-scheme, which represents the functor sending an $\CO_L$-scheme $S$ to the set of subspaces $(\CF^l_{\Lambda_{i,j}})$ for $\Lambda\in \CL, i\in I_1$, $1\leq j\leq f_i$ and $1\leq l\leq e_i$ such that the following hold. \begin{altenumerate}
    	\item There exists an injection $$\iota_{\Lambda_{i,j}}^l\colon \CF^l_{\Lambda_{i,j}}\hookrightarrow \Lambda_{i,j,S}=\Lambda_{i,j}\otimes_{\CO_{F_0}}\CO_S$$ of $\CO_{F_i}\otimes_{\CO_{F_0}}\CO_S$-modules, which makes $\CF_{\Lambda_{i,j}}^l$ a locally direct summand. 
    	\item If $b\in B^\times$ which normalizes $\CO_B$, there are $(\CO_{F_i}\otimes_{\CO_{F_0}}\CO_S)$-linear isomorphisms $$\theta_{b,\Lambda,{i,j}}^l\colon \CF_{\Lambda_{i,j}}^l \simto \CF_{(b\Lambda)_{i,j}}^l$$ which fit into the following commutative diagrams \begin{equation*}
    		\begin{tikzcd}
        \CF^l_{\Lambda_{i,j}} \ar[r, "\iota^l_{\Lambda_{i,j}}"] \ar[d, "\theta^l_{b,\Lambda,i,j}"'] & 
        \Lambda_{i,j,S} \ar[d, "b"] \\
        \CF^l_{(b\Lambda)_{i,j}} \ar[r, "\iota^l_{(b\Lambda)_{i,j}}"] & 
        (b\Lambda)_{i,j,S}.
    \end{tikzcd}
    	\end{equation*}
    	\item We have $\CF^{l-1}_{\Lambda_{i,j}}\sset \CF_{\Lambda_{i,j}}^l$ (via inclusions $\iota_{\Lambda_{i,j}}^l$) and the quotient $\CF^l_{\Lambda_{i,j}}/\CF^{l-1}_{\Lambda_{i,j}}$ is $\CO_S$-locally free of rank $r_{i,j}^l$ and is annihilated by \begin{flalign*}
    		 \pi_i\otimes 1 - 1\otimes\sigma_{i,j}^l(\pi_i)\in \CO_{F_i}\otimes_{\CO_{F_0}}\CO_S.
    	\end{flalign*}
    	\item With respect to the perfect pairing \begin{equation*}
            \pair{\ ,\ }\colon \Lambda_{i,j,S} \times (\Lambda^\vee)_{*(i,j),S} \lra \CO_S,
        \end{equation*} induced by \eqref{pairingLambda},
        we have $\pair{\CF^l_{\Lambda_{i,j}},\CF^l_{(\Lambda^\vee)_{*(i,j)}}}=0$, i.e., $\CF^l_{(\Lambda^\vee)_{*(i,j)}}$ lies in the orthogonal complement $(\CF^l_{\Lambda_{i,j}})^\perp$ of $\CF^l_{\Lambda_{i,j}}$.
        \item Set \begin{flalign*}
        	Q^{>l}_{i,j}(T)\coloneqq \prod_{a=l+1}^{e_i}(T-\sigma_{i,j}^a(\pi_i)) \in \CO_L[T].
        \end{flalign*}
        We require that $$(\CF^l_{\Lambda_{i,j}})^\perp = Q^{>l}_{i,j}(\pi_i)\inverse \CF_{(\Lambda^\vee)_{*(i,j)}}^l. $$ 
    \end{altenumerate} 
\end{defn}

For an $\CO_L$-scheme $S$ and a point $(\CF^l_{\Lambda_{i,j}})\in \RM^\nspl_\CL(S)$, the data $(\CF^{e_i}_{\Lambda_{i,j}})$ determine, by Morita equivalence, subspaces $(\CF_\Lambda\sset \Lambda_S)_{\Lambda\in\CL}$, which corresponds to a point in $\RM^\naive_\CL(S)$. Thus, we obtain a natural map \begin{flalign*}
	f\colon \RM^\nspl_\CL \lra \RM^\naive_\CL\otimes_{\CO_E}\CO_L.
\end{flalign*}
Note that $f$ induces an isomorphism over the generic fibers. For $\sigma=\sigma_{i,j}^l\in \Sigma$, let $\mu_\sigma$ denote the $\sigma$-component of $\mu$ corresponding to $(1^{(r_{i,j}^l)},0^{(d_i-r_{i,j}^l)})$. Set \begin{flalign*}
	\Lambda_\sigma\coloneqq \Lambda_{i,j}\otimes_{\CO_{F_i},\sigma_{i,j}^l}\CO_L.
\end{flalign*}
Denote by $\Sigma/\sim_*$ the set of orbits under the $*$-action on $\Sigma$. We write $[\sigma]\coloneqq \cbra{\sigma,\sigma^*}$ for the orbit of $\sigma\in \Sigma$.  
Then \begin{flalign*}
	\CL_{[\sigma]}= \cbra{\Lambda_\sigma,\Lambda_{\sigma^*}}_{\Lambda\in\CL}
\end{flalign*} defines a self-dual $\CO_L$-multichain.
Let $\CG_{[\sigma]}$ denote the associated parahoric group scheme. By \cite[pp. 56]{shen2025f}, we have a diagram of $\CO_L$-schemes 
\begin{equation}
\begin{tikzcd}[column sep={6em, between origins}, row sep={3em, between origins}]
&  \arrow[ld, "\alpha_1"']\wt{\RM}^\nspl_\CL\arrow[rd, "\alpha_2"]&\\
\RM^\nspl_\CL   &&   \prod_{[\sigma]\in\Sigma/\sim_*} \RM^\naive(\CL_{[\sigma]},\mu_{[\sigma]}), \label{diag-torsor}
\end{tikzcd}
\end{equation}
where \begin{itemize}
	\item for an $\CO_L$-scheme $S$, $\wt{\RM}^\nspl_\CL(S)$ is the set of \begin{flalign*}
	((\CF_{\Lambda_{i,j}}^l)_{1\leq l\leq e_i}, (\varphi^l_{\Lambda_{i,j}})_{2\leq l\leq e_i})_{\Lambda\in\CL},
\end{flalign*} 
          where $(\CF_{\Lambda_{i,j}}^l)\in \RM^\nspl_\CL(S)$ and $\varphi_{\Lambda_{i,j}}^l\colon \Upsilon^l_{\Lambda_{i,j}}\simeq \Lambda_{i,j}^l\otimes_{\CO_L}\CO_S$ for $$\Upsilon^l_{\Lambda_{i,j}}\coloneqq \ker(\pi_i\otimes 1-1\otimes \sigma_{i,j}^l(\pi_i)\ |\ (\Lambda_{i,j}\otimes_{\CO_{F_0}}\CO_S)/\CF_{\Lambda_{i,j}}^{l-1});$$
          the trivialization $\varphi_{\Lambda_{i,j}}^l$ exists Zariski locally on $S$ by \cite[Proposition 5.2, 9.2]{PR2} (the reference only covers the case of type (C) and (AL), but the case of (AU) and (D) can be proved by the same argument);
    \item $\alpha_1$ is the natural forgetful map;
    \item $\alpha_2$ sends $(\CF_{\Lambda_{i,j}}^l,\varphi^l_{\Lambda_{i,j}})\in \wt{\RM}^\nspl_\CL$ to $(\CF^1_{\Lambda_{i,j}},\varphi^l_{\Lambda_{i,j}}(\CF_{\Lambda_{i,j}}^l/\CF_{\Lambda_{i,j}}^{l-1}))_{2\leq l\leq e_i}$.      
\end{itemize} 
%
%
Set \begin{flalign*}
	\CH\coloneqq \prod_{[\sigma]\in\Sigma/\sim_*,l\geq 2}\CG_{[\sigma]}.
\end{flalign*}
Then $\CH$ acts on $\wt{\RM}_\CL^\naive$ via \begin{flalign*}
	h\cdot ((\CF_{\Lambda_{i,j}}^l)_{1\leq l\leq e_i}, (\varphi^l_{\Lambda_{i,j}})_{2\leq l\leq e_i}) = ((\CF_{\Lambda_{i,j}}^l)_{1\leq l\leq e_i}, (h\cdot\varphi^l_{\Lambda_{i,j}})_{2\leq l\leq e_i}).
\end{flalign*}
By \cite[pp. 55]{shen2025f}, $\alpha_1$ is an $\CH$-torsor and $\alpha_2$ is a smooth and $\CH$-equivariant morphism. Recall that $\CG$ is the parahoric group scheme attached to $\CL$. Define the following actions of $\CG$: \begin{itemize}
	\item $\CG$ acts on $\RM^\nspl_\CL$ via \begin{flalign*}
		g\cdot(\CF^l_{\Lambda_{i,j}})\coloneqq (g\cdot\CF^l_{\Lambda_{i,j}}).
	\end{flalign*} Note that $\CG$-action on $\Lambda$ induces an action on $\Lambda_{i,j}$. 
	\item $\CG$ acts on $\wt{\RM}^\nspl_\CL$ via \begin{flalign*}
		g\cdot((\CF_{\Lambda_{i,j}}^l)_{1\leq l\leq e_i}, (\varphi^l_{\Lambda_{i,j}})_{2\leq l\leq e_i}) = ((g\cdot \CF_{\Lambda_{i,j}}^l)_{1\leq l\leq e_i}, (g\cdot\varphi^l_{\Lambda_{i,j}}\cdot g\inverse)_{2\leq l\leq e_i}).
	\end{flalign*}
	\item $\CG$ acts on $\prod_{[\sigma]\in\Sigma/\sim_*} \RM^\naive(\CL_{[\sigma]},\mu_{[\sigma]})$ via the natural morphism \begin{flalign*}
		\CG\ra \prod_{[\sigma]\in\Sigma/\sim_*}\CG_{[\sigma]}.
	\end{flalign*}
\end{itemize}
It is straightforward to check that $\alpha_1$ and $\alpha_2$ are $\CG$-equivariant with respect to the above actions. In summary, we obtain the following. 


\begin{prop}\label{prop-splitting}
    There exists a diagram over $\CO_L$ 
\begin{equation}
\begin{tikzcd}[column sep={6em, between origins}, row sep={3em, between origins}]
&  \arrow[ld, "\alpha_1"']\wt{\RM}^\nspl_\CL\arrow[rd, "\alpha_2"]&\\
\RM^\nspl_\CL   &&  \hspace{2em} \prod_{[\sigma]\in\Sigma/\sim_*} \RM^\naive(\CL_{[\sigma]},\mu_{[\sigma]}), 
\end{tikzcd} \label{splitdiag}
\end{equation}
    such that $\alpha_1$ is an $\CH$-torsor and $\alpha_2$ is $\CH$-equivariant and smooth. In particular, $\RM^\nspl_\CL$ is \etale locally isomorphic to a product of naive local models $\RM^\naive(\CL_{[\sigma]},\mu_{[\sigma]})$. Furthermore, $\alpha_1$ and $\alpha_2$ are both $\CG$-equivariant. 
\end{prop}

\begin{defn}
    {Let the {\it splitting model} $\RM^\spl_\CL:=\RM^\spl({\CL,\mu,(\sigma_{i,j}^l)})$ (resp. {local model} $\RM^\loc(\CL_{[\sigma]},\mu_{[\sigma]})$) denote the schematic closure of the generic fiber of $\RM^\nspl_\calL$ (resp. $\RM^\naive(\CL_{[\sigma]},\mu_{[\sigma]})$).}    
\end{defn}
By construction, $\RM^\spl_\CL$ and $\RM^\loc(\CL_{[\sigma]},\mu_{[\sigma]})$ are flat over $\CO_L$. Furthermore, the diagram (\ref{splitdiag}) restricts to a diagram 
\begin{equation}
\begin{tikzcd}[column sep={6em, between origins}, row sep={3em, between origins}]
&  \arrow[ld, "\alpha_1"']\wt{\RM}^\spl_\CL\arrow[rd, "\alpha_2"]&\\
\RM^\spl_\CL   &&  \hspace{2em} \prod_{[\sigma]\in\Sigma/\sim_*} \RM^\loc(\CL_{[\sigma]},\mu_{[\sigma]}), 
\end{tikzcd} 
\end{equation}
where $\alpha_1$ is an $\CH$-torsor and $\alpha_2$ is $\CH$-equivariant.

\part{The Orthogonal case} \label{part-orth}
\section{Orthogonal local models}\label{orthogonal-local-models}

In this section, we recall the definitions and main propositions of the split and quasi-split (but non-split) even orthogonal local models. For the split case  we follow \cite{Yang25}, while for the quasi-split case we follow \cite{YZZ26}.

Let $F_0$ be a complete discretely valued field with ring of integers
$\mathcal{O}_{F_0}$, uniformizer $\pi_0$, and residue field $k:=k_{F_0}$ of characteristic
$p>2$. Let $F/F_0$ be a ramified quadratic extension with a uniformizer
$\pi\in F$ such that $\pi^2=-\pi_0.$ We put
\[
g=
\begin{cases}
n, & \text{in the split case},\\
n+1, & \text{in the quasi-split case}.
\end{cases}
\]

After replacing $F_0$ by a sufficiently large unramified extension, we shall work
with one of the following two symmetric spaces over $F_0$ (see \cite[\S2.1]{YZZ26} and \cite[\S 2.2]{I.Z}). In the split case, let $V=F_0^{2g}$ with ordered basis
$e_1,\dots,e_{2g}$, equipped with the split symmetric $F_0$-bilinear form
$\psi$ given by
\begin{equation}\label{eq split form orthogonal}
\psi(e_i,e_j)=\delta_{i,\,2g+1-j}.  
\end{equation}
In the quasi-split case, let $V=F_0^{2g}$ with ordered basis
\[
e_1,\dots,e_{2g-2},f_1,f_2,
\]
equipped with the symmetric $F_0$-bilinear form $\phi$ whose matrix with respect
to this basis is
\begin{equation}\label{eq quasi-split form}
\phi=
\begin{pmatrix}
H_{2g-2} & 0\\
0 & \begin{pmatrix}\pi_0&0\\0&1\end{pmatrix}
\end{pmatrix}.   
\end{equation}
We write $(\, ,\,)$ for the corresponding symmetric form, namely $\psi$ in the
split case and $\phi$ in the quasi-split case. Let
\[
G:=\GO(V,(\, ,\,))
\]
and denote by $G^\circ$ the identity component of $G$.

For $0\le i\le n$, define the standard lattices $\Lambda_i$ as follows. In the
split case we set
\[
\Lambda_i:=
\mathcal{O}_{F_0}
\langle
\pi_0^{-1}e_1,\dots,\pi_0^{-1}e_i,
e_{i+1},\dots,e_{2g}
\rangle.
\]
In the quasi-split case we set
\[
\Lambda_i:=
\mathcal{O}_{F_0}
\langle
\pi_0^{-1}e_1,\dots,\pi_0^{-1}e_i,
e_{i+1},\dots,e_{2g-2},
\pi_0^{-1}f_1,f_2
\rangle.
\]
For any lattice $\Lambda\subset V$, let $\Lambda^\vee:=\{x\in V\mid (x,\Lambda)\subset\mathcal{O}_{F_0}\}.$ For every non-empty subset $I\subset[0,n]$, define the associated standard
self-dual periodic lattice chain by $\Lambda_I:=\{\Lambda_\ell\}_{\ell\in 2n\mathbb{Z}\pm I}$,
where
\[
\Lambda_{-i}:=\Lambda_i^\vee
\qquad\text{and}\qquad
\Lambda_\ell:=\pi_0^{-d}\Lambda_{\pm i}
\quad\text{for }\ell=2nd\pm i.
\]
Any self-dual periodic lattice chain of $V$ is $G^\circ(F_0)$-conjugate to some
$\Lambda_I$; see \cite[\S1.5, Theorem~2.2]{Yang25} and
\cite[Theorem~2.7 and Corollary~2.9]{YZZ26}.

Let $\calG_I$ be the affine smooth group scheme of similitude automorphisms of the
lattice chain $\calL:=\Lambda_I$. Its neutral component $\calG_I^\circ$ is the corresponding
parahoric group scheme for $G^\circ$ in the sense of Bruhat--Tits; see
\cite[\S1.2, \S2]{Yang25} and \cite[\S 3]{YZZ26}. {Consider the minuscule cocharacter 
$\mu_\pm$ given by
\begin{flalign}
\mu_+\coloneqq (1^{(g)},0^{(g)})
\qquad\text{and}\qquad
\mu_-\coloneqq (1^{(g-1)},0,1,0^{(g-1)}).
\label{cochar}
\end{flalign}
}

\begin{defn}\label{naivelocalmodeldef}
{The \emph{naive local model} $\RM_I^{\naive}=\Mnaive(\Lambda_I, \mu)$} is the projective scheme over
$\mathcal{O}_{F_0}$ representing the functor sending an
$\mathcal{O}_{F_0}$-algebra $R$ to the set of $R$-modules $(\mathcal{F}_\Lambda)_{\Lambda\in\Lambda_I}$ such that:
\begin{itemize}
\item[(LM1)] for every $\Lambda\in\Lambda_I$, Zariski locally on $\Spec R$,
$\mathcal{F}_\Lambda$ is a direct summand of $\Lambda_R:=\Lambda\otimes_{\mathcal{O}_{F_0}}R$ of rank $g$;

\item[(LM2)] for every $\Lambda\in\Lambda_I$, one has $((\, ,\,)\otimes 1)(\mathcal{F}_\Lambda,\mathcal{F}_{\Lambda^\vee})=0$;

\item[(LM3)] for every inclusion $\Lambda\subset\Lambda'$ in $\Lambda_I$, the
natural map $\Lambda_R\to \Lambda'_R$ sends $\mathcal{F}_\Lambda$ to
$\mathcal{F}_{\Lambda'}$, and the isomorphism
$\Lambda_R\xrightarrow{\sim}(\pi_0\Lambda)_R$ identifies $\mathcal{F}_\Lambda$ with $\mathcal{F}_{\pi_0\Lambda}$.
\end{itemize}
\end{defn}

From now on, we work after base change to $\mathcal{O}_F$. Thus we write again
$\RM_I^{\naive}$ for $\RM_I^{\naive}\otimes_{\mathcal{O}_{F_0}}\mathcal{O}_F.$ The generic fiber of $\RM_I^{\naive}$ over $F$ is the union of the two connected
components $\OGr(g,V_F)_+\sqcup \OGr(g,V_F)_-$. These components are identified with the flag varieties
\[
\OGr(g,V_F)_\pm\simeq G_F^\circ/P_{\mu_\pm},
\]
where $P_{\mu_\pm}$ denotes the parabolic subgroup associated to the minuscule
cocharacter $\mu_\pm$.
In particular, in the split case the relative dimension is $n(n-1)/2$, while in
the quasi-split case it is $n(n+1)/2$.

In \cite[\S7]{PR}, Pappas--Rapoport defined
an involution operator $a$ on $\wedge^g_F V_F$ inducing a decomposition
\[
\wedge^g_F V_F=W_+\oplus W_-,
\]
where $W_\pm$ denotes the $\pm1$-eigenspace for $a$. For a lattice $\Lambda\subset V$,
set $\Lambda_{\mathcal{O}_F}:=\Lambda\otimes_{\mathcal{O}_{F_0}}\mathcal{O}_F$
and
\[
\left(\wedge^g_{\mathcal{O}_F}\Lambda_{\mathcal{O}_F}\right)_\pm
:=
\wedge^g_{\mathcal{O}_F}\Lambda_{\mathcal{O}_F}\cap W_\pm.
\]

\begin{defn}\label{defspinlocalmodel}
The \emph{spin local model} $\RM^\pm_I$ is the projective scheme over $\CO_F$ representing the functor sending each $\CO_F$-algebra $R$ to the set of $R$-modules $(\CF_\Lambda)_{\Lambda\in \Lambda_I}$ such that $(\CF_\Lambda)_{\Lambda\in \Lambda_I}\in \RM^\naive_I(R)$ and 
    \begin{itemize}
        \item[LM4$\pm$.] for any $\Lambda\in\Lambda_I$, Zariski locally on $\Spec R$,  the line $\bigwedge_R^g\mathcal{F}_\Lambda$ is contained in
    \[
    \operatorname{Im}\left[
\left(
\wedge^g_{\mathcal{O}_F}\Lambda_{\mathcal{O}_F}
\right)_\pm
\otimes_{\mathcal{O}_F}R
\longrightarrow
\wedge^g_R\Lambda_R
\right].
    \]
    \end{itemize}
\end{defn}

The generic fiber of $\RM_I^\pm$ is isomorphic to $\OGr(g,V_F)_\pm$. For
simplicity, set $\calG:=\calG_I$. The group scheme $\calG$ naturally acts on
$\RM_I^{\naive}$ and $\calG^\circ$ preserves $\RM_I^\pm$; see
\cite[\S3]{Yang25} and \cite[\S 3]{YZZ26}.

\begin{thm}\label{thm:spin-flat}
For every non-empty subset $I\subset[0,n]$, the spin local model $\RM_I^\pm$ is flat over $\mathcal{O}_F$ of relative dimension $\frac{g(g-1)}{2}$ and represents the
corresponding $v$-sheaf local model in the sense of Scholze--Weinstein \cite[\S 21.4]{scholze2020berkeley}. Moreover, $\RM_I^\pm$ is normal and Cohen--Macaulay with reduced special fiber.
\end{thm}
\begin{proof}
In the split case the result follows from \cite[Theorem 1.5]{Yang25} and  \cite[Proposition 3.6]{yang25TopFlat}. In the quasi-split case, we obtain the result from \cite[Theorem~1.3(1), Proposition~3.6]{YZZ26}.
\end{proof}

We now restrict to the (pseudo-)maximal parahoric case, i.e. $I=\{i\}$, and write the spin local model 
\[
\RM_i^\pm:=\RM_{\{i\}}^\pm.
\]

\begin{defn}
    Let $\iota: \Lambda_i\ra \pi\inverse\Lambda_{-i}=\Lambda_{2n-i}$ denote the natural inclusion map (and its base change). 
     For an integer $\ell$, denote by $\RM_i^\pm(\ell)\sset \RM^\pm_{i,k}$ the locus where $\iota(\CF_i)$ has rank $\ell$.
\end{defn}

\begin{thm}\label{thm:max-schubert}
Let $I=\{i\}\subset[0,n]$.
\begin{altenumerate}
\item If $i=0$ or $i=n$, then $\RM_i^\pm$ is smooth over $\mathcal{O}_F$. Its
relative dimension is $n(n-1)/2$ in the split case and $n(n+1)/2$ in the
quasi-split case.

\item The spin local model $\RM_i^\pm$ is isomorphic to
$\RM_{n-i}^\pm$.

\item Assume that we are in the split case and that $i\neq 0,n$. Then there exists
a stratification
\begin{equation}\label{SStratification}
\RM^\pm_{i,k}
=
\coprod_{\ell=\max\{0,2i-n\}}^{i}
\RM_i^\pm(\ell),
\end{equation}
where the top stratum $\RM_i^\pm(i)$ decomposes into exactly two Schubert cells
of dimension $n(n-1)/2$, and each lower stratum $\RM_i^\pm(\ell)$, for
$\max\{0,2i-n\}\le \ell<i$, is a single Schubert cell. Moreover,
$\RM_i^\pm(\ell')$ is contained in the Zariski closure
$\overline{\RM_i^\pm(\ell)}$ if and only if $\ell'\leq \ell$.

\item Assume that we are in the quasi-split case. Then there
exists a stratification
\begin{flalign}\label{QStratification}
\RM^\pm_{i,k}
=
\coprod_{\ell=\max\{0,2i-n\}}^i
\RM^\pm_i(\ell),
\end{flalign}
where each stratum $\RM^\pm_i(\ell)$ is a single Schubert cell and
$\RM^\pm_i(\ell')$ is contained in the Zariski closure
$\overline{\RM^\pm_i(\ell)}$ if and only if $\ell'\leq \ell$. Consequently,
$\RM^\pm_{i,k}$ is irreducible and contains precisely $\min\{i,n-i\}+1$ Schubert cells.
\end{altenumerate}
\end{thm}
\begin{proof}
For the split case see \cite[Proposition~2.8]{Yang25} and for the quasi-split case see \cite[Theorem~1.7, Remark~1.10, Theorem~2.7]{YZZ26}.
\end{proof}
By Theorem~\ref{thm:max-schubert}(2), we may assume from now
on that $0\le 2i\le n$. We now record the affine charts $U_i^\pm\subset \RM_i^\pm$ containing a point
$\ast$ in the minimal stratum of \eqref{SStratification} in the split case,
respectively of \eqref{QStratification} in the quasi-split case.

\begin{thm}\label{thm:affine-chart}
Assume that $I=\{i\}$ with $0\le 2i\le n$. The $\calG^\circ$-translates of
$U_i^\pm$ cover $\RM_i^\pm$. Moreover, the following explicit descriptions hold.
\begin{altenumerate}
\item Assume that we are in the split case. Then
\[
U_i^\pm
\simeq
\mathbb{A}^{\frac{(n-2i)(n+2i-1)}{2}}_{\mathcal{O}_F}
\times_{\mathcal{O}_F}
\Spec R_i^\pm,
\qquad
R_i^\pm=\mathcal{O}_F[X]/I_i^\pm,
\]
where $X$ is a $2i\times 2i$ matrix and
\[
I_i^\pm
:=
(X^tH_{2i}X+\pi_0H_{2i},\;
XH_{2i}X^t+\pi_0H_{2i})
+
J_i^\pm.
\]
Here $J_i^\pm$ is generated by
\[
[U:U'](X)
\mp
\sgn(\sigma_U)\sgn(\sigma_{U'})
[U^\perp:{U'}^\perp](X)
\]
and
\[
[V:V'](X)
\pm
\sgn(\sigma_V)\sgn(\sigma_{V'})
\pi_0
[V^\perp:{V'}^\perp](X),
\]
for all subsets $U,U'$ of $[1,2i]$ of cardinality $i$ and all subsets $V,V'$ of
$[1,2i]$ of cardinality $i+1$. The complements are taken with respect to the
involution $j\mapsto 2i+1-j$.

\item Assume that we are in the quasi-split case. Then
\[
U_i^\pm
\simeq
\mathbb{A}_{\mathcal{O}_F}^{\frac{(n-2i)(n+2i+1)}{2}}
\times_{\mathcal{O}_F}
\Spec R_i,
\qquad
R_i=\mathcal{O}_F[X]/I_i^\pm,
\]
where $X$ is a $(2i+1)\times(2i+1)$ matrix and
\[
I_i^\pm
:=
\bigl(
X^tH_{2i+1}X+\pi_0H_{2i+1},\;
XH_{2i+1}X^t+\pi_0H_{2i+1}
\bigr)
+
J_i^\pm.
\]
Here $J_i^\pm$ is generated by the relations
\[
[U^\perp:{U'}^\perp](X)
\mp
(-1)^i
\sgn(\sigma_U)\sgn(\sigma_{U'})
\pi
[U:U'](X),
\]
for all subsets $U,U'\subset[1,2i+1]$ of cardinality $i$. The complements are
taken with respect to the involution $j\mapsto 2i+2-j$.
\end{altenumerate}
\end{thm}
\begin{proof}
In the quasi-split case the result follows from
\cite[Theorem~1.8, Proposition~5.7, \S6]{YZZ26}. For the split case, first set 
\[
\calU_i' =\mathbb{A}^{\frac{(n-2i)(n+2i-1)}{2}}_{\mathcal{O}_F}
\times_{\mathcal{O}_F}
\Spec \frac{\mathcal{O}_F[X]}{(X^tH_{2i}X+\pi_0H_{2i},\;
XH_{2i}X^t+\pi_0H_{2i})
+
J_i^\pm}
\] 
where $J_i^\pm$ is as in the statement. The equations in $J_i^\pm$ are obtained from the
spin condition by a similar calculation as in the quasi-split case. Thus, we get a closed immersion $
U_i^\pm\hookrightarrow \calU_i'$. The generic fibers agree, and by \cite[Theorem~1.10]{Yang25} the special fibers also agree. Hence, this closed immersion is an isomorphism. 
\end{proof}
For $0\le \ell<i$, let
\[
S_\ell:=\overline{\RM_i^\pm(\ell)}\subset \RM_{i,k}^\pm
\]
be the corresponding Schubert variety.

\begin{prop}\label{prop:special-fiber-schubert-centers}
Assume that $I=\{i\}$ with $0\le 2i\le n$.

\begin{altenumerate}
\item Assume that we are in the split case. Then
\[
U_{i,k}^\pm
\simeq
\mathbb{A}_k^{\frac{(n-2i)(n+2i-1)}{2}}
\times_k
\Spec R_{i,k}^\pm,
\]
where $\Spec R^\pm_{i,k}$ is reduced and
\[
R_{i,k}^\pm:=
\frac{k[X]}
{
\bigl(
XH_{2i}X^t,\;
X^tH_{2i}X,\;
\wedge^{i+1}X
\bigr)+I_\mp
}.
\]
Here $I_\mp$ is generated by
\[
[U:U'](X)
\mp
\sgn(\sigma_U)\sgn(\sigma_{U'})
[U^\perp:{U'}^\perp](X)
\]
for all subsets $U,U'\subset[1,2i]$ of cardinality $i$. Moreover, for every $0\le \ell<i$ one has
\[
U_i^\pm\cap S_\ell
\simeq
\Spec\left(R_i^\pm/(\pi,\wedge^{\ell+1}X)\right)
\times_k
\mathbb{A}_k^{\frac{(n-2i)(n+2i-1)}{2}}.
\]
In particular, $U_i^\pm\cap S_\ell$ is cut out on the $\Spec R_i^\pm$-factor by
the ideal $I_\ell=(\pi,\wedge^{\ell+1}X)$ and
\[
\Spec \frac{R_i^\pm}{(\pi,\wedge^{\ell+1}X)}
=\Spec
\frac{k[X]}
{
\bigl(
X^tH_{2i}X,\;
XH_{2i}X^t,\;
\wedge^{\ell+1}X
\bigr)
}
\]
is integral of dimension $
\ell(4i-2\ell-1)$.

\item Assume that we are in the quasi-split case. Then
\[
U_{i,k}^\pm
\simeq
\mathbb{A}_k^{\frac{(n-2i)(n+2i+1)}{2}}
\times_k
\Spec R_{i,k},
\]
where $\Spec R_{i,k}$ is integral and
\[
R_{i,k}:=
\frac{k[X]}
{
\bigl(
X^tH_{2i+1}X,\;
XH_{2i+1}X^t,\;
\wedge^{i+1}X
\bigr)
}.
\]

Moreover, for every $0\le \ell<i$ one has
\[
U_i^\pm\cap S_\ell
\simeq
\Spec\left(R_i/(\pi,\wedge^{\ell+1}X)\right)
\times_k
\mathbb{A}_k^{\frac{(n-2i)(n+2i+1)}{2}}.
\]
In particular, $U_i^\pm\cap S_\ell$ is cut out on the $\Spec R_i$-factor by
the ideal $I_\ell=(\pi,\wedge^{\ell+1}X)$ and
\[
\Spec \frac{R_i}{(\pi,\wedge^{\ell+1}X)}
= \Spec
\frac{k[X]}
{
\bigl(
X^tH_{2i+1}X,\;
XH_{2i+1}X^t,\;
\wedge^{\ell+1}X
\bigr)
}
\]
is integral of dimension $
\ell(4i-2\ell+1)$.
\end{altenumerate}
\end{prop}
\begin{proof}
(1) The special fiber description is given by
\cite[Theorem 1.10]{Yang25}. The
description of $U_i^\pm\cap S_\ell$ is obtained from \cite[Proposition~3.7]{Yang25}, the dimension formula from \cite[Proposition~4.44]{Yang25}.

(2) The special-fiber description follows from \cite[Theorem~1.9]{YZZ26}. By \cite[Lemma~6.2]{YZZ26} and
\cite[Theorem~6.5]{YZZ26}, respectively, $\Spec R_{i,k}$ is irreducible and reduced, hence integral. The description and integrality of $U_i^\pm\cap S_\ell$ follow by the proof of \cite[Proposition~5.4]{YZZ26}. The dimension formula is obtained from \cite[Proposition~4.44]{Yang25}.
\end{proof}

\section{Semi-stable resolution for quasi-split orthogonal local models}  \label{sec-qsorth}

We assume throughout this section that we are in the quasi-split case. Our goal is to construct a semi-stable resolution of the quasi-split orthogonal local model in the maximal parahoric case. By Theorem~\ref{thm:max-schubert}(2), we may assume that $I=\{i\},\, 1\le i\le \left\lfloor\frac n2\right\rfloor$. The case $i=0$ is excluded, since $\RM_0^\pm$ is already smooth over
$\mathcal{O}_F$ by Theorem~\ref{thm:max-schubert}(1). Thus in what follows we
assume $1\le i\le \left\lfloor\frac n2\right\rfloor$. 

\subsection{Reduction to the affine local statement}
We keep the notation of \S \ref{orthogonal-local-models} for
the Schubert varieties $S_\ell$ with $0\le \ell\le i$ and set
\[
S_\ell:=\overline{\RM_i^\pm(\ell)}\subset \RM^\pm_{i,k}.
\]
By Theorem~\ref{thm:max-schubert}(4), the special fiber has a chain of Schubert
varieties
\[
S_0\subset S_1\subset\cdots\subset S_i=\RM^\pm_{i,k}.
\]

\begin{thm}\label{thm:qs-semistable-resolution}
Assume that $I=\{i\}$ with $1\le i\le \lfloor n/2\rfloor$. Let
\[
\RM_i^\pm=Z_0
\xleftarrow{b_1}
Z_1
\xleftarrow{b_2}
\cdots
\xleftarrow{b_i}
Z_i
\]
be the sequence of blow-ups such that $b_1$ is the blow-up of
$Z_0=\RM_i^\pm$ along $S_0$, and for $2\le \ell\le i$, the morphism
\[
b_\ell:Z_\ell\longrightarrow Z_{\ell-1}
\]
is the blow-up along the strict transform of $S_{\ell-1}$ in $Z_{\ell-1}$.
Then $Z_i$ is semi-stable over $\mathcal{O}_F$.
\end{thm}

\begin{remark}
    This generalizes the result in \cite[\S 7]{YZZ26}, where we resolve the singularities for $I=\cbra{1}$.
\end{remark}

We recall the following definition.
\begin{defn}[{\cite[Definition 1.1, Proposition 1.3]{Ha}}]
    An integral scheme $Y$, flat and separated of finite type over a discrete valuation ring $\CO$ with a uniformizer $\pi$, is \dfn{(strictly) semi-stable} if the generic fiber $Y_\eta$ is smooth, and every closed point $x$ in the special fiber $Y_0$ admits an open neighborhood which for some $s\in \BN$ is smooth over the scheme $\Spec \CO[x_1,\ldots,x_s]/(x_1\cdots x_s-\pi)$. 
\end{defn}

By Theorem~\ref{thm:affine-chart}, there exists an open affine subscheme
$U_i^\pm\subset \RM_i^\pm$ containing a point $\ast\in S_0$ such that
\[
U_i^\pm
\simeq
\Spec R_i
\times_{\mathcal{O}_F}
\mathbb{A}_{\mathcal{O}_F}^{\frac{(n-2i)(n+2i+1)}{2}},
\qquad
R_i=\mathcal{O}_F[X]/I_i^\pm,
\]
where $X$ is a $(2i+1)\times(2i+1)$ matrix, and the $\calG_i^\circ$-translates of
$U_i^\pm$ cover $\RM_i^\pm$.

By Proposition~\ref{prop:special-fiber-schubert-centers}, for every
$0\le \ell<i$, the intersection $U_i^\pm\cap S_\ell$ is cut out on the
$\Spec R_i$-factor by the ideal
\[
I_\ell=(\pi,\wedge^{\ell+1}X).
\]
Therefore, to prove Theorem~\ref{thm:qs-semistable-resolution}, it suffices to
establish the following affine local statement.

\begin{prop}\label{prop:qs-affine-semistable}
Let
\[
\calX_0:=\Spec R_i,
\qquad
I_\ell:=(\pi,\wedge^{\ell+1}X)\subset R_i
\qquad
(0\le \ell<i).
\]
Consider the sequence of blow-ups
\begin{equation}\label{quasisplit.seq. of blow-ups}
\calX_0
\xleftarrow{c_1}
\calX_1
\xleftarrow{c_2}
\cdots
\xleftarrow{c_i}
\calX_i,
\end{equation}
where $c_1$ is the blow-up along $V(I_0)$, and for $2\le \ell\le i$, the morphism
\[
c_\ell:\calX_\ell\longrightarrow \calX_{\ell-1}
\]
is the blow-up along the strict transform of $V(I_{\ell-1})$ in $\calX_{\ell-1}$.
Then $\calX_i$ is semi-stable over $\mathcal{O}_F$.
\end{prop}

To prove Proposition~\ref{prop:qs-affine-semistable}, we first introduce in the
next subsection a family of auxiliary affine schemes and establish their basic
properties. In the subsequent subsection, we analyze the blow-ups of these schemes.
The proof of Proposition~\ref{prop:qs-affine-semistable} is then completed in \S \ref{ProofProp3.3}.

\subsection{The affine schemes $\calU(d,\ell)$}\label{section:quasi-split-U(d,l)}
For $0\leq d,\ell\leq i$, let $\calU(d,\ell)^\pm:=\Spec \calR(d,\ell)^\pm,$ where
\[
\calR(d,\ell)^\pm :=
\frac{\calO_F[X,T_0,T_1,\ldots,T_\ell]}
{(X^tH_{2d+1}X-T_0^2H_{2d+1},
XH_{2d+1}X^t-T_0^2H_{2d+1},
\pi-T_0T_1\cdots T_\ell)+J_d^\pm}.
\]
Here $X$ is the square matrix of size $2d+1$, and $J_d^\pm$ is generated by the relations
\[
[U^\bot:{U'}^\bot](X)
\mp
(-1)^d \sgn(\sigma_U)\sgn(\sigma_{U'}) T_0 [U:U'](X).
\]
(We will omit the lower indices
of $H_{2d+1}$ if there is no confusion.) When the sign is fixed, we simply write $\calU(d,\ell)$ and $\calR(d,\ell)$. Observe that when $d=i,\ell=0$, we have 
\[
\calU(i,0)\simeq \Spec R_i;
\]
see Theorem \ref{thm:affine-chart} for the explicit definition of $R_i$.

\begin{prop}\label{prop flat R(d,l)}
The ring $\mathcal{R}(d, \ell)$ is flat over $\mathcal{O}_F$ and integral of (total) dimension $d(2d + 1) + \ell + 1$.
\end{prop}
{
\begin{remark}
    The above proposition is the main technical result of this subsection. The integrality of $\CR(d,\ell)$ plays an essential role in our inductive description of the local equations of the blow-ups; see the proof of Lemma \ref{lemma chart y_{i,j}}.
\end{remark}
}

\begin{proof}
For simplicity, set $\mathcal{R} := \mathcal{R}(d, \ell)$. Then $\mathcal{R}_k \coloneqq \mathcal{R} \otimes k$ is isomorphic to
\[ \frac{k[X, T_0, T_1, \dots, T_ \ell]}{(XH X^t - T_0^2 H,\; X^t H X - T_0^2 H, \;T_0 T_1 \cdots T_ \ell)+ J_d^\pm}. \]
Define 
\[ \mathscr{I} := (XH X^t - T_0^2 H, \; X^t H X - T_0^2 H,\; T_0 T_1 \cdots T_ \ell) + J_d^\pm
\]
and $\mathscr{I}_{r} := (XH X^t - T_0^2 H, \; X^t H X - T_0^2 H, \; T_r)+ J_d^\pm$ for $ r \in \{0,\dots,  \ell\}$. 
\[
\textbf{Claim:}\,\, \mathscr{I}_{r} \text{ is a prime ideal in } \mathcal{R}_k \text{ for } r \in \{0,\dots,  \ell\} \text{ and } \mathscr{I}= \cap_{r=0}^ \ell \mathscr{I}_{r}.
\]
\textit{Proof of the claim:} Note that
\[
\mathcal{R}_k / \mathscr{I}_0 = \frac{k[X, T_1, \dots, T_ \ell]}{(X^t H X, X H X^t, \wedge^{d+1} X)} = \frac{k[X]}{(X^t H X, X H X^t, \wedge^{d+1} X)} \otimes_k k[T_1, \dots, T_ \ell].
\]
Since the ring $k[X]/(X^t H X, X H X^t, \wedge^{d+1} X)$ is integral (see Proposition \ref{prop:special-fiber-schubert-centers} (2)), we get that $\mathcal{R}_k / \mathscr{I}_0$ is integral, and so, $\mathscr{I}_0$ is prime. Moreover, since $\mathcal{R}_k / \mathscr{I}_{r} \simeq \mathcal{R}_k / \mathscr{I}_1$ for any $ 1\leq r \leq \ell$, it remains to show that $\mathcal{R}_k / \mathscr{I}_1$ is integral which follows from Lemma \ref{lemmaRingA}. 

Next, we show that $\mathscr{I} = \cap_{r=0}^ \ell \mathscr{I}_{r}$. Clearly, $\mathscr{I} \subset \cap_{r=0}^ \ell \mathscr{I}_{r}$. Let $x \in \cap_{r=0}^ \ell \mathscr{I}_{r}$. Write $\mathscr{J} = (XH X^t - T_0^2 H,\; X^t H X - T_0^2 H)+J_d^\pm$. Since $x \in \mathscr{I}_0$, we can write $x = g_1 + T_0 f_1$ for some $g_1 \in \mathscr{J}$ and $f_1 \in k[X, T_0, \dots, T_\ell]$. Then $x \in \mathscr{I}_1$ implies that $T_0 f_1 \in \mathscr{I}_1$. Since $\mathscr{I}_1$ is prime and $T_0 \notin \mathscr{I}_1$, we obtain $f_1 \in \mathscr{I}_1$. Then we can write $x = g_2 + T_0 T_1 f_2$ for some $g_2 \in \mathscr{J}$ and $f_2 \in k[X, T_0, \dots, T_\ell]$. Similarly, we obtain $f_2 \in \mathscr{I}_2$. Iterating this process, we may write $x = g_\ell + T_0 T_1 \cdots T_\ell f_\ell$. Hence, $x \in \mathscr{I}$. This finishes the proof of the claim. 

From the above claim we deduce that the special fiber $\Spec  \mathcal{R}_k$ is reduced and has $\ell + 1$ irreducible components of dimension  $d(2d + 1) + \ell$ defined by $\mathscr{I}_{r}$ for $0 \leq r\leq \ell$.

To prove the proposition, it remains to show that $\mathcal{R}$ is topologically flat over $\mathcal{O}_F$. Note that $T_0$ is invertible over the generic fiber $\mathcal{R} \otimes_{\mathcal{O}_F} F$. Then a similar argument as in Lemma \ref{lemmaRingA} shows that $\mathcal{R} \otimes_{\mathcal{O}_F} F$ is integral and smooth of dimension $d(2d + 1) + \ell$ over $F$. In particular, the dimension of $\mathcal{R} \otimes_{\mathcal{O}_F} F$ equals the dimension of each component of the special fiber $\Spec  \mathcal{R}_k$. Let $P_0 \in \text{Spec } \mathcal{R}_k / \mathscr{I}_0$ be a closed point given by
\[ X = \begin{pmatrix} 1 & \\ & 0 \cdot I_{2d} \end{pmatrix}, \ T_1 = T_2 = \dots = T_\ell = 1. \]
Then $P_0$ lifts to an $\mathcal{O}_F$-point of $\text{Spec } \mathcal{R}$ given by
\[ \widetilde{X} = \begin{pmatrix} 1 & & \\ & \pi I_{2d-1} & \\ & & \pi^2 \end{pmatrix}, \ T_0 = \pi, \ T_1 = \dots = T_\ell = 1. \]
Let $P_{r} \in \text{Spec } \mathcal{R}_k / \mathscr{I}_{r}$ be a closed point given by
\[ X = I_{2d+1}, \ T_0 = T_1 = \dots = T_{r-1} = T_{r+1} = \dots = T_\ell = 1, \ T_{r} = 0. \]
Then $P_{r}$ lifts to an $\mathcal{O}_F$-point of $\text{Spec } \mathcal{R}$ given by
\[ \widetilde{X} = I_{2d+1}, \ T_0 = T_1 = \dots = T_{r-1} = T_{r+1} = \dots = T_\ell = 1, \ T_{r} = \pi. \]
Note that each $P_{r}$ lies in only one irreducible component of $\text{Spec } \mathcal{R}_k$. By \cite[Lemma 2]{Goertz05}, $\mathcal{R}$ is topologically flat over $\mathcal{O}_F$. Then we obtain that $\mathcal{R}$ is $\mathcal{O}_F$-flat. Since the generic fiber is integral of dimension $d(2d+1)+\ell$, we have that $\mathcal{R}$ is integral of dimension $d(2d+1)+\ell+1$. 
\end{proof}
\begin{lemma}\label{lemmaRingA}
The ring $A := k[X, T_0]/\left((XH X^t - T_0^2 H, X^t H X - T_0^2 H)+J_d^\pm\right)$ is integral of dimension $d(2d + 1) + 1$.
\end{lemma}
\begin{proof}
We first show that $A$ is flat over $k[T_0]$. Over the localization $A[T_0^{-1}]$, set $Y=T_0^{-1}X$. Then 
\[ A[T_0^{-1}] \simeq \frac{k[Y]}{(Y H Y^t - H, \; Y^t H Y - H)+\mathcal{J}_d^{\pm}} \otimes_k k[T_0, T_0^{-1}], \]
where $\mathcal{J}_d^{\pm}$ denotes the ideal generated by $[U^\bot:{U'}^\bot](Y) - (-1)^d \text{sgn}(\sigma_U) \text{sgn}(\sigma_{U'}) [U : U'](Y)$ and $U, U'$ run through subsets of $[1, 2d + 1]$ of cardinality $d$. In fact, we have that
\[ \frac{k[Y]}{(Y H Y^t - H, \; Y^t H Y - H)+\mathcal{J}_d^{\pm}} \simeq \frac{k[Y]}{(Y H Y^t - H,\; Y^t H Y - H, \; \text{det}(Y) - (-1)^d)}, \]
which is the coordinate ring of a connected component of the (split) orthogonal group over $k$ (see \cite[Lemma 4.3 (3)]{Yang25}).  
It follows that $A[T_0^{-1}]$ is integral and smooth, and so flat, over $k[T_0, T_0^{-1}]$ of (relative) dimension $d(2d + 1)$. To prove flatness of $A$ over $k[T_0]$, it remains to show that $A_{(T_0)}$ is flat over the localization $k[T_0]_{(T_0)}$ of $k[T_0]$ at the prime ideal $(T_0)$. It suffices to show that the special fiber $A_{(T_0)} \otimes_{k[T_0]_{(T_0)}} k$ is reduced and $A_{(T_0)}$ is topologically flat over $k[T_0]_{(T_0)}$. Note that the special fiber $A_{(T_0)} \otimes_{k[T_0]_{(T_0)}} k$ is isomorphic to $k[X]/(XH X^t, X^t H X, \wedge^{d+1} X)$, which is reduced (and irreducible) of dimension $d(2d + 1)$. Consider a point in the special fiber given by the diagonal matrix
\[ X = \begin{pmatrix} 1 & \\ & 0 \cdot I_{2d} \end{pmatrix}. \]
Then it lifts to a closed point in the generic fiber given by
\[ \widetilde{X} = \begin{pmatrix} 1 & & \\ & T_0 \cdot I_{2d-1} & \\ & & T_0^2 \end{pmatrix}. \]
By \cite[Lemma 2]{Goertz05}, $A_{(T_0)}$ is topologically flat over $k[T_0]_{(T_0)}$. Thus, we prove that $A$ is flat over $k[T_0]$.

By the flatness of $A$, we have $A \hookrightarrow A \otimes_{k[T_0]} k(T_0)$. By the analysis in the previous paragraph, we obtain that $A \otimes_{k[T_0]} k(T_0)$ is integral, so is $A$.
\end{proof}

\subsection{The Blow-up of $\calU(d,\ell)$}\label{section blow-ups}
In this subsection we consider the blow-up $\Bl_{\calI_0}(\calU(d,\ell))$ of $\calU(d,\ell)$ along the ideal $\calI_0= (X,T_0)=((x_{i,j})_{1\leq i,j\leq 2d+1},T_0)$. Let $\tilde{\calR}(d,\ell)$ be the graded $\calR(d,\ell)$-algebra $\oplus_{t\geq 0} \calI_0^t$ 
with $\calI_0^0=\calR(d,\ell)$. Then $\Bl_{\calI_0}(\calU(d,\ell))
=
\Proj \widetilde{\calR}(d,\ell)$ is covered by the $(2d+1)^2+1$ affine charts $D_+(y_{i,j})$ and $D_+(\alpha)$,
where $y_{i,j}$ and $\alpha$  denote the homogeneous degree-one elements corresponding to 
$x_{i,j}$ and $T_0$, respectively. In $\widetilde{\calR}(d,\ell)$ we have the
relations
\begin{equation}\label{eq wedge}
x_{i,j}y_{r,s}-x_{r,s}y_{i,j}=0,
\qquad
x_{i,j}\alpha-T_0y_{i,j}=0
\end{equation}
for all $1\leq i,j,r,s\leq 2d+1$.

\begin{lemma}\label{lem:relevant-charts}
The blow-up $\Bl_{\calI_0}(\calU(d,\ell))$ is covered by the affine charts
$D_+(y_{i_0,j_0})$ with $i_0,j_0\neq d+1$, together with the chart $D_+(\alpha)$.
\end{lemma}

\begin{proof}

We first examine the affine chart $D_+(y_{i_0, j_0})$ with $i_0=d+1, j_0\neq d+1$. Set $t^*=2d+2-t$ for $1\leq t\leq 2d+1$. In this chart, we have $i_0=i_0^*$, and $j_0\neq j_0^*$. Let $Y=(y_{i,j})$ be the square matrix with $y_{i_0,j_0}=1$. By (\ref{eq wedge}), we have $X=x_{i_0, j_0}Y$ and $T_0=x_{i_0,j_0}\alpha$. It is not hard to see that relations $X^tH_{2d+1}X-T_0^2H_{2d+1},\; XH_{2d+1}X^t-T_0^2H_{2d+1},\ \pi-T_0T_1\cdots T_\ell$ translate to 
\[
Y^tH_{2d+1}Y-\alpha^2H_{2d+1},\quad YH_{2d+1}Y^t-\alpha^2H_{2d+1},\quad \pi-\alpha x_{i_0,j_0}T_1\cdots T_\ell,
\]
and the relations in $J_d^\pm$ are equivalent to 
\[
[U^\bot:{U'}^\bot](Y)
\mp
(-1)^d \sgn(\sigma_U)\sgn(\sigma_{U'}) \alpha [U:U'](Y),
\]
for all subsets $U,U'\subset [1,2d+1]$ of cardinality $d$. 

Let $r_s$ (resp. $c_t$) denote the $s$-th row (resp. $t$-th column) in the matrix $Y$. The identity $Y^tH_{2d+1}Y=\alpha^2H_{2d+1}$ (resp. $YH_{2d+1}Y^t=\alpha^2H_{2d+1}$) is equivalent to $(r_i, r_j)=\alpha^2\cdot \delta_{i,j^*}$ (resp. $(c_i, c_j)=\alpha^2\cdot \delta_{i,j^*}$), where $(\ , \ ): \calO_F^{2d+1}\times \calO_F^{2d+1}\rightarrow \calO_F$ is the symmetric non-degenerate bilinear form associated to the matrix $H_{2d+1}$. The relation $(c_{j_0}, c_{j_0})=0$ gives
\[
1+\sum\limits_{s=1, s\neq d+1}^{2d+1}y_{s, j_0}y_{s^*, j_0}=0,
\]
which forces at least one of the variables $y_{1,j_0}, \cdots, y_{d, j_0}$ to be invertible. Hence 
\[
D_+(y_{d+1, j_0})\subset \bigcup_{s=1}^{d} D_+(y_{s, j_0}).
\]

Similarly, for the affine chart $D_+(y_{i_0, j_0})$ with $i_0\neq d+1, j_0=d+1$, the relation $(r_{i_0}, r_{i_0})=0$ yields $1+\sum_{t=1, t\neq d+1}^{2d+1}y_{i_0, t}y_{i_0, t^*}=0$, so $D_+(y_{i_0, d+1})\subset \cup_{t=1}^{d} D_+(y_{i_0, t})$. 

Finally, for the affine chart $D_+(y_{d+1, d+1})$, the identities $(r_{d+1}, r_{d+1})=\alpha^2$, $(c_{d+1}, c_{d+1})=\alpha^2$ imply:
\[
1+2\sum\limits_{s=1}^d y_{s,d+1}y_{s^*,d+1}=\alpha^2,\quad 
1+2\sum\limits_{t=1}^d y_{d+1,t}y_{d+1,t^*}=\alpha^2.
\]
Therefore we conclude 
\[
D_+(y_{d+1, d+1})\subset (\bigcup_{s=1}^d D_+(y_{s, d+1}))\cup D_+(\alpha)\subset (\bigcup_{s=1}^d\bigcup_{t=1}^d D_+(y_{s, t}))\cup D_+(\alpha).
\]
This completes the proof of the lemma.
\end{proof}


By Lemma~\ref{lem:relevant-charts}, it remains to analyze the charts $D_+(y_{i_0,j_0})$ with
$i_0\neq i_0^\ast$ and $j_0\neq j_0^\ast$, and the chart $D_+(\alpha)$.

\begin{lemma}\label{lemma chart y_{i,j}}
The affine chart $D_+(y_{i_0,j_0})$ with $i_0\neq i_0^\ast$ and $j_0\neq j_0^\ast$
is isomorphic to
\[
\calU(d-1,\ell+1)^{\pm\epsilon}\times \mathbb{A}_{\calO_F}^{4d-2},
\]
where $\epsilon=(-1)^{i_0+j_0+1}$. In particular, for $\ell=0$, the affine chart $D_+(y_{i_0,j_0})$ is isomorphic to
\begin{equation}\label{eq Z}
\begin{split}
&\Spec \frac{\calO_F[Z, T_0, T_1]}{(Z^tHZ-T_0^2H, ZHZ^t-T_0^2 H, \pi-T_0T_1)+J_{d-1}^{\pm\epsilon}}\times \mathbb{A}_{\calO_F}^{4d-2}.
\end{split}
\end{equation}
\end{lemma}
\begin{proof}
Consider the affine chart $D_+(y_{i_0,j_0})$. Assume $y_{i_0, j_0}=1$ with $i_0, j_0\neq d+1$. Then we obtain $i_0\neq i_0^*, j_0\neq j_0^*$.  Evaluating the relations  $(r_{i_0}, r_{i_0})=0$ (resp. $(c_{j_0}, c_{j_0})=0$) gives
\begin{equation}\label{eq 312}
\begin{array}{ll}
  y_{i_0,j_0^*}=-\frac12\sum\limits_{t=1, t\neq j_0, j_0^*}^{2d+1} y_{i_0,t}y_{i_0,t^*},   &  
  y_{i_0^*,j_0}=-\frac12\sum\limits_{s=1, s\neq i_0, i_0^*}^{2d+1} y_{s,j_0}y_{s^*,j_0}.\
\end{array}
\end{equation}
Consider the relations $(r_{i_0}, r_s)=0$ (resp. $(c_{j_0}, c_t)=0$) for $s\neq i_0, i_0^*$, $t\neq j_0, j_0^*$, we obtain
\begin{equation}\label{eq 313}
\begin{array}{ll}
 y_{s,j_0^*}=-\sum\limits_{t=1, t\neq j_0, j_0^*}^{2d+1} y_{i_0,t}y_{s,t^*}+\frac12 y_{s,j_0}\sum\limits_{t=1, t\neq j_0, j_0^*}^{2d+1}y_{i_0,t}y_{i_0,t^*},   \\ 
  y_{i_0^*,t}=-\sum\limits_{s=1, s\neq i_0, i_0^*}^{2d+1} y_{s,j_0}y_{s^*,t}+\frac12 y_{i_0,t}\sum\limits_{s=1, s\neq i_0, i_0^*}^{2d+1}y_{s, j_0}y_{s^*,j_0}.
\end{array}
\end{equation}
From (\ref{eq 313}), we derive the identity $\sum_{t\neq j_0, j_0^*}y_{i_0,t}y_{i_0^*,t^*}=\sum_{s\neq i_0, i_0^*}y_{s,j_0}y_{s^*,j_0^*}$ . Applying either $(r_{i_0}, r_{i_0^*})=\alpha^2$ or $(c_{j_0}, c_{j_0^*})=\alpha^2$ then yields
\begin{equation}\label{eq 314}
y_{i_0^*, j_0^*}=\alpha^2-\frac34\sum\limits_{s\neq i_0, i_0^*}\sum\limits_{t\neq j_0, j_0^*}y_{i_0,t}y_{i_0,t^*}y_{s,j_0}y_{s^*,j_0}+\sum\limits_{s\neq i_0, i_0^*}\sum\limits_{t\neq j_0, j_0^*}y_{i_0,t}y_{s,j_0}y_{s^*,t^*}.
\end{equation}
Equations (\ref{eq 312})--(\ref{eq 314}) show that the variables $y_{i_0^*,t}, y_{s, j_0^*}$ for $1\leq s,t\leq 2d+1$ are uniquely determined by $\alpha$ and the remaining entries of $Y$.

It remains to verify the bilinear form identities $(r_i, r_j)=\alpha^2\cdot \delta_{i,j^*}$ for $i, j\neq i_0$ (resp. $(c_i, c_j)=\alpha^2\cdot \delta_{i,j^*}$ for $i, j\neq j_0$). Set 
\[
z_{s,t}:=y_{s,t}-y_{i_0,t}y_{s,j_0},\quad \text{for}~s\neq i_0, i_0^*, t\neq j_0, j_0^*.
\]
Let $Z=(z_{s,t})_{s\neq i_0, i_0^*,~ t\neq j_0, j_0^*}$ be the square matrix of size $2d-1$. Consider the relations $(r_{s}, r_{t})=\alpha^2\cdot \delta_{s,t^*}$ for $s, t\neq i_0, i_0^*$. By using (\ref{eq 313}), we obtain
\begin{flalign*}
(r_s, r_t)&=y_{t,j_0}y_{s,j_0^*}+y_{s,j_0}y_{t,j_0^*}+\sum\limits_{k\neq j_0, j_0^*} y_{s,k}y_{t,k^*}\\
&=\sum\limits_{k\neq j_0, j_0^*}(y_{s,k}-y_{i_0,k}y_{s,j_0})(y_{t,k^*}-y_{i_0,k^*}y_{t, j_0})\\
&=\sum\limits_{k\neq j_0, j_0^*}z_{s,k}z_{t,k^*}=\alpha^2 \delta_{s,t^*}
\end{flalign*}
This translates to $Z^tH_{2d-1}Z=\alpha^2H_{2d-1}$. Similarly, for the relation $(c_s, c_t)=\alpha^2\cdot \delta_{s,t^*}$ with $s, t\neq j_0, j_0^*$, we obtain $\sum_{k\neq i_0, i_0^*}z_{k,s}z_{k^*,t}=\alpha^2\delta_{s,t^*}$, which is equivalent to $ZH_{2d-1}Z^t=\alpha^2H_{2d-1}$. The remaining bilinear form identities, i.e., $(r_{i_0^*}, r_s)=0$ for $s\neq i_0$, $(c_{j_0^*}, c_t)=0$ for $t\neq j_0$ are automatically satisfied. We omit these routine verifications for brevity.

For the spin condition, consider $U=\{s_1, s_2, \cdots , s_{d-1}\}\subset [1, 2d+1]\setminus\{i_0, i_0^*\}$, and $U'=\{t_1, t_2, \cdots , t_{d-1}\}\subset [1, 2d+1]\setminus\{j_0, j_0^*\}$ of cardinality $d-1$, where $[U:U'](Z)$ ranges over all the $(d-1)$ minors of the matrix $Z$. Set $V:=U\cup\{i_0\}, V'=U'\cup\{j_0\}$. By Sylvester's determinant identity and $y_{i_0, j_0}=1$, we obtain:
\begin{equation}
[V:V'](Y)=
\left(\begin{array}{ccc}
  y_{s_1,t_1}-y_{s_1,j_0}y_{i_0,t_1}   & \cdots & y_{s_1,t_{d-1}}-y_{s_1,j_0}y_{i_0, t_{d-1}}\\
   \vdots  & \ddots & \vdots\\
   y_{s_{d-1},t_1}-y_{s_{d-1},j_0}y_{i_0,t_1}   & \cdots & y_{s_{d-1},t_{d-1}}-y_{s_{d-1},j_0}y_{i_0, t_{d-1}}
\end{array}
\right)=[U:U'](Z).
\end{equation}
Observe that $V^\bot=U^{\bot}\cup\{i_0\}$, ${V'}^\bot={U'}^{\bot}\cup\{i_0\}$, which immediately gives $[U^\bot: {U'}^\bot](Z)=[V^\bot:{V'}^\bot](Y)$. Moreover, the signature terms satisfy 
\begin{flalign*}
\sgn(\sigma_V)\sgn(\sigma_{V'})&=(-1)^{\sum V+\sum V'}=(-1)^{\sum U+\sum U'+i_0+j_0}\\
&=(-1)^{i_0+j_0}\sgn(\sigma_U)\sgn(\sigma_{U'}).     
\end{flalign*}
Consequently, the relations in $J_d^\pm$ translate to
\begin{equation}
[U^\bot: {U'}^\bot](Z)\mp (-1)^{d+i_0+j_0}\sgn(\sigma_U)\sgn(\sigma_{U'})\alpha [U:U'](Z).
\end{equation}

By setting $\epsilon=(-1)^{i_0+j_0+1}$, we have thus established that the affine chart $D_+(y_{i_0, j_0})$ with $i_0\neq i_0^*, j_0\neq j_0^*$ is the closed subscheme of
\begin{equation}\label{eq 317}
\frac{\calO_F[Z,\alpha,x_{i_0, j_0}, T_1,\cdots, T_\ell]}{(Z^tHZ-\alpha^2H,\; ZHZ^t-\alpha^2H,\pi-\alpha x_{i_0, j_0}T_1\cdots T_\ell)+J_{d-1}^{\pm\epsilon}}\times \mathbb{A}_{\calO_F}^{4d-2}   
\end{equation}
where $Z$ is the square matrix of size $2d-1$, and 
\[
\mathbb{A}_{\calO_F}^{4d-2}=\Spec \calO_F[y_{i_0,t}, y_{s,j_0}]_{s\neq i_0, i_0^*, t\neq j_0, j_0^*}
\]
is the corresponding affine space.
By sending $\alpha\mapsto T_0, x_{i_0, j_0}\mapsto T_{\ell+1}$, we obtain that (\ref{eq 317}) is isomorphic to $\calR(d-1,\ell+1)^\mp\times \mathbb{A}_{\calO_F}^{4d-2}$. Hence we obtain a closed immersion
\[
D_+(y_{i_0,j_0}) \hookrightarrow \calU(d-1,\ell+1)^{\pm\epsilon}\times \mathbb{A}_{\calO_F}^{4d-2}.
\]
By Proposition~\ref{prop flat R(d,l)}, the target is integral of dimension
\[
(d-1)(2d-1)+(\ell+1)+1+(4d-2)=d(2d+1)+\ell+1.
\]
Since $D_+(y_{i_0,j_0})$ is an open subscheme of
$\Bl_{(X,T_0)}(\calU(d,\ell))$, it has the same dimension as $\calU(d,\ell)$, namely
$d(2d+1)+\ell+1$. Therefore the above closed immersion is an isomorphism. \end{proof}



We next analyze the chart $D_+(\alpha)$.

\begin{lemma}\label{lem:alpha-chart-semistable}
The affine chart $D_+(\alpha)$ is
semi-stable over $\calO_F$. More precisely, $D_+(\alpha)$ is isomorphic to
\[
C_d^{\pm}\times_{\calO_F}
\Spec \calO_F[T_0,\ldots,T_\ell]/(\pi-T_0T_1\cdots T_\ell),
\]
where $C_d^{\pm}$ is a connected component of the split orthogonal group over $\calO_F$. In particular, for $\ell=0$, the chart
$D_+(\alpha)$ is smooth over $\calO_F$.
\end{lemma}
\begin{proof}
Consider the affine chart $D_+(\alpha)$, and assume $\alpha=1$. Then
\[
x_{i,j}=T_0y_{i,j}
\qquad (1\le i,j\le 2d+1).
\]
Substituting this into the defining equations of $\calU(d,\ell)$, the relations
\[
X^tHX-T_0^2H=0,\qquad XHX^t-T_0^2H=0
\]
become
\begin{equation}
Y^tHY-H=0,\qquad YHY^t-H=0.
\end{equation}
Similarly, the relations in $J_d^\pm$ become
\begin{equation}
[U^\bot:{U'}^\bot](Y)
\mp
(-1)^d \sgn(\sigma_U)\sgn(\sigma_{U'}) [U:U'](Y),
\end{equation}
for all subsets $U,U'\subset [1,2d+1]$ of cardinality $d$.
As in Lemma~\ref{lemmaRingA}, these equations cut out a connected component
$C_d^\pm$ of the split orthogonal group over $\calO_F$, i.e. 
\[
\Spec
\frac{\calO_F[Y]}
{(Y^tHY-H,\; YHY^t-H,\; [U^\bot:{U'}^\bot](Y)
\mp
(-1)^d \sgn(\sigma_U)\sgn(\sigma_{U'}) [U:U'](Y))}\]
is isomorphic to 
\[
C_d^\pm := \Spec\frac{\calO_F[Y]}{(Y^tHY-H,\; YHY^t-H,\; \det(Y)-(-1)^d)}.
\]
In particular,
$C_d^\pm$ is smooth over $\calO_F$ of relative dimension $d(2d+1)$,
hence of total dimension $d(2d+1)+1$. The remaining equation is
\[
\pi=T_0T_1\cdots T_\ell.
\]
Hence, $ D_+(\alpha)$ is a closed subscheme of
\[
C_d^\pm\times_{\calO_F}
\Spec \calO_F[T_0,\ldots,T_\ell]/(\pi-T_0T_1\cdots T_\ell).
\]
The second factor is a standard semi-stable scheme over $\mathcal O_F$ of
relative dimension $\ell$. Since $C^\pm_d$ is smooth over $\mathcal O_F$ of
relative dimension $d(2d+1)$, the product is semi-stable over $\mathcal O_F$
of total dimension $d(2d+1)+\ell+1$.

On the other hand, $D_+(\alpha)$ is an open subscheme of
$\Bl_{\calI_0}(\calU(d,\ell))$, hence has the same dimension as $\calU(d,\ell)$, namely
$d(2d+1)+1+\ell$ (see Proposition \ref{prop flat R(d,l)}). Therefore the above closed immersion is an isomorphism, and
$D_+(\alpha)$ is semi-stable over $\calO_F$.

If $\ell=0$, then the second factor is
\[
\Spec \calO_F[T_0]/(\pi-T_0)\simeq \Spec \calO_F,
\]
so $D_+(\alpha)\simeq C_d^\pm$ is smooth over $\calO_F$.
\end{proof}

\subsection{Semi-stable models}\label{ProofProp3.3}
\begin{proof}[Proof of Proposition~\ref{prop:qs-affine-semistable}]
Since $\calU(i,0)\simeq \calX_0$, the first blow-up $c_1:\calX_1\rightarrow \calX_0$ is along $V(I_0)=V((X,\pi))$. From the discussion in \S \ref{section blow-ups}, the scheme $\calX_1$ is covered by open charts $D_+(\alpha)$ and $D_+(y_{i_0, j_0})$ with $i_0\neq i_0^*$, $j_0\neq j_0^*$. 

\begin{lemma}\label{lem:alpha-chart}
Let
\[
\rho:\Bl_{(X,T_0)}(\calU(d,\ell))\longrightarrow \calU(d,\ell)
\]
be the blow-up along the ideal $(X,T_0)$. Then for every $r\ge 1$, the strict
transform of $V(T_0,\wedge^r X)$ does not meet the chart $D_+(\alpha)$.
\end{lemma}

\begin{proof}
On the chart $D_+(\alpha)$, one has $X=T_0\cdot Y$. Hence
\[
\wedge^r X=T_0^r \wedge^rY.
\]
It follows that
\[
(T_0,\wedge^r X)\mathcal{O}_{D_+(\alpha)}
=
(T_0,T_0^r\cdot \wedge^r Y)
=
(T_0).
\]
On the other hand,
\[
(X,T_0)\mathcal{O}_{D_+(\alpha)}=(T_0).
\]
Therefore the inverse images of the closed subschemes $V(T_0,\wedge^r X)$ and
$V(X,T_0)$ in $D_+(\alpha)$ coincide. Hence the inverse image of
\[
V(T_0,\wedge^r X)\setminus V(X,T_0)
\]
in $D_+(\alpha)$ is empty, and so the strict transform of $V(T_0,\wedge^r X)$
does not meet $D_+(\alpha)$.
\end{proof}

In particular, when $d=i, \ell=0$, we have the blow-up $\rho: \calX_1=\Bl_{I_0}(\calU(i,0))\rightarrow \calU(i,0)$. The strict transforms of $V(\pi, \wedge^r X)$ do not meet the chart $D_+(\alpha)$. It is enough to consider the sequence of blow-ups restricting to $D_+(y_{i_0, j_0})$.

The strict transform of $V(\pi, \wedge^{r+1} X)\subset \calU(i,0)$ is cut out by the ideal $(T_0, \wedge^r Z)$ in $D_{+}(y_{i_0, j_0})\simeq U(i-1,1)^{\pm\epsilon}\times \mathbb{A}_{\calO_F}^{4i-2}$. Indeed, for $s\leq r$, all $s\times s$-minors of the matrix $Z$ in (\ref{eq Z}) can be represented by the $(s+1)\times (s+1)$-minors of $X$, and $\pi=T_0\cdot x_{i_0, j_0}$. Thus, 
the strict transform of $V(\pi, \wedge^{r+1} X)$ is contained in $V((T_0, \wedge^r Z))\subset D_{+}(y_{i_0, j_0})$. The scheme $V(\pi, \wedge^{r+1} X)$ is isomorphic to 
\[
\Spec k[X]/(X^tH_{2d+1}X, XH_{2d+1}X^t, \wedge^{r+1}X),
\]
which is integral of dimension $r(4i-2r+1)$ (see Proposition \ref{prop:special-fiber-schubert-centers} (2)). On the other hand, the scheme
\[
\mathbb{A}_{\calO_F}^{4i-2}\times \calU(i-1,1)^{\pm\epsilon}/(\wedge^r Z, T_0)\simeq
\mathbb{A}_{\calO_F}^{4i-2}\times \Spec k[Z, T_1]/(ZH_{2d-1}Z^t, Z^tH_{2d-1}Z, \wedge^r Z),
\]
which is integral of dimension $(r-1)(4(i-1)-2(r-1)+1)+(4i-2)+1$. Therefore, the above two schemes are integral with the same dimension, so that the strict transform of $V(\pi, \wedge^{r+1} X)$ is equal to $V((T_0, \wedge^r Z))$. Thus, the blow-up sequence in (\ref{quasisplit.seq. of blow-ups})
\[
\mathcal{X}_0 \xleftarrow{\,c_1\,} \mathcal{X}_1 \xleftarrow{\,c_2\,} \cdots \xleftarrow{\,c_i\,} \mathcal{X}_i,
\]
reduces to
\begin{equation}
\mathcal{Y}_0 \xleftarrow{\,c_1'\,} \mathcal{Y}_1 \xleftarrow{\,c_2'\,} \cdots \xleftarrow{\,c_{i-1}'\,} \mathcal{Y}_{i-1},    
\end{equation}
where $\mathcal{Y}_0=\calU(i-1,1)^{\pm\epsilon}$, $c_1'$ is the blow-up along $V(T_0, Z)$, and for $2\leq \ell\leq i-1$, the morphism
$c'_{\ell}\colon \mathcal{Y}_{\ell}\to \mathcal{Y}_{\ell-1}$ is the blow-up along the strict transform of
$V((T_0, \wedge^\ell Z))$ in $\mathcal{Y}_{\ell-1}$. By induction on $i$, after performing $i$ successive blow-ups of $\calX_0$, 
each affine chart in $\calX_i$ is either semi-stable, coming from an
$\alpha$-chart as in Lemma~\ref{lem:alpha-chart-semistable}, or isomorphic to
\begin{equation}
\mathbb{A}_{\calO_F}^{2i^2}\times \calU(0,i)^{\epsilon'}\simeq
\mathbb{A}_{\calO_F}^{2i^2}\times \Spec \calO_F[T_0, T_1,\cdots ,T_i]/(\pi-T_0\cdots T_i)
\end{equation}
where $\epsilon'\in\{+, -\}$ is the initial sign. The right-hand side is semi-stable over $\calO_F$. This finishes the proof of
Proposition~\ref{prop:qs-affine-semistable} and therefore of
Theorem~\ref{thm:qs-semistable-resolution}.
\end{proof}

\section{Semi-stable resolution for split orthogonal local models}\label{s.s. resolution 2}

In this section, we show that a suitable sequence of blow-ups of the
split orthogonal local model $\RM_{i,\CO_F}^\pm$, along the Schubert strata in the special fiber, yields a semi-stable model. 

\subsection{Split orthogonal local model.} We consider the (pseudo-)maximal parahoric case. As in the quasi-split case, we may assume $I=\{i\}$ with $1\leq i\leq \lfloor \frac{n}{2}\rfloor$ by Theorem~\ref{thm:max-schubert}(2). The spin local model $\M_i^\pm$ is an irreducible projective scheme over $\calO_{F_0}$. Set
\[
\M^\pm_{i,\calO_F}:=\M^\pm_i\otimes_{\calO_{F_0}}\calO_F.
\]
For each integer $\ell$ with $0\leq \ell\leq i$, let $S_\ell$ be the Schubert variety of the special fiber $\M_{i,\calO_F}^\pm\otimes_{\calO_F} k\simeq \M_{i}^\pm\otimes_{\calO_{F_0}} k$. By Theorem~\ref{thm:max-schubert}(3), we have a chain of Schubert varieties:
\[
\begin{tikzcd}[column sep=1em, row sep=0.6em]
&  S_i\\
S_0\subset  S_1\subset  \cdots \subset S_{i-1} \arrow[ru, phantom,sloped, "\subset"near end]\arrow[rd, phantom,sloped, "\subset"near end] &  \\
& S_i',
\end{tikzcd}
\]
where $\overline{\M_i^\pm(\ell)}=S_\ell$ for $0\leq \ell\leq i-1$, and $\overline{\M_i^\pm(i)}=S_i\cup S_i'$. 

\begin{thm}\label{thm:split-semistable-resolution}
Assume $I=\{i\}$ with $1\le i\le \lfloor \frac n2\rfloor$. Let
\[
\M^\pm_{i,\calO_F}=Z_0 \xleftarrow{\,b_1\,} Z_1 \xleftarrow{\,b_2\,} \cdots \xleftarrow{\,b_i\,} Z_i
\]
be the sequence of blow-ups such that $b_1$ is the blow-up of $Z_0=\M^\pm_{i,\calO_F}$ along
$S_0$, and for $2 \leq \ell \leq i$, the morphism
\[
b_{\ell}\colon Z_{\ell}\to Z_{\ell-1}
\]
is the blow-up along the strict transform of $S_{\ell-1}$ in $Z_{\ell-1}$.
Then $Z_i$ is semi-stable over $\mathcal{O}_F$.
\end{thm}

By Theorem~\ref{thm:affine-chart}, there exists an open affine subscheme $U_i^{\pm}\subset \M_{i,\calO_F}^{\pm}$ containing a worst point $\ast \in S_0$ such that
\[
       U_i^\pm\simeq \BA^{\frac{(n-2i)(n+2i-1)}{2}}_{\CO_{F}}\times_{\CO_{F}}\Spec R_i \text{\ with\ } R_i=\CO_{F}[X]/I_i^\pm,
\]
where the matrix $X$ has size $2i\times 2i$, and the $\calG_i^\circ$-translates of
$U_i^{\pm}$ cover $M_{i, \calO_F}^{\pm}$.

By Theorem~\ref{thm:affine-chart} and
Proposition~\ref{prop:special-fiber-schubert-centers}, to prove Theorem~\ref{thm:split-semistable-resolution} it suffices to establish the following affine local statement.

\begin{prop}\label{prop:split-affine-semistable}
Let
\[
\mathcal{X}_0:=\operatorname{Spec} R_i,
\qquad
I_{\ell}:=(\pi,\wedge^{\ell+1}X)\subset R_i
\qquad (0\leq \ell<i).
\]
Consider the sequence of blow-ups
\begin{equation}\label{split.seq. of blow-ups}
\mathcal{X}_0 \xleftarrow{\,c_1\,} \mathcal{X}_1 \xleftarrow{\,c_2\,} \cdots \xleftarrow{\,c_i\,} \mathcal{X}_i,    
\end{equation}
where $c_1$ is the blow-up along $V(I_0)$, and for $2\leq \ell\leq i$, the morphism
$c_{\ell}: \mathcal{X}_{\ell}\to \mathcal{X}_{\ell-1}$ is the blow-up along the strict transform of
$V(I_{\ell-1})$ in $\mathcal{X}_{\ell-1}$.
Then $\mathcal{X}_i$ is semi-stable over $\mathcal{O}_F$.
\end{prop}

The proof of Proposition \ref{prop:split-affine-semistable} is similar to the proof of Proposition \ref{prop:qs-affine-semistable}. Indeed, we only need to focus on the ideal $J_i^\pm$ given by the spin condition, which is different from the quasi-split case. Similar to \S \ref{section:quasi-split-U(d,l)}, we study a family of auxiliary affine schemes $\calU(d,\ell)^\pm:=\Spec \calR(d,\ell)^\pm$, where
\[
\calR(d,\ell)^\pm :=
\frac{\calO_F[X,T_0,T_1,\ldots,T_\ell]}
{(X^tH_{2d}X-T_0^2H_{2d},
XH_{2d}X^t-T_0^2H_{2d},
\pi-T_0T_1\cdots T_\ell)+J_d^\pm}.
\]
Here $X$ is the square matrix of size $2d$, and $J_d^\pm$ is generated by the relations
\begin{flalign*}
&[U^\bot:{U'}^{\bot}](X)\mp\sgn(\sigma_U)\sgn(\sigma_{U'})[U:U'](X) \\
\text{\ and\ } & [V^\bot:{V'}^\bot](X)\mp\sgn(\sigma_V)\sgn(\sigma_{V'})\cdot T_0^2[V:{V'}](X)
    \end{flalign*} 
for all subsets $U,U'$ (resp. $V,V'$) of $[1,2d]$ of cardinality $d$ (resp. $d-1$). We omit the lower indices of $H_{2d}$ if there is no confusion. When the sign is fixed, we simply write $\calU(d,\ell)$ and $\calR(d,\ell)$. Observe that when $d=i,\ell=0$, we have 
\[
\calU(i,0)\simeq \Spec R_i;
\]

\begin{prop}
The ring $\mathcal{R}(d, \ell)$ is flat over $\mathcal{O}_F$ and integral of (total) dimension $d(2d - 1) + \ell + 1$.    
\end{prop}

\begin{proof}
    Write $\CR\coloneqq \CR(d,l)^\pm$. Then $\mathcal{R}_k \coloneqq \mathcal{R} \otimes k$ is isomorphic to
\[ \frac{k[X, T_0, T_1, \dots, T_ \ell]}{(XH X^t - T_0^2 H,\; X^t H X - T_0^2 H, \;T_0 T_1 \cdots T_ \ell)+ J_d^\pm}. \]
Define 
\begin{flalign*}
    \mathscr{I} &:= (XH X^t - T_0^2 H, \; X^t H X - T_0^2 H,\; T_0 T_1 \cdots T_ \ell) + J_d^\pm,\\ \mathscr{I}_{r} &:= \sI+(T_r), \text{\ for $ r \in \{0,\dots,  \ell\}$}.
\end{flalign*}
    By Theorem \ref{thm:max-schubert} (3), we have \begin{flalign*}
        \sI_0=\fp_1\cap \fp_2, \text{\ and\ $\fp_1, \fp_2$ are prime ideals.}
    \end{flalign*}   
    A similar proof of Proposition \ref{prop flat R(d,l)} implies that $\sI_r$ is a prime ideal for $r\in [1,\ell]$, and that \[ \sI=\cap_{r=0}^\ell\sI_r=\fp_1\cap\fp_2\cap(\cap_{r=1}^\ell\sI_r).\]
    The rest of the proof proceeds similarly as in Proposition \ref{prop flat R(d,l)}.
\end{proof}

\subsection{Semi-stable resolution} Consider the blow-up $\Bl_{\calI_0}(\calU(d,\ell))$ of $\calU(d,\ell)$ along the ideal $\calI_0= (X,T_0)=((x_{i,j})_{1\leq i,j\leq 2d},T_0)$. The blow-up $\Bl_{\calI_0}(\calU(d,\ell))$ is covered by the $(2d)^2+1$ affine charts $D_+(y_{i,j})$ and $D_+(\alpha)$,
where $y_{i,j}$ and $\alpha$  denote the homogeneous degree-one elements corresponding to $x_{i,j}$ and $T_0$, respectively. For $1\leq i_0\leq 2d$, set $i_0^*=2d+1-j$. Since $X$ is of size $2d$, we always have $i_0^*\neq i_0, j_0^*\neq j_0$.

\begin{lemma}\label{lemma64}
The blow-up $\Bl_{\calI_0}(\calU(d,\ell))$ is covered by the affine charts
$D_+(y_{i_0,j_0})$ for $1\leq i_0,j_0\leq 2d$, together with the chart $D_+(\alpha)$.
\begin{altenumerate}
    \item The affine chart $D_+(y_{i_0, j_0})$ is isomorphic to
\[
\calU(d-1,\ell+1)^{\pm\epsilon}\times \mathbb{A}_{\calO_F}^{4d-4},
\]
where $\epsilon=(-1)^{i_0+j_0}$.
    \item The affine chart $D_+(\alpha)$ is isomorphic to 
\[
\Spec\frac{\calO_F[Y]}{(Y^tH_{2d}Y-H_{2d},\; YH_{2d}Y^t-H_{2d},\; \det(Y)\mp 1)}\times_{\calO_F}
\Spec \frac{\calO_F[T_0,\ldots,T_\ell]}{(\pi-T_0T_1\cdots T_\ell)}.
\]
\end{altenumerate}
\end{lemma}

\begin{proof}
The proof is analogous to Lemma \ref{lemma chart y_{i,j}} and Lemma \ref{lem:alpha-chart}. 
By setting $z_{s,t}=y_{s,t}-y_{i_0,t}y_{s,j_0}$,for $s\neq i_0, i_0^*, t\neq j_0, j_0^*$. The relations $X^tH_{2d}X-T_0^2H_{2d},
XH_{2d}X^t-T_0^2H_{2d}$ translate into
\[
Z^tH_{2d-2}Z-T_0^2H_{2d-2},\quad
ZH_{2d-2}Z^t-T_0^2H_{2d-2}.
\]
Moreover, the relations in $J_d^\pm$ translate to 
\begin{flalign*}
&[U^\bot:{U'}^{\bot}](Z)\mp (-1)^{i_0+j_0}\sgn(\sigma_U)\sgn(\sigma_{U'})[U:U'](Z) \\
\text{\ and\ } & [V^\bot:{V'}^\bot](Z)\mp (-1)^{i_0+j_0}\sgn(\sigma_V)\sgn(\sigma_{V'})\cdot \alpha^2[V:{V'}](Z).
    \end{flalign*}
The free variables $y_{i_0, t}, y_{s, j_0}$ for $s\in [1,2d]\setminus\{i_0, i_0^*\}$, $t\in [1,2d]\setminus\{j_0, j_0^*\}$ yield the affine space $\mathbb{A}_{\calO_F}^{4d-4}$. This completes the proof of (1).

For (2), note that we have $X=T_0Y$ such that $Y^tH_{2d}Y=H_{2d}, YH_{2d}Y^t=H_{2d}$. Similarly to (1), the relations in $J_d^\pm$ become
\[
\det(Y)=\pm 1
\]
by Jacobi's identity. This finishes the proof of (2).
\end{proof}
The proof of Proposition~\ref{prop:split-affine-semistable} follows exactly as
in the quasi-split case; see the proof of Proposition~\ref{ProofProp3.3}.
Indeed, Lemma~\ref{lemma64} gives the same induction on $d$: the chart
$D_+(\alpha)$ is already strictly semi-stable, and each chart
$D_+(y_{i_0,j_0})$ is isomorphic to
\[
\calU(d-1,\ell+1)^{\pm\epsilon}\times \mathbb A_{\calO_F}^{4d-4}.
\]
Iterating gives a strictly semi-stable modification of $\calU(d,\ell)$ over
$\calO_F$.

\section{Semi-stable resolution for Weil restriction orthogonal case}\label{section ss ortho Res}
In this section, we construct a semi-stable resolution for the Weil restriction local model in the maximal parahoric case. We keep the same notation as in \S \ref{sec-splitting}. In particular, $F_0$ is a complete discretely valued field with ring of integers $\calO_{F_0}$, uniformizer $\pi_0$, and perfect residue field $k_{F_0}$ of characteristic $p> 2$; and $(B,*,V, \pair{\ ,\ })$ is a (local) PEL-datum over $F_0$ as in \S \ref{sec-splitting}.  

We assume that in \eqref{Bprod}, $I_1=\cbra{1}$ is a singleton and of type (D), i.e., $$B\simeq M_{m_1}(F_1)$$ for a finite extension $F_1$ of $F_0$, and the involution\footnote{Note that the condition on $*$ implies that $m_1$ is an even integer.} $*:M_{m_1}(F_1)\rightarrow M_{m_1}(F_1)$ is given by $b^*=Jb^tJ^{-1}$ for some invertible alternating matrix $J$. In particular, $*$ acts trivially on $F_1$, which is identified with scalar matrices in $M_{m_1}(F_1)$.  {For simplicity, set $d=[F_1:F_0]$ and let $e=[F_1: F_1^{\rm unr}]$ denote the ramification index, $f=[F_1^{\rm unr}: F_0]$ so that $d=ef$. Let $L$ be the Galois closure of $F_1/F_0$, with  ring of integers $\calO_{L}$. We write $\Sigma \coloneqq \cbra{\sigma_{j}^l\colon F_1\ra L}_{1\leq j\leq f,1\leq l\leq e} \text{\ and\ } \Sigma^\unra\coloneqq \cbra{\sigma_{j}\colon F_1^\unra\ra L}_{1\leq j\leq f}$. 
We set \begin{flalign}  \label{eq-KK'}
    K\coloneqq L(\sqrt{-\sigma_j^l(\pi_1)}\ |\ 1\leq j\leq f, 1\leq l\leq e)  
\end{flalign}  with ring of integers $\calO_K$. } 

{
\begin{lemma}\label{lem-KL0}
   Let $L^0$ be the unique unramified quadratic extension of $L$. Then we have $K=L(\sqrt{-\pi_1})$ if $F_1=F_0$; and $K=L^0(\sqrt{-\pi_1})$ otherwise. 
\end{lemma}
\begin{proof}
    If $F_1=F_0$, then the lemma is clear. Now we assume $F_1\neq F_0$. Note that for any two embeddings $\sigma, \sigma'\in \Sigma$, elements $-\sigma(\pi_1)$ and $-\sigma'(\pi_1)$ differ by a unit in $\CO_L$. Since $p\neq 2$, any unit in $\CO_L$ admits a square root in $L^0$ by Hensel's lemma. The lemma then follows. 
\end{proof}
}
Let $\calL$ be a self-dual multichain of $\calO_B$--lattices in $V$ and $\calG =\calG_\calL^\circ$ be the neutral component of the similitude automorphism group of $\calL$. Note that $V\simeq V_1^{m_1}$ by Morita equivalence, where $V_1$ is an $F_1$-vector space of even dimension. We assume that  $\dim V_1=2g$. Define 
\begin{equation}\label{eq-sym-pair-weilres}
\{x,y \}:=\bb x,J y\pp  
\end{equation}
for all $x,y\in V$.  Since $J$ is alternating, the pairing $\{\ ,\ \}$ is a non-degenerate $F_0$-symmetric pairing on $V$. We also use $\{\ ,\ \}$ to denote its restriction to the Morita component $V_1$. 
Since the duality isomorphism $\Hom_{F_1} (V_1, F_1)\simeq \Hom_{F_0}(V_1,F_0)$ given by composing with the trace $\Tr_{F_1/F_0}: F_1\rightarrow F_0$, there exists a unique non-degenerate symmetric form $(\ ,\ ): V_1\times V_1\rightarrow F_1$ such that
\[
\{x,y\}=\Tr_{F_1/F_0}(\delta\cdot ( x,y)),
\] 
for all $x, y\in V_1$, where $\delta$ is an $\calO_{F_1}$-generator of the inverse different $\mathcal{D}^{-1}_{F_1/F_0}$. Let $G$ denote the group over $F_0$ whose set of $R$-points ($R$ is any $F_0$-algebra) is the following \begin{flalign*}
    G(R)\coloneqq \cbra{g\in \GL_{B\otimes_{F_0}R}(V\otimes_{F_0}R)\ |\ \pair{gv,gw}=c(g)\pair{v,w}, c(g)\in R\cross }.
\end{flalign*}
Let $H=\GO(V_1,(\ ,\ ))$ denote the orthogonal similitude group over $F_1$, whose set of $R$-points ($R$ is any $F_1$-algebra) is \begin{flalign*}
    H(R)\coloneqq \cbra{g\in \GL(V_1\otimes_{F_1}R)\ |\ (gx,gy)=c'(g)(x,y), c'(g)\in {R'}^\times}.
\end{flalign*}

\begin{lemma}\label{lm-weilres-gpG-ortho}
    The group $G$ is isomorphic to the subgroup $G'$ of $\Res_{F_1/F_0}H$ whose similitude character is defined over $F_0$. More precisely, for any $F_0$-algebra $R$, \begin{flalign*}
        G(R)=H'(R) \coloneqq \cbra{h\in H(R\otimes_{F_0}F_1)\ |\ c'(h)\in R\cross }.
    \end{flalign*}
\end{lemma}
\begin{proof}
    Let $R$ be a $F_0$-algebra. By Morita equivalence, we have $$\GL_{B\otimes_{F_0}R}(V\otimes_{F_0} R)\simeq \GL_{F_1\otimes_{F_0}R}(V_1\otimes_{F_0} R).$$ Recall that $V=V_1^{m_1}$ by Morita equivalence. For any $v, w\in V\otimes_{F_0} R$, write $v_1$ (resp. $w_1$) for the corresponding vector in $V_1\otimes_{F_0} R$. For any $g\in \GL_{B\otimes_{F_0}R}(V\otimes_{F_0} R)$, let $g_1$ be the corresponding element in $\GL_{F_1\otimes_{F_0}R}(V_1\otimes_{F_0} R)$. Then we obtain
\begin{flalign*}
G(R)&= \cbra{g\in \GL_{B\otimes_{F_0}R
    }(V\otimes_{F_0}R)\mid \pair{gv,gw}=c(g)\pair{v,w}, c(g)\in R^\times }\\
    &=\cbra{g_1\in \GL_{F_1\otimes_{F_0}R
    }(V_1\otimes_{F_0}R)\mid \{g_1v_1,g_1w_1\}=c(g)\{v_1,w_1\}, c(g)\in R^\times }\\
    &=\cbra{g_1\in \GL_{F_1\otimes_{F_0}R
    }(V_1\otimes_{F_0}R)\mid (g_1v_1,g_1w_1)=c(g)(v_1,w_1), c(g)\in R^\times }\\ &=H'(R).
    \end{flalign*}
\end{proof}

We assume that the matrix with respect to the standard basis of the symmetric form $(\ , \ )$,  is $\psi$ (\ref{eq split form orthogonal}) in the split case, and $\phi$ in the quasi-split case (\ref{eq quasi-split form}). Let $\mu_\pm$ denote the minuscule cocharacter of the standard maximal torus in $\GO(V_1, (\ ,\ ))$. For each embedding $\sigma_j^l\in \Sigma$, we fix a minuscule cocharacter $\mu_{\sigma_j^l}=\mu_\pm$, and let $\mu: \mathbb{G}_{m, \bar{F}_0}\rightarrow G_{\bar{F}_0}$ denote the geometric cocharacter of $G$ whose $\sigma_j^l$-component in is $\mu_{\sigma_j^l}$. Let $E\subset L$ denote the reflex field of $\{\mu\}$ with the ring of integers $\calO_E$. Thus, by Def \ref{def-naivelocmod}, we obtain the naive local model $\Mnaive(\calL, \mu)$ over $\calO_E$ for the local model triple $(G, \calG,\mu)$.

From now on, we assume that $\calL=\Lambda_I$, where $I=\{2n\ZZ\pm i\}$, such that $\calG$ is a maximal parahoric subgroup. We claim that there exists a semi-stable resolution for $\Mloc(\calL, \mu)$ after base change to $\calO_K$.

Recall that we have the decompositions 
\begin{flalign*}
V_{1,L}=V_1\otimes_{F_0}L=\bigoplus_{1\leq j\leq f} V_{1,j} \text{\ and\ } V_{1,j}=\bigoplus_{1\leq l\leq e}V_{1,j}^l
\end{flalign*} 
on $V$. Since the involution $*$ acts trivially on $\sigma_{j}$, we have $*(1,j)=(1,j)$, which induces a non-degenerate $F_0$-bilinear alternating pairing $\pair{\ ,\ }\colon V_{1,j}\times V_{1,j}\ra F_0$. The pairing $\pair{\ ,\ }$ induces a perfect symmetric pairing $\{\ ,\ \}$ by (\ref{eq-sym-pair-weilres}). Thus, for each lattice $\Lambda_{1,j}\subset V_{1,j}$, we obtain a symmetric pairing:
\begin{flalign}
	\{\ ,\ \}: \Lambda_{1,j}\times (\Lambda^\vee)_{1,j}\lra \CO_{F_0}. 
\end{flalign}

By Definition \ref{def-naivesplmod}, there exists a naive splitting model $\M^{\rm nspl}_\calL=\M^{\rm nspl}(\calL, \mu, (\sigma_j^l))$. It is a projective scheme over $\Spec \calO_{L}$ that represents a rigidified version of the moduli problem defining $\Mnaive(\calL, \mu)$. Recall that there exists a forgetful morphism 
\begin{equation}\label{eq splitting to local}
f: \M^{\rm nspl}_\calL\rightarrow \Mnaive_\calL\otimes_{\calO_E} \calO_{L}, 
\end{equation}
which  restricts to an isomorphism on generic fibers.

Since $*(1,j)=(1,j)$, by Proposition \ref{prop-splitting}, this yields a diagram of morphisms of schemes over $\CO_L$:
\begin{equation}\label{eq-locdiag-ortho}
\begin{tikzcd}[column sep={6em, between origins}, row sep={3em, between origins}]
&  \arrow[ld, "\alpha_1"']\tilde{\M}^{\rm nspl}_\calL\arrow[rd, "\alpha_2"]&\\
\M^{\rm nspl}_\calL   &&   \hspace{2em}\prod_{\sigma_j^l\in\Sigma} \Mnaive(\calL_{\sigma_j^l}, \mu_{\sigma_j^l}),
\end{tikzcd}
\end{equation}
in which the slanted arrow $\alpha_1$ is an $\calH$-torsor, and $\alpha_2$ is a smooth, $\mathcal{H}$-equivariant morphism. Note that each $\Mnaive(\calL_{\sigma_j^l}, \mu_{\sigma_j^l})$ is isomorphic to the orthogonal naive local model $\Mnaive_i$ (see Definition \ref{naivelocalmodeldef}).

Consider the product of local models {$\prod_{\sigma_j^l\in\Sigma} \Mloc_i(\sigma_j^l)$}, where $\RM^\loc_i(\sigma_j^l)$ denotes the flat closure of $\Mnaive(\calL_{\sigma_j^l}, \mu_{\sigma_j^l})$. Since the closed immersion  $\iota\colon \prod_{\sigma_j^l\in\Sigma} \Mloc_i(\sigma_j^l)\rightarrow \prod_{\sigma_j^l\in\Sigma} \Mnaive(\calL_{\sigma_j^l}, \mu_{\sigma_j^l})$ is also a $\mathcal{H}$-equivariant morphism, we have a linear
modification such that the diagram restricts to a new diagram:
\begin{equation}\label{eq-local model diagram spl-ortho}
\begin{tikzcd}[column sep={6em, between origins}, row sep={3em, between origins}]
&  \arrow[ld, "\alpha_1"']\tilde{\M}^{\rm spl}_\calL\arrow[rd, "\alpha_2"]&\\
\M^{\rm spl}_\calL   &&   \prod_{\sigma_j^l\in\Sigma} \Mloc_i(\sigma_j^l),
\end{tikzcd}
\end{equation}
where $\alpha_1$ is a  $\mathcal{H}$-torsor and $\alpha_2$ is a smooth $\mathcal{H}$-equivariant morphism.

\begin{thm}\label{thm ss Weil ortho}
Suppose $F_1/F_0$ is finite separable. Assume that $I=\{i \}$ with $0< 2i\leq n$. After base change to $\calO_{K}$ (see \eqref{eq-KK'}), there exists a projective
$\calG_{\CO_{K}}$-equivariant morphism
\[
\M^{\rm ss}(\calL,\mu)\rightarrow \Mloc(\calL, \mu)\otimes_{\calO_E} \calO_{K},
\]
which is an isomorphism on the generic fibers and $\M^{\rm ss}(\calL,\mu)$ is semi-stable over $\calO_{K}$.    
\end{thm}

\begin{proof}
Consider the product of local models $\prod_{\sigma_j^l\in\Sigma} \Mloc_i(\sigma_j^l)$. For each embedding $\sigma=\sigma_j^l\in \Sigma$, the local model $\Mloc_i(\sigma)$ is isomorphic to the spin local model $\M^\pm_i$ over the ring of integers of ${L(\sqrt{-\sigma(\pi_1)})}$ by Theorem \ref{thm:spin-flat}. We obtain a semi-stable resolution $Z_{\sigma_j^l} \rightarrow \Mloc_i(\sigma_j^l)$ by Theorem \ref{thm:qs-semistable-resolution}, \ref{thm:split-semistable-resolution}. We then consider the product scheme $\prod_{\sigma_j^l\in \Sigma} Z_{\sigma_j^l}$ over $\calO_K$. While the product of semi-stable models is not necessarily semi-stable, one may apply 
succesive blow-ups along the products of irreducible components of the special fibers of the factors $Z_{\sigma_j^l}$ and obtain a semi-stable resolution
\[
\pi: Z\rightarrow \prod_{\sigma_j^l\in\Sigma} \Mloc_i(\sigma_j^l)
\] 
over $\calO_K$ (see \cite[Proposition 2.1]{Ha}).

Note that each irreducible component of the special fibers is a Schubert variety. Since Schubert varieties are stable under the action of $\mathcal{H}$, the morphism  $\pi$ is $\mathcal{H}$-equivariant. Using the linear modification \cite[\S 2]{P}, we obtain the following cartesian diagrams:
\begin{equation}
\begin{array}{ccccc}
   \M^{\rm ss}(\calL, \mu) & \xleftarrow{\ \alpha_1'\ } &
\tilde{Z} & \xrightarrow{\ \alpha_2'\quad } &  Z\\
\Bigg\downarrow{\gamma} &&\Bigg\downarrow \beta &&\Bigg\downarrow{\pi}\\  
\M^{\rm spl}_\calL &\xleftarrow{\ \alpha_1\ } &
\tilde{\M}^{\rm spl}_\calL & \xrightarrow{\ \alpha_2\ } &  \prod_{\sigma_j^l\in\Sigma} \Mloc_i,
\end{array}
\end{equation}
where $\alpha_1$ (resp. $\alpha_1'$) is $\calH$-torsor, and $\alpha_2$ (resp. $\alpha_2'$) is smooth, $\calH$-equivariant. In fact, since blow-up commutes with \'etale localization,
$\M^{\rm ss}(\calL, \mu)$ is a successive blow-up of $\M^{\rm spl}_\calL$.

We claim that the morphism $\ga: \M^{\rm ss}(\calL, \mu)\rightarrow \M^{\rm spl}_\calL$ is $\calG$-equivariant. Indeed, Proposition \ref{prop-splitting} shows that the splitting diagram
\eqref{eq-local model diagram spl-ortho} is $\calG$-equivariant; in particular, both $\al_1$ and $\al_2$ are $\calG$-equivariant. From the diagram morphism $\calG\rightarrow \prod_{\sigma_j^l\in \Sigma}\calG_{\sigma_j^l}$, the factorwise resolutions and
Hartl's construction are equivariant for the $\calG$-actions on the unramified
factors. Thus, the morphism $\ga$ is $\calG$-equivariant. Composing the $\calG$-equivariant morphisms $\ga$ and $f$, we obtain the desired projective $\calG$-equivariant morphism $\M^{\rm ss}(\calL, \mu)\rightarrow \Mloc(\calL, \mu)\otimes \calO_K$, which is an isomorphism on the generic fiber.
\end{proof}

\part{The Symplectic case}\label{part-symp}
\section{Symplectic local models}\label{symplectic-local-models}

In this section, we recall the definition of the symplectic local models
following \cite{Goertz03}. Let $F_0$ be a complete discretely valued field with ring of integers
$\mathcal{O}_{F_0}$, uniformizer $\pi_0$, and residue field $k$ of
characteristic $p>2$. 
Let $V=F_0^{2g}$ with ordered basis $e_1,\ldots,e_{2g}$. We equip $V$
with the standard symplectic form $\langle\ ,\ \rangle$ whose matrix with
respect to this basis is
\[
J_{2g}:=
\begin{pmatrix}
0&H_g\\
-H_g&0
\end{pmatrix},
\]
where $H_g$ denotes the unit anti-diagonal matrix of size $g$. Set $G:=\GSp(V,\langle\ ,\ \rangle).$
For $0\leq i\leq g$, define the standard lattices
\[
\Lambda_i:=
\mathcal{O}_{F_0}
\langle
\pi_0^{-1}e_1,\ldots,\pi_0^{-1}e_i,
e_{i+1},\ldots,e_{2g}
\rangle.
\]
For any lattice $\Lambda\subset V$, let
$\Lambda^\vee:=\{x\in V\mid \langle x,\Lambda\rangle\subset
\mathcal{O}_{F_0}\}$. For every non-empty subset $I\subset[0,g]$, define the associated standard
self-dual periodic lattice chain by
\[
\Lambda_I:=\{\Lambda_\ell\}_{\ell\in 2g\mathbb{Z}\pm I},
\]
where
\[
\Lambda_{-i}:=\Lambda_i^\vee
\qquad\text{and}\qquad
\Lambda_\ell:=\pi_0^{-d}\Lambda_{\pm i}
\quad\text{for }\ell=2gd\pm i.
\]

Let $\calG_I$ be the affine smooth group scheme of symplectic similitude
automorphisms of the lattice chain $\Lambda_I$. Thus $\calG_I$ is the
Bruhat--Tits group scheme attached to the corresponding parahoric subgroup
of $G(F_0)$.

\begin{defn}\label{def:symp-local-model}
The \emph{symplectic local model} $\Mloc_I$ is the projective scheme over
$\mathcal{O}_{F_0}$ representing the functor sending an
$\mathcal{O}_{F_0}$-algebra $R$ to the set of $R$-modules
$(\mathcal{F}_\Lambda)_{\Lambda\in\Lambda_I}$ such that:
\begin{itemize}
\item[(LM1)] for every $\Lambda\in\Lambda_I$, Zariski locally on $\Spec R$,
$\mathcal{F}_\Lambda$ is a direct summand of $\Lambda_R:=\Lambda\otimes_{\mathcal{O}_{F_0}}R$ of rank $g$;

\item[(LM2)] for every $\Lambda\in\Lambda_I$, one has $(\langle\ ,\ \rangle\otimes 1)
(\mathcal{F}_\Lambda,\mathcal{F}_{\Lambda^\vee})=0$;

\item[(LM3)] for every inclusion $\Lambda\subset\Lambda'$ in $\Lambda_I$,
the natural map $\Lambda_R\to\Lambda'_R$ sends
$\mathcal{F}_\Lambda$ to $\mathcal{F}_{\Lambda'}$, and the isomorphism $\Lambda_R\xrightarrow{\sim}(\pi_0\Lambda)_R$
identifies $\mathcal{F}_\Lambda$ with $\mathcal{F}_{\pi_0\Lambda}$.
\end{itemize}
\end{defn}

The generic fiber of $\Mloc_I$ is the Lagrangian Grassmannian
$\mathrm{LGr}(g,V)
\simeq
G_{F_0}/P_\mu$, where $P_\mu$ is the parabolic subgroup associated to the minuscule
cocharacter $\mu=(1^{(g)},0^{(g)})$. 

\begin{thm}[{\cite[Theorem~2.1]{Goertz03}}]\label{thm:symp-flat}
For every non-empty $I\subset[0,g]$, the symplectic local model
$\Mloc_I$ is flat over $\mathcal{O}_{F_0}$ of relative dimension
$\frac{g(g+1)}{2}$ and canonical in the sense of Scholze--Weinstein \cite[\S 21.4]{scholze2020berkeley}. Moreover, $\Mloc_I$ is normal and Cohen–Macaulay with reduced special fiber.
\end{thm}

We now restrict to maximal parahoric level. For $0\leq i\leq g$, set
\[
\Mloc_i:=\Mloc_{\{i\}}.
\]

\begin{defn}\label{def:symp-rank-strata}
Let $\iota:\Lambda_i\longrightarrow \pi_0^{-1}\Lambda_{-i}=\Lambda_{2g-i}$
denote the natural inclusion map and its base change. For an integer $\ell$,
denote by $
\Mloc_i(\ell)\subset \Mloc_{i,k}$
the locus where $\iota(\mathcal{F}_i)$ has rank $\ell$. 
\end{defn}

\begin{thm}\label{thm:symp-max-schubert}
Let $I=\{i\}\subset[0,g]$.
\begin{altenumerate}
\item The local models $\Mloc_0$ and $\Mloc_g$ are smooth over
$\mathcal{O}_{F_0}$.

\item The local model $\Mloc_i$ is isomorphic to $\Mloc_{g-i}$.

\item Assume that $i\neq 0,g$. The scheme $\Mloc_{i,k}$ is irreducible and reduced. Also, there exists a stratification
\begin{equation}\label{SymplecticStratification}
 \Mloc_{i,k}
=
\coprod_{\ell=\max\{0,2i-g\}}^i
\Mloc_i(\ell),
\end{equation}
where each stratum $\Mloc_i(\ell)$ is a single Schubert cell. Moreover,
$\Mloc_i(\ell')$ is contained in the Zariski closure
$\overline{\Mloc_i(\ell)}$ if and only if $\ell'\leq \ell$. Consequently,
$\Mloc_{i,k}$ contains precisely $\min\{i,g-i\}+1$
Schubert cells.
\end{altenumerate}
\end{thm}
\begin{proof}
The smoothness of $\Mloc_0$ and $\Mloc_g$, and the irreducibility and
reducedness of $\Mloc_{i,k}$ for $1\leq i\leq g-1$, follow from
\cite[\S5.1]{Goertz03}. For (2), let $\tau\in \GL(V)(F_0)$ be defined, with respect to the ordered basis
$e_1,\ldots,e_g,e_{g+1},\ldots,e_{2g}$, by
\[
\tau=
\begin{pmatrix}
0&-\pi_0^{-1}I_g\\
I_g&0
\end{pmatrix}.
\]
Equivalently,
\[
\tau(e_j)=e_{g+j},\qquad
\tau(e_{g+j})=-\pi_0^{-1}e_j
\qquad (1\leq j\leq g).
\]
Since
\[
\tau^tJ_{2g}\tau=\pi_0^{-1}J_{2g},
\]
we have $\tau\in\GSp(V)(F_0)$. Moreover, from the definition of the standard
lattices one checks that $\tau(\Lambda_m)=\Lambda_{m+g}$
for all $m$ and thus 
\[
\tau(\Lambda_{\{i\}})
=
\{\Lambda_{m+g}\mid m\in 2g\mathbb Z\pm i\}
=
\{\Lambda_m\mid m\in 2g\mathbb Z\pm(g-i)\}
=
\Lambda_{\{g-i\}}.
\]
Therefore we deduce that $\Mloc_i\xrightarrow{\sim}\Mloc_{g-i}.$

Next, observe that for every $g\in \calG_{i,k}$, the submodules
$\iota(\mathcal F_i)$ and $\iota(g\mathcal F_i)$ have the same rank.
Hence each $\Mloc_i(\ell)$ is $\calG_{i,k}$-stable. Under the usual embedding of $\Mloc_{i,k}$ into the corresponding partial affine
flag variety, the $\calG_{i,k}$-orbits are precisely the Schubert cells. Thus each $\Mloc_i(\ell)$ is a union of Schubert cells. Also, by definition of $ \Mloc_i(\ell)$, it follows that $\max\{0,2i-g\}\leq \ell\leq i.$ Now, from Lemma \ref{SchubertcellsNumberSymplectic} we obtain that there are precisely $ i-\max\{0,2i-g\}+1$ Schubert cells in $\Mloc_{i,k}$. Thus, each stratum $\Mloc_i(\ell)$ is a single Schubert cell. 
\end{proof}

\begin{lemma}\label{SchubertcellsNumberSymplectic}
Assume that $1\leq i\leq g-1$. There are $  i-\max\{0,2i-g\}+1$ Schubert cells in $\Mloc_{i,k}$.  
\end{lemma}
One can prove the lemma by a similar argument in \cite[Corollary 4.8]{yang25TopFlat} or \cite[Corollary 4.13]{YZZ26}, using ``minuscule $G$-faces" in \cite[\S 9--10]{KR}. Here, we give an alternative proof by reducing to the corresponding result in \cite{YZZ26}. 
\begin{proof}
	Let $T\sset G$ be the maximal (split) torus consisting of diagonal matrices. We  identify the standard apartment corresponding to $T$ in the (extended) Bruhat--Tits building $\sB(G,K)$ with \begin{flalign*}
	\CA_G=X_*(T)\otimes\BR= \cbra{(r_1,\ldots,r_{2g})\in \BR^{2g}\ |\ r_1+r_{2g}=\cdots=r_{g}+r_{g+1} }\simeq \BR^{g+1}.
\end{flalign*}
Let $\eta$ denote the similitude character of $G$, and let $\chi_j\in X^*(T)$ be the character sending $\diag(x_1,\ldots,x_{2g})\in T$ to $x_j$, for $1\leq j\leq g$. Then the affine root system $\Phi_{\aff} =\Phi_{\aff}(G,T,K)$ is \begin{flalign*}
	\cbra{\pm (\chi_j- \chi_k)+\BZ, \pm(\chi_j+\chi_k-\eta)+\BZ\ |\ 1\leq j<k\leq g} \cup \cbra{\pm(2\chi_j-\eta)+\BZ\ |\ j\in [1,g]}.
\end{flalign*} 

Let $G'\coloneqq \GO(V',\psi)^\circ$ be the identity component of the (quasi-split but non-split) orthogonal similitude group attached to a symmetric space $(V',\psi)$ over $K=k((t))$ of dimension $2g+2$ considered in \cite[\S 2.1]{YZZ26}. 
Fix the maximal $K$-split torus $S$ consisting of diagonal matrices in $G'\sset \GL(V')\simeq\GL_{2g+2}$. We identify its associated standard apartment $\CA_{G'}$ with \begin{flalign*}
	X_*(S)\otimes\BR = \cbra{(r_1,r_2,\ldots,r_{2g},s)\in \BR^{2g+1}\ |\ r_1+r_{2g}=\cdots= r_g+r_{g+1}=s}\simeq \BR^{g+1}.
\end{flalign*}
Then the affine root system $\Phi_\aff' = \Phi'_\aff(G',S,K)$ is \begin{flalign*}
	\cbra{\pm\chi_j'\pm\chi_k'+\BZ\ |\ j,k\in[1,g] \text{\ and\ }j\neq k} \cup \tcbra{\pm\chi_j'+\half\BZ\ |\ j\in [1,g]},
\end{flalign*} 
where $\chi_i'\in X^*(S)$ is the character sending $\diag(x_1,\ldots,x_{2g},y,y)\in S$ to $x_iy\inverse$.
Let $T'$ be the centralizer of $S$. Then $T'$ is a maximal torus of $G'$ and \begin{flalign*}
	X_*(T')_\Gamma\simeq \cbra{(b_1,\ldots,b_g,c-b_g,\ldots,c-b_{1},c)\in\BZ^{2g+1}} \sset X_*(S)\otimes\BR\sset \BR^{2g+1}.
\end{flalign*}
Here $\Gamma\coloneqq \Gal(k((u))/K)$ where $k((u))$ is the quadratic extension of $K$ with $u^2=-t$. 
Then we have an isomorphism \begin{flalign*}
	f\colon X_*(T) &\simto X_*(T')_\Gamma, \quad (r_1,\ldots,r_{2g}) \mapsto (r_1,\ldots,r_{2g},r_1+r_{2g}),
\end{flalign*}
which induces an isomorphism of the apartments $f\colon \CA_G\simto \CA_{G'}$. 
Moreover, by construction, we have $$f^*(\chi_j')=\chi_j-\eta/2.$$ Therefore, $f$ carries the root system $$\Phi_G=\cbra{\pm(\chi_j-\chi_k), \pm(\chi_j+\chi_k)\ |\ 1\leq j<k\leq g}\cup \cbra{\pm2\chi_j\ |\ j\in [1,g]} $$ to the (relative) root system $$\Phi_{G'} = \cbra{\pm\chi_j'\pm \chi_k'\ |\ j,k\in[1,g]\text{\ and\ } j\neq k}\cup \cbra{\pm\chi_j\ |\ j\in[1,g]}.$$ Then $f$ induces an isomorphism between the (finite) Weyl groups $W(\Phi_G)\simeq W(\Phi_{G'})$. Furthermore, $f$ sends the base alcove \begin{flalign*}
	\fa_G\colon -1< 2\chi_1-\eta< 2\chi_2-\eta< \cdots< 2\chi_g-\eta<0
\end{flalign*} to the alcove  \begin{flalign*}
	\fa_{G'}\colon -1/2<\chi_1'<\chi_2'<\cdots<\chi_g'<0.
\end{flalign*} 
Hence, $f$ induces an isomorphism of Iwahori--Weyl groups \begin{flalign*}
	\wt{W}_G\simeq X_*(T)\rtimes W(\Phi_G) \simto \wt{W}_{G'}\simeq X_*(T')_\Gamma\rtimes W(\Phi_{G'}),
\end{flalign*} identifying the $\fa_G$-Bruhat order on $\wt{W}_G$ with the $\fa_{G'}$-Bruhat order on $\wt{W}_{G'}$. Since $f$ maps the minuscule cocharacter $\mu_g=(1^{(g)},0^{(g)})$ of $G$ to the (minuscule) cocharacter $\mu_+=(1^{(g)},0^{(g)},1)\in X_*(T')_\Gamma$ of $G'$. Hence the admissible subset $\Adm(\mu_g)$ for $G$ is identified with $\Adm(\mu_+)$ for $G'$.

It follows from \cite[Corollary 4.13 and Lemma 4.20]{YZZ26} that the Schubert cells in $\RM^\loc_{i,k}$ are indexed by the same set as in the quasi-split orthogonal case. This proves the lemma.  
\end{proof}

By Theorem~\ref{thm:symp-max-schubert}(2), we may assume from now on that
$0\leq 2i\leq g$. We now record the affine chart around a point
$\ast\in \Mloc_{i,k}(0)$ in the minimal stratum.

\begin{prop}\label{prop:symp-affine-chart}
Assume that $I=\{i\}$ with $1\leq i\leq \lfloor g/2\rfloor$. There is an
affine open neighborhood $U_i\subset \Mloc_i$ of a point
$\ast\in \Mloc_{i,k}(0)$ such that
\begin{equation}\label{affinechart}
  U_i
\simeq
\mathbb A_{\mathcal O_{F_0}}^{\frac{(g-2i)(g+2i+1)}{2}}
\times_{\mathcal O_{F_0}}
\Spec R_i,  
\end{equation}
where
\[
R_i
:=
\frac{\mathcal O_{F_0}[X]}
{
\left(
XJ_{2i}X^t-\pi_0J_{2i},\;
X^tJ_{2i}X-\pi_0J_{2i}
\right)
}
\]
with $X$ a $2i\times 2i$ matrix. Moreover, for $0\leq \ell<i$, if $S_\ell:=\overline{\Mloc_i(\ell)}\subset \Mloc_{i,k}$,
then $U_i\cap S_\ell$ is cut out on the $\Spec R_i$-factor by $I_\ell=(\pi_0,\wedge^{\ell+1}X)$. Equivalently,
\begin{equation}\label{Sschubertsymplectic}
U_i\cap S_\ell
\simeq
\Spec
\frac{k[X]}
{
\left(
XJ_{2i}X^t,\;
X^tJ_{2i}X,\;
\wedge^{\ell+1}X
\right)
}
\times_k
\mathbb A_k^{\frac{(g-2i)(g+2i+1)}{2}}.
\end{equation}
\end{prop}
\begin{proof}
By \cite[\S5.1]{Goertz03}, we obtain the description of the affine chart in
\eqref{affinechart}. In loc. cit. the equations are written in the range
$2i\geq g$; applying the result to the isomorphic local model
$\Mloc_{g-i}$ gives the above description. 

On this chart, the rank of $\iota(\mathcal{F}_i)$ is the rank of the matrix
$X$; see \cite[\S5.1]{Goertz03}. Hence the condition
$\wedge^{\ell+1}\bigl(\iota:\mathcal F_i\to \mathcal F_{2g-i}\bigr)=0$ is cut out on the $\Spec R_i$-factor by
$(\pi_0,\wedge^{\ell+1}X)$. Using Theorem~\ref{thm:symp-max-schubert} and Proposition~\ref{IntegralDomainSymplectic} we obtain \eqref{Sschubertsymplectic}.
\end{proof}

For $0\leq \ell\leq i$, define the quotient ring
\[
R_i(\ell) := \frac{k[X]}
{
\left(
XJ_{2i}X^t,\;
X^tJ_{2i}X,\;
\wedge^{\ell+1}X
\right)
}.
\]
Denote by
\[
        \Delta_\ell:=(\ell,\ell-1,\ldots,1\mid 1,2,\ldots,\ell)
\]
the minor of $X$ whose underlying row and column
subsets are both $\{1,\ldots,\ell\}$, written in the convention for
symplectic minors. As in \cite[\S 4 Example]{Goertz03}, this is a doubly
admissible minor.
Recall from \cite{Con} the order on subsets used in the definition of symplectic standard tableaux:
if
\[
A_0=\{a_1<\cdots<a_p\},\qquad B_0=\{b_1<\cdots<b_q\}
\]
are subsets of $\{1,\ldots,i\}$, then
\[
A_0\leq B_0
\]
means that $p\geq q$ and $a_\nu\leq b_\nu$ for all
$1\leq \nu\leq q$. Hence, if $B_0=\{b_1<\cdots<b_q\}$ has cardinality
$q\leq \ell$, then
\[
\{1,\ldots,\ell\}\leq B_0,
\]
because $b_\nu\geq \nu$ for all $1\leq \nu\leq q$. It follows that $\Delta_\ell$
is smaller than every doubly admissible minor of size at most
$\ell$, in the partial order used for doubly symplectic standard tableaux.
Indeed, in the order on doubly admissible minors one compares the corresponding
row and column subsets separately, and both subsets attached to $\Delta_\ell$
are equal to $\{1,\ldots,\ell\}$.

\begin{lemma}\label{newbasis}
The doubly symplectic standard
tableaux all of whose minors have size at most $\ell$ form a $k$-basis
of $R_i(\ell)$.
\end{lemma}
\begin{proof}
Set
\[
\calR_i:=\frac{k[X]}{(XJ_{2i}X^t,\ X^tJ_{2i}X)}.
\]
By \cite[Theorem 6.1]{Con}, the doubly symplectic standard tableaux form
a $k$-basis of $\calR_i$. Hence their images span $R_i(\ell)=\calR_i/(\wedge^{\ell+1}X)$.
Moreover, any tableau containing a minor of size $>\ell$ maps to zero in
$R_i(\ell)$. Thus $R_i(\ell)$ is spanned by the images of the doubly symplectic
standard tableaux all of whose minors have size at most $\ell$.

It remains to prove linear independence. The proof is similar to the proof of \cite[Theorem 3.7]{Con}. Indeed, let $S_{\Delta_\ell}$ be the corresponding orbit closure in the doubly
symplectic setting of \cite[Theorem 6.1]{Con}, namely the closure of the
$B_i^-\times B_i^+$-orbit of the rank-$\ell$ matrix corresponding to $\Delta_\ell$. Here $B_i^-$ and $B_i^+$ are the lower and upper
triangular Borel subgroups of $\text{Sp}_{2i}$. Since $S_{\Delta_\ell}$ is the closure of
the orbit of a rank-$\ell$ matrix under $B_i^-\times B_i^+$, we have
\[
S_{\Delta_\ell}\subset  \text{Spec}\calR_i \cap \{ \text{rk}\,X\leq \ell\}=\operatorname{Spec}R_i(\ell) .
\]
On the other hand, let $T=(Q_1,\ldots,Q_m)$ be a doubly symplectic standard
tableau all of whose minors have size at most $\ell$. Then its first minor
$Q_1$ satisfies $\Delta_\ell\leq Q_1 .$ Therefore the same argument as in \cite[Lemma 3.5 and
Theorem 3.7]{Con}, applied in the doubly symplectic setting of
\cite[Theorem 6.1]{Con}, shows that the restrictions of these tableaux to
$S_{\Delta_\ell}$ are linearly independent. Hence they are linearly
independent in $R_i(\ell)$. This proves the lemma.    
\end{proof}
\begin{prop}\label{IntegralDomainSymplectic}
The ring $R_i(\ell)$ is a domain.
\end{prop}
\begin{proof}
The case $\ell=0$ is clear, since the ideal $\wedge^1X$ is generated by
all entries of $X$, and hence $R_i(0)\simeq k$. We therefore assume that
$1\leq \ell\leq i$. We first prove that $\Delta_\ell$ is not a zero divisor in $R_i(\ell)$. Let
\[
T=(P_1,\ldots,P_m)
\]
be a doubly symplectic standard tableau which occurs in the basis of
Lemma \ref{newbasis}. Thus each $P_j$ has size at most $\ell$,
and $P_1\leq P_2\leq\cdots\leq P_m.$ Since $\Delta_\ell\leq P_1$, the tuple $\Delta_\ell T:=(\Delta_\ell,P_1,\ldots,P_m)$ is again a doubly symplectic standard tableau. As in the proof of \cite[Corollary 4.4]{Goertz03} we deduce that $\Delta_\ell$ is not a zero divisor in $R_i(\ell)$.

We next compute the localization $(R_i(\ell))_{\Delta_\ell}$. On
$D(\Delta_\ell)$, write $X$ in block form
\[
X=
\begin{pmatrix}
A & B\\
C & D
\end{pmatrix},
\]
where $A$ is of size $\ell\times \ell$, $B$ of size $\ell\times (2i-\ell) $, $C$ of size $(2i-\ell)\times \ell $ and $D$ of size  $(2i-\ell)\times (2i-\ell) $. Since
$\Delta_\ell= \det A$, the matrix $A$ is invertible on this open set. \quash{The condition $\wedge^{\ell+1}X=0$ is equivalent to $ \text{rk}(X) \leq \ell$. Since $ A$ is of size $\ell\times \ell$ and invertible we get that $ \text{rk}(X) = \ell$. Thus, the columns in $\begin{pmatrix}
 B\\
 D
\end{pmatrix}$ should be a linear combination of the columns of $\begin{pmatrix}
 A\\
 C
\end{pmatrix}$ and thus $\begin{pmatrix}
 B\\
 D
\end{pmatrix} = \begin{pmatrix}
 A\\
 C
\end{pmatrix} v $. From this, using invertibility of $A$, we get that $D=CA^{-1}B $.} The condition $\wedge^{\ell+1}X=0$ means that all
$(\ell+1)$-minors of $X$ vanish. We claim that this implies
\[
        D=CA^{-1}B.
\]
Indeed, set
\[
P_1:=
\begin{pmatrix}
I_\ell & 0\\
-CA^{-1} & I_{2i-\ell}
\end{pmatrix},
\qquad
P_2:=
\begin{pmatrix}
I_\ell & -A^{-1}B\\
0 & I_{2i-\ell}
\end{pmatrix}.
\]
Then $P_1$ and $P_2$ are invertible, and
\[
X':=P_1XP_2
=
\begin{pmatrix}
A & 0\\
0 & D-CA^{-1}B
\end{pmatrix}.
\]
Since $P_1$ and $P_2$ are invertible, the vanishing of all
$(\ell+1)$-minors of $X$ implies the vanishing of all
$(\ell+1)$-minors of $X'$. Put $S:=D-CA^{-1}B.$ Let $s$ be any entry of $S$. Taking the $(\ell+1)$-minor of $X'$
given by the whole $A$-block together with the row and column containing
$s$, we obtain $\det(A)s = 0$. Since $A$ is invertible on $D(\Delta_\ell)$, we get $s=0$. Thus we get $S=0$, i.e. $D=CA^{-1}B$.

Define $U:=CA^{-1},\, V:=A^{-1}B.$ Then
\[
X=
\begin{pmatrix}
I_\ell\\
U
\end{pmatrix}
A
\begin{pmatrix}
I_\ell & V
\end{pmatrix}.
\]
Also, put $F:=
\begin{pmatrix}
I_\ell\\
U
\end{pmatrix},
\,
G:=
\begin{pmatrix}
I_\ell & V
\end{pmatrix}.$ Then $X=FAG$. We claim that the two symplectic equations translate to
\[
X^tJ_{2i}X=0
\quad\Longleftrightarrow\quad
F^tJ_{2i}F=0,
\]
and
\[
XJ_{2i}X^t=0
\quad\Longleftrightarrow\quad
GJ_{2i}G^t=0.
\]
It's enough to show the first equivalence and the second follows similarly. For simplicity, set $E:=F^tJ_{2i}F$. Then
\[
0 = X^tJ_{2i}X
=
G^tA^tEAG =\begin{pmatrix}
A^t EA  & A^t EB\\
B^t EA & B^t EB
\end{pmatrix} .
\]
From $A^t EA = 0 $ we deduce $E =0$ and so $F^tJ_{2i}F =0$. The converse is immediate. Therefore
\[
(R_i(\ell))_{\Delta_\ell}
\simeq
\frac{
k[A,\det(A)^{-1},U,V]
}{
\left(
F^tJ_{2i}F,\,
GJ_{2i}G^t
\right)
}.
\]
Equivalently,
\[
D(\Delta_\ell) =\operatorname{Spec}(R_i(\ell)_{\Delta_\ell})
\simeq
\mathcal U_\ell^{\operatorname{col}}
\times_k
\GL_{\ell,k}
\times_k
\mathcal U_\ell^{\operatorname{row}},
\]
where
\[
\mathcal U_\ell^{\operatorname{col}}
:=
\operatorname{Spec}
\frac{k[U]}{(F^tJ_{2i}F)}
=
\left\{F=\binom{I_\ell}{U}:F^tJ_{2i}F=0\right\},
\]
and
\[
\mathcal U_\ell^{\operatorname{row}}
:=
\operatorname{Spec}
\frac{k[V]}{(GJ_{2i}G^t)}
=
\left\{G=(I_\ell\;V):GJ_{2i}G^t=0\right\}.
\]
The scheme $D(\Delta_\ell)$ is smooth and irreducible of dimension $\ell (4i -2\ell+1)$. To see the dimension first set $F:=
\begin{pmatrix}
I_\ell\\
U_1\\
U_2\end{pmatrix}$ where $U_1$ is of size $(2i-2\ell)\times \ell$ and $U_2$ is of size $\ell \times \ell$. Note that $F^tJF=0$ translates to 
\[
\begin{pmatrix}
I_\ell\\
U_1\\
U_2\end{pmatrix}^t \begin{pmatrix}
0 & 0 &H_\ell\\
0 &  J_{2i -2\ell}&0\\
-H_\ell & 0 & 0\end{pmatrix} \begin{pmatrix}
I_\ell\\
U_1\\
U_2\end{pmatrix} =0
\]
which equals to $ H_{\ell}U_2-U_2^t H_{\ell} +U^t_1J_{2i-2\ell}U_1 =0 $. Thus $ \{F=(I_\ell,U):F^tJF=0\} \simeq \mathbb{A}_k^{\ell(2i-\ell)-\ell(\ell-1)/2}$. Similarly, $\{G=(I_{\ell},V):GJG^t=0\} \simeq \mathbb{A}_k^{\ell(2i-\ell)-\ell(\ell-1)/2}$. So, $\text{dim}(R_i(\ell))_{\Delta_\ell} = 2\left(\ell(2i-\ell)-\frac{\ell(\ell-1)}{2}\right)+\ell^2
=
\ell(4i-2\ell+1)$. 

Since $\Delta_\ell$ is not a zero divisor in $R_i(\ell)$, localization gives an
injection $R_i(\ell)\hookrightarrow (R_i(\ell))_{\Delta_\ell}$. The latter ring is a domain by the above and thus $R_i(\ell)$ is a domain.
\end{proof}

As a consequence of the preceding stratification and affine chart description, we obtain the following moduli-theoretic description of the Schubert varieties appearing in Theorem~\ref{thm:symp-max-schubert}.

\begin{prop}\label{sympschubertmoduli}
Assume that $1\leq 2i\leq  g$. For $0\leq \ell\leq i$, the Schubert variety
 $S_\ell:=\overline{\Mloc_i(\ell)}\subset \Mloc_{i,k}$ represents the functor which sends a $k$-algebra $R$ to the set of pairs $(\calF_i,\calF_{-i})$ satisfying the following conditions:
\begin{altenumerate}
    \item $\calF_{\pm i}\subset \Lambda_{\pm i}\otimes_{\CO_{F_0}}R$ is, Zariski
    locally on $\Spec R$, a direct summand of rank $g$;

    \item $(\langle\ ,\ \rangle\otimes 1)(\calF_i,\calF_{-i})=0$;

    \item if $ \iota_1:\Lambda_{-i}\longrightarrow \Lambda_i,\,
        \iota_2:\Lambda_i\longrightarrow \pi_0^{-1}\Lambda_{-i}$ denote the natural inclusion maps, then after base change to $R$ one has
    \[
        \iota_1(\calF_{-i})\subset \calF_i,\qquad
        \iota_2(\calF_i)\subset \pi_0^{-1}\calF_{-i};
    \]

    \item $\wedge^{\ell+1}
        \bigl(\iota_2:\calF_i\longrightarrow \pi_0^{-1}\calF_{-i}\bigr)=0.$
\end{altenumerate}
\end{prop}

\begin{proof}
Let $T_\ell\subset \Mloc_{i,k}$ be the closed subscheme defined by conditions
$(1)$--$(4)$ in the statement. By Theorem~\ref{thm:symp-max-schubert}, the
Schubert cells in $\Mloc_{i,k}$ are described set-theoretically by the rank
condition in Definition~\ref{def:symp-rank-strata}. Hence $T_\ell$ has the same
underlying topological space as $ S_\ell$. It remains to show that $T_\ell$ is reduced.

Since the $\calG_i$-translates of $U_i$ cover $\Mloc_i$, it is enough to prove
that $U_i\cap T_\ell$ is reduced. By the proof of
Proposition~\ref{prop:symp-affine-chart}, the condition
\[
        \wedge^{\ell+1}
        \bigl(\iota_2:\calF_i\longrightarrow \pi_0^{-1}\calF_{-i}\bigr)=0
\]
amounts to $\wedge^{\ell+1}X=0$ on $U_i\cap T_\ell$. Hence
\[
        U_i\cap T_\ell
        \simeq
        \mathbb A_k^{\frac{(g-2i)(g+2i+1)}{2}}
        \times_k
        \Spec R_i(\ell), \text{  with }   R_i(\ell):=
        \frac{k[X]}
        {
        \left(
        XJ_{2i}X^t,\;
        X^tJ_{2i}X,\;
        \wedge^{\ell+1}X
        \right)
        }. 
\]
By Proposition~\ref{IntegralDomainSymplectic}, the ring $R_i(\ell)$ is integral,
hence reduced. Thus $T_\ell$ is reduced and the claim follows.
\end{proof}

\section{Semi-stable resolution for symplectic local models}
\label{symp-semistable-resolution}

In this section, we construct a semi-stable resolution of the symplectic local model in the
maximal parahoric case by a similar sequence of blow-ups as in the orthogonal case in Part I. 

Let $F/F_0$ be a ramified quadratic extension, and choose a uniformizer $\pi\in F$ such that $\pi^2=-\pi_0$.
From now on we work after base change to $\calO_F$. We
write
\[
\Mloc_{i,\calO_F}:=\Mloc_i\otimes_{\calO_{F_0}}\calO_F.
\]
By Theorem~\ref{thm:symp-max-schubert}(2), we may assume that
$I=\{i\}$ with $1\leq i\leq \lfloor g/2\rfloor$. The case $i=0$ is excluded,
since $\Mloc_0$ is already smooth. For $0\leq \ell\leq i$, let
\[
S_\ell:=\overline{\Mloc_i(\ell)}\subset \Mloc_{i,k}
\]
be the corresponding Schubert variety. By Theorem~\ref{thm:symp-max-schubert},
the special fiber has a chain of Schubert varieties
\[
S_0\subset S_1\subset\cdots\subset S_i=\Mloc_{i,k}.
\]

\begin{thm}\label{thm:symp-semistable-resolution}
Assume that $I=\{i\}$ with $1\leq i\leq \lfloor g/2\rfloor$. Let
\[
\Mloc_{i,\calO_F}=Z_0
\xleftarrow{b_1}
Z_1
\xleftarrow{b_2}
\cdots
\xleftarrow{b_i}
Z_i
\]
be the sequence of blow-ups such that $b_1$ is the blow-up of
$Z_0=\Mloc_{i,\calO_F}$ along $S_0$, and for $2\leq \ell\leq i$, the morphism $b_\ell:Z_\ell\longrightarrow Z_{\ell-1}$ is the blow-up along the strict transform of $S_{\ell-1}$ in $Z_{\ell-1}$.
Then $Z_i$ is semi-stable over $\calO_F$.
\end{thm}

By Proposition~\ref{prop:symp-affine-chart}, there exists an open affine
subscheme $U_i\subset \Mloc_{i,\calO_F}$ containing a point
$\ast\in S_0$ such that
\[
U_i
\simeq
\mathbb A_{\calO_F}^{\frac{(g-2i)(g+2i+1)}{2}}
\times_{\calO_F}
\Spec R_i,  \text{ where } R_i=
\frac{\calO_F[X]}
{
\left(
XJ_{2i}X^t+\pi^2J_{2i},\;
X^tJ_{2i}X+\pi^2J_{2i}
\right)
}.
\]
Moreover, for $0\leq \ell<i$, the intersection $U_i\cap S_\ell$ is cut out
on the $\Spec R_i$-factor by $I_\ell=(\pi,\wedge^{\ell+1}X).$ Since the $\mathcal G_i$-translates of $U_i$ cover $\Mloc_{i,\mathcal O_F}$
and the centers $S_\ell$ are $\mathcal G_i$-stable, it is enough to prove the
corresponding affine statement on $U_i$.

\begin{prop}\label{prop:symp-affine-semistable}
Let
\[
\calX_0:=\Spec R_i,
\qquad
I_\ell:=(\pi,\wedge^{\ell+1}X)\subset R_i
\qquad
(0\leq \ell<i).
\]
Consider the sequence of blow-ups
\[
\calX_0
\xleftarrow{c_1}
\calX_1
\xleftarrow{c_2}
\cdots
\xleftarrow{c_i}
\calX_i,
\]
where $c_1$ is the blow-up along $V(I_0)$, and for $2\leq \ell\leq i$, the
morphism $c_\ell:\calX_\ell\longrightarrow \calX_{\ell-1}$ is the blow-up along the strict transform of $V(I_{\ell-1})$ in
$\calX_{\ell-1}$. Then $\calX_i$ is semi-stable over $\calO_F$.
\end{prop}

To prove Proposition~\ref{prop:symp-affine-semistable}, we introduce a family of auxiliary affine schemes. For $0\leq d,\ell\leq i$, let
\[
\calU(d,\ell):=\Spec \calR(d,\ell),
\]
where
\[
\calR(d,\ell):=
\frac{
\calO_F[X,T_0,T_1,\ldots,T_\ell]
}{
\left(
XJ_{2d}X^t+T_0^2J_{2d},\;
X^tJ_{2d}X+T_0^2J_{2d},\;
\pi-T_0T_1\cdots T_\ell
\right)
}.
\]
Here $X$ is a $2d\times 2d$ matrix and
\[
J_{2d}=
\begin{pmatrix}
0&H_d\\
-H_d&0
\end{pmatrix}.
\]
Observe that $\calU(i,0)\simeq \Spec R_i.$ The proof of the following proposition is same as that of Proposition~\ref{prop flat R(d,l)}, using
Proposition~\ref{IntegralDomainSymplectic} in place of Proposition~\ref{prop:special-fiber-schubert-centers}.
We omit the details.
\begin{prop}
The ring $\mathcal{R}(d, \ell)$ is flat over $\mathcal{O}_F$ and integral of (total) dimension $d(2d + 1) + \ell + 1$.    
\end{prop} 

\subsection{The blow-up of $\calU(d,\ell)$}

Consider the blow-up
\[
\Bl_{\calI_0}(\calU(d,\ell))\longrightarrow \calU(d,\ell)
\]
along the ideal $\calI_0=(X,T_0)$. Let $y_{a,b}$ and $\alpha$ denote the homogeneous degree-one elements
corresponding to $x_{a,b}$ and $T_0$, respectively. Then the blow-up is covered
by the affine charts $D_+(y_{a,b})$ and $D_+(\alpha)$.

\begin{prop}\label{prop:symp-y-chart}
For every $1\leq a,b\leq 2d$, the affine chart $D_+(y_{a,b})$ is isomorphic to
\[
\calU(d-1,\ell+1)\times_{\calO_F}\mathbb A_{\calO_F}^{4d-2}.
\]
\end{prop}

\begin{proof}
For simplicity, let $N=2d$, $J=J_{2d}$ and put $t^\ast=N+1-t$. We write 
\[
\varepsilon_t=
\begin{cases}
1, & 1\leq t\leq d,\\
-1, & d+1\leq t\leq 2d.
\end{cases}
\]
Thus $\varepsilon_{t^\ast}=-\varepsilon_t$. 
On the affine chart
$D_+(y_{i_0,j_0})$, with $y_{i_0,j_0}=1$, we have
\[
X=x_{i_0,j_0}Y,\qquad T_0=x_{i_0,j_0}\alpha.
\]
Substituting these expressions into the defining equations, we obtain
\[
YJY^t+\alpha^2J=0,\qquad
Y^tJY+\alpha^2J=0,
\]
together with $\pi=\alpha x_{i_0,j_0}T_1\cdots T_\ell.$ For simplicity, write $i=i_0,\, j=j_0$ and set
\[
I'\coloneqq \{1,\ldots,N\}\setminus\{i,i^\ast\},\qquad
J'\coloneqq \{1,\ldots,N\}\setminus\{j,j^\ast\}.
\]
The row equations and column equations in $YJY^t+\alpha^2J=0$ and $Y^tJY+\alpha^2J=0$ give
\[
\sum_{t=1}^N \varepsilon_t y_{a,t}y_{b,t^\ast}
=
-\alpha^2\varepsilon_a\delta_{b,a^\ast},\quad
\sum_{s=1}^N \varepsilon_s y_{s,a}y_{s^\ast,b}
=
-\alpha^2\varepsilon_a\delta_{b,a^\ast}.
\]
 
For $s\in I'$, the row equation $(r_i,r_s)=0$ gives
\[
\varepsilon_j y_{s,j^\ast}
+\varepsilon_{j^\ast}y_{i,j^\ast}y_{s,j}
+
\sum_{t\in J'}\varepsilon_t y_{i,t}y_{s,t^\ast}=0.
\]
Since $\varepsilon_{j^\ast}=-\varepsilon_j$, this gives
\begin{equation}\label{1}
y_{s,j^\ast}
=
y_{i,j^\ast}y_{s,j}
-
\varepsilon_j
\sum_{t\in J'}\varepsilon_t y_{i,t}y_{s,t^\ast}.
\end{equation}

Similarly, for $t\in J'$, the column equation $(c_j,c_t)=0$ gives
\begin{equation}\label{2}
y_{i^\ast,t}
=
y_{i^\ast,j}y_{i,t}
-
\varepsilon_i
\sum_{s\in I'}\varepsilon_s y_{s,j}y_{s^\ast,t}.
\end{equation}

Finally, the row equation $(r_i,r_{i^\ast})=-\alpha^2\varepsilon_i$
gives
\begin{equation}\label{3}
y_{i^\ast,j^\ast}
=
y_{i,j^\ast}y_{i^\ast,j}
-
\varepsilon_j
\sum_{t\in J'}\varepsilon_t y_{i,t}y_{i^\ast,t^\ast}
-
\alpha^2\varepsilon_i\varepsilon_j.
\end{equation}

Now define, for $s\in I'$ and $t\in J'$,
\begin{equation}\label{4}
z_{s,t}:=y_{s,t}-y_{s,j}y_{i,t}.
\end{equation}
Let $Z=(z_{s,t})_{s\in I',\ t\in J'}$. We claim that the remaining
row and column equations are exactly the symplectic equations for $Z$. Indeed, put
$
C_s:=\sum_{t\in J'}\varepsilon_t y_{i,t}y_{s,t^\ast}$.
Then (\ref{1}) traslates to
\[
y_{s,j^\ast}=y_{i,j^\ast}y_{s,j}-\varepsilon_jC_s.
\]
Using (\ref{4}), we compute for $s,r\in I'$:
\[
\begin{aligned}
&\sum_{t\in J'}\varepsilon_t z_{s,t}z_{r,t^\ast}
=
\sum_{t\in J'}\varepsilon_t
\bigl(y_{s,t}-y_{s,j}y_{i,t}\bigr)
\bigl(y_{r,t^\ast}-y_{r,j}y_{i,t^\ast}\bigr)  \\
&=
\sum_{t\in J'}\varepsilon_t y_{s,t}y_{r,t^\ast}
-y_{r,j}\sum_{t\in J'}\varepsilon_t y_{s,t}y_{i,t^\ast}
-y_{s,j}\sum_{t\in J'}\varepsilon_t y_{i,t}y_{r,t^\ast}+ y_{s,j}y_{r,j}\sum_{t\in J'}\varepsilon_t y_{i,t}y_{i,t^\ast}.
\end{aligned}
\]
The term
$
\sum_{t\in J'}\varepsilon_t y_{i,t}y_{i,t^\ast}
$
equals to zero because the form is alternating. Moreover,
\[
\sum_{t\in J'}\varepsilon_t y_{s,t}y_{i,t^\ast}
=
-\sum_{t\in J'}\varepsilon_t y_{i,t}y_{s,t^\ast}
=
-C_s.
\]
Therefore
\[
\sum_{t\in J'}\varepsilon_t z_{s,t}z_{r,t^\ast}
=
\sum_{t\in J'}\varepsilon_t y_{s,t}y_{r,t^\ast}
+y_{r,j}C_s-y_{s,j}C_r.
\]
On the other hand, the $(s,r)$-entry of $YJY^t$ is
\[
\begin{aligned}
\sum_{t=1}^N\varepsilon_t y_{s,t}y_{r,t^\ast}
&=
\sum_{t\in J'}\varepsilon_t y_{s,t}y_{r,t^\ast}
+\varepsilon_j y_{s,j}y_{r,j^\ast}
-\varepsilon_j y_{s,j^\ast}y_{r,j}  \\
&=
\sum_{t\in J'}\varepsilon_t y_{s,t}y_{r,t^\ast}
+y_{r,j}C_s-y_{s,j}C_r.
\end{aligned}
\]
Thus
\[
\sum_{t\in J'}\varepsilon_t z_{s,t}z_{r,t^\ast}
=
\sum_{t=1}^N\varepsilon_t y_{s,t}y_{r,t^\ast}.
\]
Since $YJY^t+\alpha^2J=0$, we obtain
\begin{equation}
\sum_{t\in J'}\varepsilon_t z_{s,t}z_{r,t^\ast}
=
-\alpha^2\varepsilon_s\delta_{r,s^\ast}
\end{equation}
for all $s,r\in I'$. Equivalently,
\[
ZJ_{2d-2}Z^t+\alpha^2J_{2d-2}=0.
\]

The column equations are treated in the same way. Put
\[
D_t:=\sum_{s\in I'}\varepsilon_s y_{s,j}y_{s^\ast,t}.
\]
Then (\ref{2}) says
\[
y_{i^\ast,t}=y_{i^\ast,j}y_{i,t}-\varepsilon_iD_t.
\]
Using again the definition of $z_{s,t}$, one obtains, for $t,u\in J'$,
$
\sum_{s\in I'}\varepsilon_s z_{s,t}z_{s^\ast,u}
=
\sum_{s=1}^N\varepsilon_s y_{s,t}y_{s^\ast,u}.
$
Hence, by the equation $Y^tJY+\alpha^2J=0$,
\begin{equation}
\sum_{s\in I'}\varepsilon_s z_{s,t}z_{s^\ast,u}
=
-\alpha^2\varepsilon_t\delta_{u,t^\ast}
\end{equation}
for all $t,u\in J'$. Equivalently,
\[
Z^tJ_{2d-2}Z+\alpha^2J_{2d-2}=0.
\]

Consequently the chart $D_+(y_{i_0,j_0})$ is isomorphic to
\[
Spec
\frac{
\mathcal O[Z,\alpha,x_{i_0,j_0},T_1,\ldots,T_\ell]
}{
\left(
ZJ_{2d-2}Z^t+\alpha^2J_{2d-2},\,
Z^tJ_{2d-2}Z+\alpha^2J_{2d-2},\,
\pi-\alpha x_{i_0,j_0}T_1\cdots T_\ell
\right)
}
\times_{\mathcal O}\mathbb A^{4d-2}_{\mathcal O}.
\]
Here the affine-space factor has coordinates
\[
y_{i,t}\quad (t\in J'),\qquad
y_{s,j}\quad (s\in I'),\qquad
y_{i,j^\ast},\qquad
y_{i^\ast,j}.
\]
Thus its dimension is $(2d-2)+(2d-2)+2=4d-2.$ After renaming $\alpha=T_0',\, x_{i_0,j_0}=T_{\ell+1}',$
we get $\pi=T_0'T_1\cdots T_\ell T_{\ell+1}'.$ Therefore $D_+(y_{i_0,j_0})
\simeq
U(d-1,\ell+1)\times_{\mathcal O}\mathbb A^{4d-2}_{\mathcal O}.$
\end{proof}

\begin{prop}\label{prop:symp-alpha-chart}
The affine chart $D_+(\alpha)$ is isomorphic to
\[
C_d\times_{\calO_F}
\Spec \frac{\calO_F[T_0,\ldots,T_\ell]}
{\left(\pi-T_0T_1\cdots T_\ell\right)},
\]
where
\[
C_d:=
\Spec
\frac{\calO_F[Y]}
{
\left(
YJ_{2d}Y^t+J_{2d},\;
Y^tJ_{2d}Y+J_{2d}
\right)
}.
\]
In particular, $C_d$ is smooth over $\calO_F$, and $D_+(\alpha)$ is
semi-stable over $\calO_F$. If $\ell=0$, then $D_+(\alpha)$ is smooth over
$\calO_F$.
\end{prop}

\begin{proof}
On the chart $D_+(\alpha)$, we may set $\alpha=1$. Then \[
XJ_{2d}X^t+T_0^2J_{2d}=0,\qquad
X^tJ_{2d}X+T_0^2J_{2d} =0 ,\qquad
\pi=T_0T_1\cdots T_\ell
\] translate to 
\[
YJ_{2d}Y^t+J_{2d}=0,
\qquad
Y^tJ_{2d}Y+J_{2d}=0, \qquad \pi=T_0T_1\cdots T_\ell.
\]
Hence $D_+(\alpha)$ is isomorphic to the closed subscheme
\[
C_d\times_{\mathcal O_F}
\Spec \,\mathcal O_F[T_0,\ldots,T_\ell]/(\pi-T_0T_1\cdots T_\ell).
\]
It remains to note that $C_d$ is smooth over $\mathcal O_F$. Let $S=\operatorname{diag}((-1)^{(d)},(1)^{(d)})$. Then
\[
SJ_{2d}S^t=-J_{2d}.
\]
Thus multiplication by $S^{-1}$ identifies $C_d$ with the usual
symplectic group:
\[
C_d\xrightarrow{\ \sim\ }\text{Sp}_{2d,\mathcal O_F},
\qquad
Y\longmapsto S^{-1}Y.
\]
Indeed, if $YJ_{2d}Y^t=-J_{2d}$, then
\[
(S^{-1}Y)J_{2d}(S^{-1}Y)^t
=
S^{-1}(-J_{2d})(S^{-1})^t
=
J_{2d}.
\]
Conversely, if $A\in \text{Sp}_{2d}$, then $Y=SA$ satisfies
\[
YJ_{2d}Y^t=SJ_{2d}S^t=-J_{2d}.
\]
Therefore $C_d$ is smooth over $\calO_F$ of relative dimension
$d(2d+1)$.
\end{proof}
\begin{proof}[Proof of Proposition~\ref{prop:symp-affine-semistable}]
The proof is the same as the proof of Proposition~\ref{prop:qs-affine-semistable},
replacing Lemmas~\ref{lemma chart y_{i,j}} and
\ref{lem:alpha-chart-semistable} by
Propositions~\ref{prop:symp-y-chart} and
\ref{prop:symp-alpha-chart}. Thus, after $i$ successive blow-ups, every affine
chart is semi-stable either coming from an $\alpha$-chart as in Proposition \ref{prop:symp-alpha-chart} or isomorphic to
\[
\mathbb A_{\calO_F}^{2i^2}\times_{\calO_F}
\Spec\frac{\calO_F[T_0,\ldots,T_i]}{(\pi-T_0\cdots T_i)}.
\]
Hence, $\calX_i$ is semi-stable over $\calO_F$.
\end{proof}

\section{Semi-stable resolution for Weil restriction symplectic case}\label{section ss sym Res}

In this section, we construct a semi-stable resolution for the Weil restriction symplectic case. We continue with the notation of \S \ref{section ss ortho Res}. Suppose that $(B, *, V, \bb\ ,\ \pp)$ is a local PEL-datum over $F_0$ as in \S \ref{sec-splitting}.

We assume that in (\ref{Bprod}), $I_1=\{1\}$ is a singleton and of type (C), i.e., 
\[
B\simeq M_{m_1}(F_1)
\]
for any positive integer $m_1$, and the involution  $*:M_{m_1}(F_1)\rightarrow M_{m_1}(F_1)$ is given by $b^*=b^t$. In particular, $*$ acts trivially on $F_1$, which is identified with scalar matrices in $M_{m_1}(F_1)$.  Here $F_1$ is a finite extension  of $F_0$. As in \S \ref{section ss ortho Res}, we define the fields $L$ and $K$, and the set of field embeddings $\Sigma=\cbra{\sigma_j^l\colon F_1\ra L}$. 

Let $\calL=\Lambda_{2g\ZZ\pm i}$ be a self-dual multichain of $\calO_B$--lattices in $V$, and let $\calG$ be the neutral component of the similitude automorphism group of $\calL$. By Morita equivalence, $V\simeq V_1^{m_1}$  where $V_1$ is an $F_1$-vector space of even dimension. We assume that  $\dim V_1=2g$. Since the duality isomorphism $\Hom_{F_1} (V_1, F_1)\simeq \Hom_{F_0}(V_1,F_0)$ given by composing with the trace $\Tr_{F_1/F_0}: F_1\rightarrow F_0$, there exists a unique non-degenerate alternating form $\bb\ ,\ \pp': V_1\times V_1\rightarrow F_1$ such that
\[
\bb x,y\pp=\Tr_{F_1/F_0}(\delta\cdot \bb x,y\pp'),
\] 
for all $x, y\in V_1$, where $\delta$ is an $\calO_{F_1}$-generator of the inverse different $\mathcal{D}^{-1}_{F_1/F_0}$. Let $G$ denote the group over $F_0$ whose set of $R$-points ($R$ is any $F_0$-algebra) is the following \begin{flalign*}
    G(R)\coloneqq \cbra{g\in \GL_{B\otimes_{F_0}R}(V\otimes_{F_0}R)\ |\ \pair{gv,gw}=c(g)\pair{v,w}, c(g)\in R\cross }.
\end{flalign*}
Let $H=\GSp(V_1,\bb\ ,\ \pp')$ denote the symplectic similitude group over $F_1$, whose set of $R$-points ($R$ is any $F_1$-algebra) is \begin{flalign*}
    H(R)\coloneqq \cbra{g\in \GL(V_1\otimes_{F_1}R)\ |\ \bb gx,gy\pp'=c'(g)\bb x,y\pp', c'(g)\in {R'}^\times}.
\end{flalign*}

\begin{lemma}\label{lm-weilres-gpG-symp}
    The group $G$ is isomorphic to the subgroup $G'$ of $\Res_{F_1/F_0}H$ whose similitude character is defined over $F_0$. More precisely, for any $F_0$-algebra $R$, \begin{flalign*}
        G(R)=H'(R) \coloneqq \cbra{h\in H(R\otimes_{F_0}F_1)\ |\ c'(h)\in R\cross }.
    \end{flalign*}
\end{lemma}

We omit the proof here since it is similar to Lemma \ref{lm-weilres-gpG-ortho}. We assume that the matrix with respect to the standard basis of the symplectic form $\bb \ , \ \pp'$  is $J_{2g}$. Let $\mu$ denote the minuscule cocharacter of the standard maximal torus in $\GSp(V_1, \bb \ ,\ \pp')$. For each embedding $\sigma_j^l\in \Sigma$, we fix a minuscule cocharacter $\mu_{\sigma_j^l}=(1^{(g)}, 0^{(g)}))$, and let $\mu: \mathbb{G}_{m, \bar{F}_0}\rightarrow G_{\bar{F}_0}$ denote the geometric cocharacter of $G$ whose $\sigma_j^l$-component in is $\mu_{\sigma_j^l}$. Let $E\subset L$ denote the reflex field of $\{\mu\}$ with the ring of integers $\calO_E$. Thus, by Definition  \ref{def-naivelocmod} (resp. Definition \ref{def-naivesplmod}), we obtain the naive local model $\Mnaive(\calL, \mu)$ over $\calO_E$ (resp. the naive splitting model $\M^{\rm nspl}_\calL$ over $\calO_L$). Note 
that in the symplectic case, the naive local models and naive splitting models coincide with the local models (Definition \ref{def:symp-local-model}) and the splitting model (\cite[\S 9]{PR2}) respectively. Similar to (\ref{eq-local model diagram spl-ortho}), we have a local model diagram in the following:
\begin{equation}\label{eq-local model diagram spl-symp}
\begin{tikzcd}[column sep={6em, between origins}, row sep={3em, between origins}]
&  \arrow[ld, "\alpha_1"']\tilde{\M}^{\rm spl}_\calL\arrow[rd, "\alpha_2"]&\\
\M^{\rm spl}_\calL   &&   \prod_{\sigma_j^l\in\Sigma} \Mloc_i,
\end{tikzcd}
\end{equation}
where $\alpha_1$ is a  $\mathcal{H}$-torsor and $\alpha_2$ is a smooth $\mathcal{H}$-equivariant morphism. 

\begin{thm}\label{thm ss Weil symplec}
Let $F_1$ be a finite separable extension of $F_0$. Assume that $I=\{i \}$ with $1\leq i\leq \lfloor g/2\rfloor$. After the base change to $\calO_K$, there exists a projective
$\calG_{\calO_K}$-equivariant morphism
\[
\M^{\rm ss}(\calL, \mu)\rightarrow \Mloc(\calL, \mu)\otimes_{\calO_E} \calO_K,
\]
which is an isomorphism on the generic fibers and whose source is semi-stable over $\calO_K$.    
\end{thm}

\begin{proof}
We perform a similar proof as in Theorem \ref{thm ss Weil ortho}. For each embedding $\sigma_j^l\in\Sigma$, we obtain a semi-stable resolution $Z_{\sigma_j^l}\rightarrow \Mloc_i\otimes_{\sigma_j^l} \calO_K$ by Theorem \ref{thm:symp-semistable-resolution}. By using Hartl's method again, we obtain a semi-stable resolution 
\[
\pi: Z\rightarrow \prod_{\sigma_j^l} \Mloc_i\otimes \calO_K.
\]
Using the linear modification \cite[\S 2]{P}, we obtain a semi-stable model $\M^{\rm ss}(\calL, \mu)$ over $\calO_K$ such that the morphsim $\ga: \M^{\rm ss}(\calL, \mu)\rightarrow \Mspl_\calL$ is $\calG$-equivalent. The remainder of the argument proceeds exactly as in the proof of Theorem \ref{thm ss Weil ortho}.
\end{proof}

\part{Applications} \label{part-application}
\section{Applications to Shimura varieties}
\label{applications_shimura}

We keep the notation of the Introduction \S \ref{Intro}. Thus $(\bG,X)$ is the Shimura datum
attached to the PEL datum $(\bB,*,\bV,\pair{\ ,\ })$ satisfying the assumption $(\dagger)$, and
$\rK=\rK_p\rK^p$ with $\rK_p\sset \bG(\BQ_p)$ \dfn{maximal parahoric} and $\rK^p$ sufficiently small.
Let $v\mid p$ be the fixed place of the reflex field $\bE$, and put
$E=\bE_v$. 

\subsection{Semi-stable integral models}
We denote by $\sS_\rK$ the canonical integral model over $\CO_E$.
Let $(G,\CG,\mu)$ be the local model triple attached to the maximal parahoric level
$\rK_p$, and write $\RM^\loc_{\CG,\mu}$ for the corresponding canonical local
model. Recall that there exists a local model diagram
\begin{equation}\label{LMDapplication}
\begin{tikzcd}[column sep={5em, between origins}, row sep={2em, between origins}]
&  \arrow[ld, "\alpha"']\wt{\sS}_\rK\arrow[rd, "\beta"]&\\
\sS_\rK   &&   \RM^\loc_{\CG,\mu},
\end{tikzcd}
\end{equation}
where $\alpha$ is a $\CG$-torsor and $\beta$ is smooth and
$\CG$-equivariant. Equivalently, we have a relatively representable smooth morphism of algebraic stacks
\[
        \lambda\colon
        \sS_\rK
        \longrightarrow
        [\RM^\loc_{\CG,\mu}/\CG].
\]

{
\begin{remark} \label{rmk-lmdpara}
    The local model diagram in \eqref{LMDapplication} is constructed in \cite[(8.12)]{PZ} when $\rK_p$ is the stabilizer subgroup of a self-dual lattice chain. It descends to the corresponding diagram for the associated parahoric subgroup by \cite[Proposition 4.3.3]{daniels2026conjecture}. Note that the proof in \cite[\S A.3]{daniels2026conjecture} shows that the local model diagram in \cite[(8.12)]{PZ} is schematic.
\end{remark}
}

\begin{corollary}\label{applicationsFlatintegralmodels}
The canonical integral model $\sS_\rK$ is flat and normal over $\CO_E$, and
its special fiber is reduced.
\end{corollary}
\begin{proof}
The statement follows by the local model diagram (\ref{LMDapplication}) and Theorem~\ref{thm-repre}.    
\end{proof}

\begin{thm}\label{applications-semistable}
Assume $(\dagger)$. Then, after a finite extension $E'/E$, there exists a semi-stable integral model $\sS_\rK^{\rm ss}$ over $\CO_{E'}$, together with a projective morphism 
\[
        \sS_\rK^{\rm ss}
        \longrightarrow
        \sS_{\rK,\CO_{E'}}
\]
which is an isomorphism on the generic fiber. In particular, $\sS_\rK^{\rm ss}$ is regular and its special
fiber is a reduced divisor with normal crossings. 
\end{thm}
\begin{proof}
By discussion in \S \ref{sec-splitting}, under the assumption $(\dagger)$, factors of $\bB_{\BQ_p}$ fall into types (AL), (AU), (C) or (D). For types (C) or (D), by results in Part \ref{part-orth} and Part \ref{part-symp}, there exists a finite extension $E'$ of $E$ and a $\CG_{\CO_E'}$-equivariant projective morphism $\RM^{\rm ss}\ra \RM^\loc_{\CG,\mu}\otimes_{\CO_E}\CO_E'$, which is an isomorphism on the generic fiber, and $\RM^{\rm ss}$ has semi-stable reduction over $\CO_{E'}$. For types (AL) or (AU), the corresponding local models are smooth over $\CO_E$, and hence we may take $\RM^{\rm ss}=\RM^\loc_{\CG,\mu}$. Form the cartesian diagram
\[
\begin{matrix}
\sS_\rK^{\rm ss} & \longrightarrow & [\RM^{\rm ss}/\calG_{\CO_{E'}}] \\
\Big\downarrow && \Big\downarrow \\
\sS_{\rK,\CO_E'} & \longrightarrow & [\RM^\loc_{\CG,\mu}\otimes_{\CO_E}\CO_E'/\calG_{\CO_{E'}}].
\end{matrix}
\]
Then $\sS_\rK^{\rm ss}\longrightarrow \sS_{\rK,\CO_E'}$ is a linear modification in the sense of \cite[\S 2]{P}. Also, from the above cartesian diagram we deduce that there is a local model diagram for $\sS_\rK^{\rm ss}$ similar to (\ref{LMDapplication}) but with $\RM^\loc_{\CG,\mu}$ replaced by $\RM^{\rm ss}$; thus $\sS_\rK^{\rm ss}$ is semi-stable over $\CO_{E'}$.
\end{proof}
\begin{remark}\label{rem:KR-strata-centers}
Suppose that $\bG_{\BQ_p}= \GSp_{2g},$ or $ \GO(V,\psi)$ where $(V,\psi)$ is a non-degenerate symmetric space of $\BQ_p$-dimension $2g$. Let
\[
        \lambda:
        \sS_{\rK,\CO_{E'}}
        \longrightarrow
        [\RM^\loc_{\CG,\mu,\CO_{E'}}/\calG]
\]
be the morphism induced by the local model diagram, and let
\[
        \lambda_k:
        \sS_{\rK,k}
        \longrightarrow
        [\RM^\loc_{\CG,\mu,k}/\calG_k]
\]
be its restriction to the special fiber. Let
\[
        S_0\subset S_1\subset\cdots\subset S_{d-1}
        \subset \RM^\loc_{\CG,\mu,k}
\]
be the chain of Schubert varieties used as centers in the blow-up construction (see Theorem \ref{introthm-sstable}), and
let $C_j\subset S_j$ denote the corresponding Schubert cell. Following the
standard definition of the Kottwitz--Rapoport stratification via the local model diagram
\cite[\S3.2]{HeR}, we define
\[
        \sS_{\rK}(j)
        :=
        \lambda_k^{-1}\bigl([C_j/\calG_k]\bigr)
        \subset \sS_{\rK,k},
        \qquad 0\leq j\leq d-1.
\]
These are the Kottwitz--Rapoport strata corresponding to the Schubert cells
$C_j$. Since $S_j=\overline{C_j}$, the Schubert variety $S_j$ corresponds,
under the local model diagram, to the closure
\[
        \overline{\sS_{\rK}(j)}
        =
        \lambda_k^{-1}\bigl([S_j/\calG_k]\bigr)
        \subset \sS_{\rK,k}.
\]
Consequently, since blowing up commutes with flat base change, the morphism
\[
        \sS_{\rK}^{\rm ss}
        \longrightarrow
        \sS_{\rK,\CO_{E'}}
\]
is obtained by successively blowing up the closures of these Kottwitz--Rapoport
strata. More precisely, if
\[
        \sS_{\rK,\CO_{E'}}=Y_0
        \xleftarrow{c_1}
        Y_1
        \xleftarrow{c_2}
        \cdots
        \xleftarrow{c_d}
        Y_d
\]
is the sequence in which $c_1$ is the blow-up of $Y_0$ along
$\overline{\sS_{\rK}(0)}$, and, for $2\leq j\leq d$, $c_j:Y_j\longrightarrow Y_{j-1}$ is the blow-up along the strict transform of
$\overline{\sS_{\rK}(j-1)}$, then $\sS_{\rK}^{\rm ss}=Y_d$.
\end{remark}

\begin{corollary}\label{coro-modiconj}
    Assume $(\dagger)$ and suppose that $\mathcal G$ is maximal parahoric. Then Conjecture \ref{conj-modif} holds after base change to $\CO_{E'}$.
\end{corollary}
\begin{proof}
    By the assumption $(\dagger)$ and the proof of Theorem \ref{applications-semistable}, we obtain Theorem \ref{IntroductionMainResultMloc}. In particular, after base change to $E'$, $\RM^\loc_{\CG,\mu}$ is resolved $\CG_{\CO_{E'}}$-equivariantly by a semi-stable model. 
\end{proof}

\subsection{Nearby cycles and \etale cohomologies}\label{NcyclesEcohomology}
Recall that there exists a decomposition in \eqref{decompBi} given by \[B=\bB_{\BQ_p}\simeq \prod_{[i]\in I_1/\sim_*}B_{[i]}, \]
where the factor $B_{[i]}$ is of type (AL), (AU), (C) or (D), under the assumption ($\dagger$). Following the notation in \S \ref{sec-splitting} (for $F_0=\BQ_p$), let $F_i$ denote the center of $B_{[i]}$. Let $L_i$ be the Galois closure of $F_i$ in $\ol{\BQ}_p$. Let $L_i^0$ be the (unique) unramified quadratic extension of $L_i$.

For each $[i]$, let $E'_i$ be the field extension of $F_i$ defined as follows.
\begin{altenumerate}
    \item If $B_{[i]}$ is of type (AL) or (AU), then $E'_{[i]}\coloneqq L_i$;
    \item Suppose $B_{[i]}$ is of type (C). Let $\pi_i$ be a uniformizer of $F_i$. If $F_i=\BQ_p$, then set $E'_i\coloneqq L_i(\sqrt{-\pi_i})$. If $F_i\neq \BQ_p$, then set $E_i'\coloneqq L_i^0(\sqrt{-\pi_i})$.
    \item Suppose $B_{[i]}$ is of type (D). 
    Let $\pi_i$ be a uniformizer of $F_i$. If $F_i=\BQ_p$, then set $E_i'\coloneqq L_i(\sqrt{-\pi_i})$; and if $F_i\neq \BQ_p$, then set $E_i'\coloneqq L_i^0(\sqrt{-\pi_i})$.
\end{altenumerate}

Recall that $E$ is the (local) reflex field. Set \begin{flalign}
    E'\coloneqq \text{the composite of $E$ and $E_i'$ for all $[i]$}.   \label{eq-E'}
\end{flalign}
By Lemma \ref{lem-KL0} and the proof of Theorem \ref{applications-semistable}, the finite extension of $E$ used in Theorem \ref{applications-semistable} may be taken to be $E'$ in \eqref{eq-E'}. In particular, we obtain a semi-stable integral model $\sS^{\rm ss}_\rK$ over $\CO_{E'}$.

For the scheme $S\coloneqq \Spec \CO_{E'}$, we denote by $s$ (resp. $\eta$) the special (resp. generic) point of $S$. We choose a separable closure $\ol{\eta}$ of $\eta$. Then $\Gal(\ol{\eta}/\eta)\simeq\Gal(\ol{\BQ}_p/E')$. Denote by $\ol{S}$ the normalization of $S$ in $\ol{\eta}$. Write $\ol{s}$ for the special point of $\ol{S}$. For any scheme $X$ separated and of finite type over $S$, denote \begin{flalign*}
    \bari\colon X_{\ol{s}}\hookrightarrow X_{\ol{S}} \ \ resp.\ \barj\colon X_{\ol{\eta}}\hookrightarrow X_{\ol{S}}
\end{flalign*}
the natural closed (resp. open) immersions of the geometric special (resp. generic) fibers of $X/S$. Let $\ell\neq p$ be a rational prime. For any bounded constructible $\ell$-adic complex $\CF$ in $D_c^b(X_\eta,\ol{\BQ}_\ell)$, define the \dfn{nearby cycles sheaf} to be \begin{flalign*}
    R\Psi^X(\CF)\coloneqq \bari^*R\,\barj_*(\CF_{\ol{\eta}}),
\end{flalign*}
which is an $\ell$-adic complex in $D^b_c(X_{\ol{s}},\ol{\BQ}_\ell)$ together with a continuous action of $\Gal(\ol{\eta}/\eta)$. Sometimes we simply write $R\Psi(\CF)$ for $R\Psi^X(\CF)$. We refer to \cite[\S 10]{HainesBadReduction} (see also \cite[XIII]{deligne-katz1973sga7ii}) for more details. 

Let $$I_p\coloneqq \ker(\Gal(\ol{\eta}/\eta)\ra \Gal(\ol{s}/s))$$ denote the inertia group of $\Gal(\ol{\BQ}_p/E')$. 
For any scheme $X$ separated of finite type over $S$, denote by $$D^b_c(X\times_s\eta^{nr},\ol{\BQ}_\ell)$$ the category of bounded constructible $\ell$-adic complexes in $D^b_c(X_{\ol{s}},\ol{\BQ}_\ell)$ equipped with continuous actions of $I_p$ compactible with its action on $X_{\ol{s}}$; see \cite[page 185]{GortzHaines} for more details. Following \cite[Definition 5.4]{GortzHaines}, we say that an object $\CF\in D^b_c(X\times_s\eta^{nr},\ol{\BQ}_\ell)$ is \dfn{unipotent} if for any geometric point $y\in X_{\ol{s}}(\ol{\BF}_p)$ and any integer $r\geq 0$, the inertia group $I_p$ acts unipotently on the stalk of the cohomology sheaf $\CH^r\CF$ at $y$.

\begin{lemma}\label{lem-f*f!}
    Let $f\colon X\ra Y$ be a morphism between separated finite type $S$-schemes. Then $f$ induces two functors \begin{flalign*}
        Rf_{\ol{s},!},\ Rf_{\ol{s},*}\colon D^b_c(X\times_s\eta^{nr},\ol{\BQ}_\ell) \ra D^b_c(Y\times_s\eta^{nr},\ol{\BQ}_\ell).
    \end{flalign*}
    If $\CF\in D^b_c(X\times_s\eta^{nr},\ol{\BQ}_\ell)$ is unipotent, then $Rf_{\ol{s},!}\CF$ and $Rf_{\ol{s},*}\CF$ are also unipotent.
\end{lemma}
\begin{proof}
    The statement for $Rf_{\ol{s},!}\CF$ follows from \cite[Lemma 5.6 (1)]{GortzHaines}. Moreover, the proof in \loccit\  carries over verbatim to $Rf_{\ol{s},*}\CF$ after replacing the functor $f_!$ by $f_*$, showing that $Rf_{\ol{s},*}\CF$ is unipotent.

\end{proof}

\begin{thm}\label{prop:inertia-action}
Denote by $$f\colon \sS^{\rm ss}\ra \sS$$ the projective morphism in Theorem \ref{applications-semistable}. 
Write $R\Psi\coloneqq R\Psi^\sS\text{\ and\ } R\Psi^{\rm ss}\coloneqq R\Psi^{\sS^{\rm ss}}$. 

\begin{altenumerate}
    \item For every prime
$\ell\neq p$ and any integer $r\geq 0$, the inertia group $I_p$ acts trivially on $R^r\Psi^{\rm ss}(\ol{\BQ}_\ell)$, and $I_p$ acts unipotently on $R^r\Psi(\ol{\BQ}_\ell)$.
\item We have a $\Gal(\ol{\BQ}_p/E')$-equivariant isomorphism $Rf_{\ol{s},*} R\Psi^{{\rm ss}}(\ol{\mathbb Q}_\ell)\simeq R\Psi(\ol{\mathbb Q}_\ell)$. In particular, we have \begin{flalign*}
    &H^r_c(\sS^{\rm ss}_{\ol{s}},R\Psi^{\rm ss}(\ol{\BQ}_\ell))\simeq H^r_c(\sS_{\ol{s}},R\Psi(\ol{\BQ}_\ell)) \text{\ and\ } \\  &H^r(\sS^{\rm ss}_{\ol{s}},R\Psi^{\rm ss}(\ol{\BQ}_\ell))\simeq H^r(\sS_{\ol{s}},R\Psi(\ol{\BQ}_\ell))
\end{flalign*} as $\Gal(\ol{\BQ}_p/E')$-modules. Moreover, $I_p$ acts unipotently on the above \etale cohomologies.
   \item For each integer $r\geq 0$, we have \begin{flalign*}
    &H^r_c(\sS_{\ol{\eta}},\ol{\BQ}_\ell)\simeq H^r_c(\sS_{\ol{s}},R\Psi(\ol{\BQ}_\ell))\ \text{and\ }
    H^r(\sS_{\ol{\eta}},\ol{\BQ}_\ell)\simeq H^r(\sS_{\ol{s}},R\Psi(\ol{\BQ}_\ell))
\end{flalign*} as $\Gal(\ol{\BQ}_p/E')$-modules. In particular, $I_p$ acts unipotently on
\[
   H^r\bigl(
      \operatorname{Sh}_{\rK}(\bG,X)_{\overline E},
      \ol{\mathbb Q}_\ell
   \bigr) \text{\ and\ } H^r_c\bigl(
      \operatorname{Sh}_{\rK}(\bG,X)_{\overline E},
      \ol{\mathbb Q}_\ell
   \bigr)
\]
for every $r\geq0$.
\end{altenumerate}
\end{thm}

\begin{proof}
By \cite[Theorem~4.1]{GortzNearby}, $I_p$ acts trivially
on 
\[
   R^r\Psi^{{\rm ss}}(\ol{\mathbb Q}_\ell)
\] for each $r\geq 0$.
Hence, $R\Psi^{{\rm ss}}(\ol{\BQ}_\ell)$ is unipotent in the sense of
\cite[Definition~5.4]{GortzHaines}.  Since $f_{\ol{s}}$ is proper, $Rf_{\ol{s},*} R\Psi^{{\rm ss}}(\ol{\mathbb Q}_\ell)$ is unipotent by Lemma \ref{lem-f*f!}. 
Since $f_{\ol{\eta}}$ is an isomorphism and nearby cycles commute with proper push-forward (see e.g., \cite[Theorem~4.2]{GortzNearby}),
we have 
\begin{flalign}
    Rf_{\ol{s},*} R\Psi^{{\rm ss}}(\ol{\mathbb Q}_\ell)=Rf_{\ol{s},!} R\Psi^{{\rm ss}}(\ol{\mathbb Q}_\ell)\simeq R\Psi(\ol{\mathbb Q}_\ell).   \label{eqrfiso}
\end{flalign}
It follows that
$R\Psi(\ol{\mathbb Q}_\ell)$ is unipotent, and hence, $I_p$ acts unipotently on $R^r\Psi(\ol{\BQ}_\ell)$. Note that, for each $r\geq 0$, we have \begin{flalign*}
    H^r_c(\sS^{\rm ss}_{\ol{s}}, R\Psi^{\rm ss}(\ol{\BQ}_\ell)\simeq H^r_c(\sS_{\ol{s}}, Rf_{\ol{s},!}R\Psi^{\rm ss}(\ol{\BQ}_\ell)).
\end{flalign*} 
By \eqref{eqrfiso}, we obtain \begin{flalign*}
    H^r_c(\sS^{\rm ss}_{\ol{s}},R\Psi^{\rm ss}(\ol{\BQ}_\ell))\simeq H^r_c(\sS_{\ol{s}},R\Psi(\ol{\BQ}_\ell)).
\end{flalign*}
A similar result holds for the usual cohomology $H^r$.
By applying Lemma \ref{lem-f*f!} to the morphisms $\sS\ra S$ and $\sS^{\rm ss}\ra S$, the unipotence of $R\Psi(\ol{\BQ}_\ell)$ implies that $I_p$ acts unipotently on $H^r_c(\sS_{\ol{s}},R\Psi(\ol{\BQ}_\ell))$ and $H^r(\sS_{\ol{s}},R\Psi(\ol{\BQ}_\ell))$. This proves (1) and (2) of Theorem \ref{prop:inertia-action}.

Since $\sS$ has parahoric level at $p$ and the coefficient $\ol{\BQ}_\ell$ is constant, the isomorphisms in (3) follows from the (DL) case of \cite[Proposition 5.21]{Wu} (cf. \cite[Corollary 4.6]{lan2018nearbyII}). Here, we use the fact that the canonical integral model $\sS$ coincides with the construction in \cite{Wu} by the proof of \cite[Theorem D]{MW26}.
Therefore, by (2), $I_p$ acts unipotently on $H^r(\sS_{\ol{\eta}},\ol{\BQ}_\ell)$ and $H^r_c(\sS_{\ol{\eta}},\ol{\BQ}_\ell)$ for any $r\geq 0$. This proves (3).
\end{proof}

\Addresses
\end{document}